\documentclass[11pt]{amsart}
\usepackage[margin=1in,letterpaper,portrait]{geometry}
\usepackage{mathtools, 
latexsym,amssymb,verbatim,multirow,graphicx,epsfig,enumerate,
  paralist,
  hyperref}
\usepackage[style=alphabetic]{biblatex}
\bibliography{ref-fac-gen}
\usepackage{graphbox}
\usepackage{subcaption}
\usepackage{xcolor}
\usepackage[normalem]{ulem}
\usepackage{tikz}
\usepackage{bm}
\usepackage{mathabx}
\usepackage{float}
\usepackage{multirow}
\usepackage{thm-restate}

\theoremstyle{plain}
   \newtheorem{theorem}{Theorem}[section]
   \newtheorem{proposition}[theorem]{Proposition}
   
   \newtheorem{lemma}[theorem]{Lemma}
   \newtheorem{corollary}[theorem]{Corollary}
    
   \newtheorem{conjecture}[theorem]{Conjecture}
   
   \newtheorem*{theorem*}{Theorem}
\theoremstyle{definition}
   
   \newtheorem{definition}[theorem]{Definition}
   \newtheorem{example}[theorem]{Example}
   
   \newtheorem{question}[theorem]{Question}
   
   \newtheorem{remark}[theorem]{Remark}

\numberwithin{equation}{section}

\usepackage{xparse}

\DeclareDocumentCommand \ltr { o } {%
  \IfNoValueTF {#1} {%
    \ell_W^{\mathrm{full}} %
  }{%
    \ell_{#1}^{\mathrm{full}}%
  }%
}
\DeclareDocumentCommand \lR { o } {%
  \IfNoValueTF {#1} {%
    \ell_W^{\mathrm{red}} %
  }{%
    \ell_{#1}^{\mathrm{red}}%
  }%
}

\DeclareDocumentCommand \Fred { o } {%
  \IfNoValueTF {#1} {%
    F_W^{\mathrm{red}} %
  }{%
    F_{#1}^{\mathrm{red}}%
  }%
}

\newcommand{\op}[1]{\operatorname{#1}}

\newcommand{\Ftr}{F^\mathrm{full}}

\newcommand{\RGS}{\op{RGS}}

\newcommand\Symm{\mathfrak{S}}

\newcommand{\ttt}{\mathbf{t}}

\newcommand\rank{\operatorname{rank}}
\newcommand\codim{\operatorname{codim}}
\newcommand{\Mov}{\operatorname{Mov}}

\newcommand\wt{\col}
\newcommand{\col}{\operatorname{col}}

\newcommand{\RRR}{\mathcal{R}}

\newcommand{\defn}[1]{{\color{blue} \it {#1}}}

\newcommand{\BBB}{\mathcal{B}}

\newcommand\CC{{\mathbb{C}}}
\newcommand\ZZ{{\mathbb{Z}}}

\newcommand\GL{{\mathrm{GL}}}

\newcommand{\id}{\op{id}}

\begin{document}

\title[W-Hurwitz numbers: Part~II]{Hurwitz numbers for reflection groups II:\\ Parabolic quasi-Coxeter elements}
\author{Theo Douvropoulos, Joel Brewster Lewis, Alejandro H. Morales}

\maketitle

\begin{abstract}
We define parabolic quasi-Coxeter elements in well generated complex reflection groups. We characterize them in multiple natural ways, and we study two combinatorial objects associated with them: the collections $\op{Red}_W(g)$ of reduced reflection factorizations of $g$ and $\RGS(W,g)$ of relative generating sets of $g$. We compute the cardinalities of these sets for large families of parabolic quasi-Coxeter elements and, in particular, we relate the size $\#\op{Red}_W(g)$ with geometric invariants of Frobenius manifolds. This paper is second in a series of three; we will rely on many of its results in part III to prove uniform formulas that enumerate full reflection factorizations of parabolic quasi-Coxeter elements, generalizing the genus-$0$ Hurwitz numbers. \end{abstract}

\setcounter{tocdepth}{1}
\tableofcontents

\section{Introduction}

In the symmetric group $\Symm_n$, a factorization of a permutation $w$ as a product $w = t_1 \cdots t_r$ of transpositions is said to be \emph{transitive} if the factors $t_1, \ldots, t_r$ act transitively on the set of numbers $\{1, \ldots, n\}$. In the late 19th century, Hurwitz \cite{Hurwitz} used such combinatorial structures in his study of branched Riemann surfaces. In particular, he gave the following formula for the number \defn{$H_0(\lambda)$} of transitive factorizations of minimum length,\footnote{The index $0$ of the notation $H_0(\lambda)$ here refers to the \emph{genus} of the related Riemann surface and corresponds to the factorization being of minimum length. We do not deal with higher genus factorizations here, but the interested reader may consult \cite{DLM1}.} where $\lambda=(\lambda_1,\ldots,\lambda_k)$ is the  cycle type of $w$:
\begin{equation} \label{eq: formula H_0}
H_0(\lambda) =  n^{k-3} (n+k-2)! \prod_{i=1}^k \frac{\lambda_i^{\lambda_i}}{(\lambda_i-1)!}.
\end{equation}

Special cases of this formula have meaningful analogues in complex reflection groups, replacing the set of transpositions by the set of reflections \cite{chapuy_Stump,michel-deligne-lusztig-derivation, LM, chapuy_theo_lapl}.  This raises the question whether a full generalization may be possible. This paper is the second in a series of three (preceded by \cite{DLM1} and followed by \cite{DLM3}) that will culminate in such a generalization. 

One of the first difficulties in defining Hurwitz numbers for reflection groups $W$ is to find an analogue of the notion of transitivity of factorizations: there is no natural generalization of the set $\{1,\ldots,n\}$ in $\Symm_n$ for arbitrary $W$. In the first part of the series \cite{DLM1}, we addressed this by defining \defn{full} reflection factorizations as those factorization $g = t_1 \cdots t_k$ where the reflections $\{t_1,\ldots, t_k\}$ generate the \emph{full} group $W$. In the third part \cite{DLM3}, we will give uniform formulas for the number of minimum-length full reflection factorizations of a wide class of elements (the \emph{parabolic quasi-Coxeter elements}) in well generated complex reflection groups. It turns out that all elements in the symmetric group $\Symm_n$ are parabolic quasi-Coxeter, so that the formulas in \cite{DLM3} will be strict generalizations of \eqref{eq: formula H_0} to reflection groups. 

In the rest of the introduction, we present the main contributions of this paper. They are broadly divided in two parts, presented in \S\ref{sec:intro:struct} and \S\ref{sec:intro:enum}. First, we define our main combinatorial objects (the parabolic quasi-Coxeter elements and the relative generating sets) and discuss the structural results we prove for them, including multiple characterizations. Then, we give our main enumerative results: formulas that count the number of reduced factorizations and the number of relative generating sets for various families of parabolic quasi-Coxeter elements. Apart from their independent interest, the majority of these findings will be necessary in \cite{DLM3} to prove the uniform formula counting full reflection factorizations. In \S\ref{sec:intro:towards_uniform_formulas}, we briefly present (without proof) the main theorem of \cite{DLM3} for Weyl groups to motivate this connection. In \S\ref{sec:intro:summary}, we give a summary of the structure of the paper.

\subsection{Structural results}
\label{sec:intro:struct}

In a well generated complex reflection group $W$ of rank $n$, an element $g\in W$ is a \defn{quasi-Coxeter element} if there exists a length-$n$ reflection factorization $t_1\cdots t_n=g$ whose factors $t_i$ generate $W$. A quasi-Coxeter element of a parabolic subgroup of $W$ is a \defn{parabolic quasi-Coxeter element} of $W$. We introduce this definition in \S\ref{sec.parabolic qce}, extending the work of \cite{BGRW} on real groups.

These (parabolic) quasi-Coxeter elements generalize the (parabolic) Coxeter elements and share many of their properties. They are important in various mathematical areas ranging from singularity theory to combinatorics, particularly in the case that $W$ is a Weyl group; see \S\ref{sec.origins}. The reduced (i.e., minimum-length) reflection factorizations of parabolic quasi-Coxeter elements have rigid geometric properties. Particularly in the case of Weyl groups, this allows us to give a characterization for them (which includes earlier work of Baumeister--Wegener \cite{BW}) in terms of natural invariants of the group. (See \S\ref{Sec: preliminaries} for definitions of technical terms.) 

\begin{restatable*}[]{theorem}{pqcoxWeylchar}
\label{Prop: pqCox characterization for Weyl W}

Let $W \leq \GL(V)$ be a Weyl group of rank $n$, $g$ an element of $W$, and $W_g$ the parabolic closure of $g$. Then the following statements are equivalent. 
\begin{enumerate}[(i)]
    \item $g$ is a quasi-Coxeter element (respectively, parabolic quasi-Coxeter element).
    \item There exists a reduced reflection factorization $g=t_1\cdots t_k$ for which the associated roots $\rho_{t_i}$ and coroots $\widecheck{\rho}_{t_i}$ form $\ZZ$-bases of the root and coroot lattices of $W$ (resp., of $W_g$).
    \item $g$ satisfies $|\det(g - I_V)|=I(W)$, where $I_V$ is the identity on $V$ and $I(W)$ the connection index of $W$ (resp., $|\op{pdet}(g - I_V)|=I(W_g)$ where $\op{pdet}$ denotes the pseudo-determinant).
    \item $g$ does not belong to any proper reflection subgroup of $W$ (resp., of $W_g$).
\end{enumerate}

\end{restatable*}

Moving back to the context of parabolic quasi-Coxeter elements in an arbitrary well generated group, we discuss in \S\ref{ssec.pcqe.hurw.action} their interaction with the Hurwitz action, and we prove an analogue of the unique cycle decomposition for them in Proposition~\ref{Prop: gob-cycle-decomp}. 

The second combinatorial object we introduce in this paper is the collection \defn{$\RGS(w)$} of \defn{good generating sets} for a complex reflection group $W$; these are the sets of $\rank(W)$-many reflections that generate the full group $W$. Only well generated groups possess good generating sets. We also consider a generalization of this notion: we say that a set of $\big(\rank(W)-k\big)$-many reflections that generate the full group $W$ \emph{when combined with a reduced reflection factorization of $g$} is a \defn{relative generating set} for the element $g \in W$, and we denote by \defn{$\RGS(W, g)$} the collection of relative generating sets for $g$.

Remarkably, the existence of relative generating sets is tied to the parabolic quasi-Coxeter property.  Our main structural theorem gives three uniform characterizations of parabolic quasi-Coxeter elements in well generated groups, in terms of relative generating sets, in terms of their \defn{full reflection length} (the minimum length of a full factorization), and in terms of the smaller collection of quasi-Coxeter elements.

\begin{restatable*}[characterization of parabolic quasi-Coxeter elements]{theorem}{pqCoxChar} 
\label{Prop: characterization of pqCox}
For a well generated group $W$ of rank $n$ and an element $g\in W$, the following are equivalent.
\begin{enumerate}[(i)]
    \item The element $g$ is a parabolic quasi-Coxeter element of $W$.
    \item The collection $\RGS(W,g)$ of relative generating sets with respect to $g$ is nonempty.
    \item There exists a quasi-Coxeter element $w\in W$ such that $g\leq_{\RRR} w$.
    \item The full reflection length of $g$ satisfies $\ltr(g)=2n-\lR(g)$.
\end{enumerate}

\end{restatable*}

\noindent The relationship between relative generating sets and parabolic quasi-Coxeter elements is reflected in the uniform formulas of \cite{DLM3} (see \S\ref{sec:intro:towards_uniform_formulas} below).

In the infinite families $W=G(m,1,n)$ and $W=G(m,m,n)$ of ``combinatorial'' well generated groups, the relative generating sets $\RGS(W,g)$ correspond to tree-like structures where the vertices are the (generalized) cycles of the corresponding element $g\in W$. This is developed in detail in \S\ref{sec:characterize RGS}.

\subsection{Enumerative results}
\label{sec:intro:enum}

The second family of results in this paper involves enumerative questions around parabolic quasi-Coxeter elements and their reduced factorizations and relative generating sets. It is a striking phenomenon, for which we do not have a complete explanation, that for all parabolic quasi-Coxeter elements $g$ in a well generated group $W$, the number $\Fred(g)$ of reduced reflection factorizations of $g$ is always a product of small primes. In Section~\ref{sec: Fred numbers of qCox}, we develop a (partially conjectural) framework to explain this phenomenon. We also give explicit formulas for the numbers $\Fred(g)$ in many cases; in particular, in \S\S\ref{ssec: Fred of qcox_infinite family}--\ref{ssec:reduction from pqc to qc} we compute the numbers $\Fred(g)$ for all (parabolic) quasi-Coxeter elements in the infinite families of the well generated groups $G(m,1,n)$ and $G(m,m,n)$.

When $c\in W$ is a Coxeter element, the the number $\Fred(c)$ is well understood and given by an explicit formula: when $W$ is irreducible of rank $n$ and with Coxeter number $h$, the Arnold--Bessis--Chapoton formula states that
\begin{equation}\label{eq:intro:Arnold_Bessis_Chapoton}
    \Fred(c)=\dfrac{h^n \cdot n!}{\#W}.
\end{equation}
There are many proofs of this formula; in \S\ref{ssec: Fred of qcox_Coxeter case} we review a geometric approach that relates the count \eqref{eq:intro:Arnold_Bessis_Chapoton} with the degree of the Lyashko--Looijenga ($LL$) morphism associated to the braid monodromy of the discriminant hypersurface $\mathcal{H}$ of $W$. In \S\ref{Sec: Frobenius mflds} we explain how most of this theory has analogues for arbitrary quasi-Coxeter elements, at least for real reflection groups, and we conjecture that the number $\Fred(g)$ always agrees with the degree of an associated quasi-homogeneous $LL$ map (Conjecture~\ref{conj alg frob mfld}). This conjecture would explain the phenomenon that these counts are products of small primes.

The context of Conjecture~\ref{conj alg frob mfld} is the geometry of algebraic Frobenius manifolds, whose differential and analytic structure can be encoded via the reduced factorizations of quasi-Coxeter elements. Even though we do not prove the conjecture, we verify it for many cases (see Remark~\ref{Rem. numbers dg are hard} and \S\ref{sec:8:wdvv}).  Moreover, it motivates an almost explicit, uniform generalization of \eqref{eq:intro:Arnold_Bessis_Chapoton} for the class of regular quasi-Coxeter elements in Weyl groups. The following proposition is proved case-by-case; the non-explicit part is the interpretation of the numbers $\delta_g$ (see Remark~\ref{Rem. numbers dg are hard}). The \emph{exponents} of an element $g\in W\leq \GL(V)$ are certain positive integers $e_i(g)$ that encode the eigenvalues of $g$ in $V$ (see \S\ref{sssec.reg_qcox_freds}).

\begin{restatable*}[]{proposition}{FrobFredNums}
\label{Obs. Fred regular qCox}

For a regular quasi-Coxeter element $g$ in an irreducible Weyl group $W$, we have 
\begin{equation}
\label{Eqn: Red_w(g) Frob.} 
\Fred(g) = \frac{|g|^n \cdot n!}{\prod_{j=1}^n(e_j(g)+1)}\cdot \delta_g,
\end{equation} where the numbers $e_j(g)$ are the exponents of $g$ and $\delta_g$ is a small integer. When $g$ is a Coxeter element, $\delta_g$ is equal to $1$, while the rest of the cases are given in Table~\ref{Table: the numbers d_g}.

\end{restatable*}

The other major enumerative question we ask in this work is for the number of relative generating sets associated to a given parabolic quasi-Coxeter element $g\in W$. We show in Proposition~\ref{Prop: RGS(W,g)=RGS(W,W_g)} that the set $\RGS(W,g)$ only depends on the parabolic closure $W_g$ of $g$ (in which $g$ is a quasi-Coxeter element).  In the case of the infinite families $W=G(m,1,n)$ and $W=G(m,m,n)$, the enumeration reduces to counting certain weighted trees or unicycles whose vertices are the (generalized) cycles of $g$. In \S\ref{sec: counting rrgs} we prove the following result by making repeated use of the weighted Cayley theorem (Theorem~\ref{thm: weighted cayley}). 

\begin{restatable*}[]{theorem}{thmcountrgs}

\label{thm:count rgs}
Suppose that $W$ is either $G(m, 1, n)$ or $G(m, m, n)$ and that $g$ is a parabolic quasi-Coxeter element for $W$.
\begin{enumerate}[(i)]
\item If either $W = G(m, 1, n)$ and $g$ is a quasi-Coxeter element for the subgroup $G(m, 1, \lambda_0) \times \Symm_{\lambda_1} \times \cdots \times \Symm_{\lambda_k}$, or $W = G(m, m, n)$ and $g$ is a quasi-Coxeter element for the subgroup $G(m, m, \lambda_0) \times \Symm_{\lambda_1} \times \cdots \times \Symm_{\lambda_k}$, then 
\begin{equation} \label{eq: RGS first family}
\# \RGS(W, g) = m^k \cdot n^{k - 1} \cdot \prod_{i = 0}^k \lambda_i.
\end{equation}
\item If $W = G(m, 1, n)$ and $g$ is a quasi-Coxeter element for the subgroup $\Symm_{\lambda_1} \times \cdots \times \Symm_{\lambda_k}$, then 
\begin{equation}
\# \RGS(W, g) = \varphi(m) \cdot m^{k - 1} \cdot n^{k - 1} \cdot \prod_{i = 1}^k \lambda_i,
\end{equation}
where $\varphi$ denotes Euler's totient function. 
\item  If $W = G(m, m, n)$ and $g$ is a quasi-Coxeter element for the subgroup $\Symm_{\lambda_1} \times \cdots \times \Symm_{\lambda_k}$, then 
\begin{equation}
\#\RGS(W, g) =  \frac{\varphi(m) \cdot m^{k - 1}}{2}  \cdot
\Big(n^k - n^{k-1} -\sum_{j=2}^k  (j-2)! \cdot n^{k-j} e_j(\lambda)\Big) \cdot  \prod_{i = 1}^k \lambda_i,
\end{equation}
where $e_i(\lambda)$ denotes the $i$th elementary symmetric function in the variables $\lambda=(\lambda_1,\ldots,\lambda_k)$.
\end{enumerate}

\end{restatable*}

The case of the symmetric group $W=\Symm_n$ is discussed in further detail in \S\ref{sec: comparison DPS cacti}, where we relate the enumeration of our sets $\RGS(W,g)$ with that of \emph{Cayley cacti}, which were used by Duchi--Poulalhon--Shaeffer \cite{DPS_II} to give a bijective proof of the Hurwitz formula \eqref{eq: formula H_0}. 

\subsection{Towards counting full factorizations of parabolic quasi-Coxeter elements}\label{sec:intro:towards_uniform_formulas}

Apart from their own individual interest, the results of this paper will be crucial components of the proofs for the uniform formulas in the third part \cite{DLM3} of this series. To motivate them further, we give here without proof (but see \cite{DLM-extended-abstract} for a sketch) the main theorem of \cite{DLM3} for Weyl groups.

\begin{theorem*}
For any Weyl group $W$ and any parabolic quasi-Coxeter element $g\in W$
with generalized cycle decomposition $g=g_1\cdot g_2\cdots g_m$, the number $\Ftr_W(g)$
of minimum-length full reflection factorizations of $g$ is given by the formula  
\begin{equation}
\Ftr_W(g)=\ltr(g)!\cdot
\#\RGS(W,g)\cdot\dfrac{I(W_g)}{I(W)}\cdot\prod_{i=1}^m\dfrac{\Fred(g_i)}{\ell_R(g_i)!}
\, ,
\label{EQ: F_W(g)}
\end{equation}
where $\ltr(g)$ is the full reflection length of $g$, $W_g$ is the parabolic closure of $g$, and $I(W)$ is the connection index of $W$.
\end{theorem*}

The enumerative results of this paper, presented in \S\ref{sec:intro:enum}, will be used in \cite{DLM3} to calculate the right-hand side of \eqref{EQ: F_W(g)} for the infinite families $A_n$, $B_n$, $D_n$ (and their generalization in the complex types), while the left-hand side has been calculated in \cite{DLM1}. The structural results presented in \S\ref{sec:intro:struct} will be used in an inductive argument that helps prove \eqref{EQ: F_W(g)} for the  exceptional Weyl groups $E_7$ and $E_8$, where the calculation of the numbers $\Ftr_W(g)$ with the techniques of \cite{DLM1} is not computationally feasible.
Moreover, we hope that the structural results presented in \S\ref{sec:intro:struct} support our case -- reinforced by \eqref{EQ: F_W(g)} -- that the collection $\RGS(W,g)$ of relative generating sets is a natural combinatorial object in the study of reflection factorizations.

\subsection{Summary}
\label{sec:intro:summary}

The structure of the present paper is as follows.  In Section~\ref{Sec: preliminaries}, we give background, introduce the supporting objects of our study, and include some preliminary properties. In Section~\ref{sec.parabolic qce}, we introduce the main objects we study in this paper: the family of elements  called parabolic quasi-Coxeter elements. In the next Section~\ref{Sec: relative gen sets}, we characterize such elements in terms of certain sets of reflections that we call relative reflection generating sets, which will be fundamental in the uniform formulas for $W$-Hurwitz numbers in the sequel \cite{DLM3}. In Section~\ref{Sec: full factorizations and p-q-Cox}, we give two other characterizations of parabolic quasi-Coxeter elements, in terms of absolute order and in terms of full reflection length. After this, in Section~\ref{sec: Fred numbers of qCox}, we study the number of reduced reflection factorizations of parabolic quasi-Coxeter elements in well generated complex reflection groups. In Section~\ref{sec: counting rrgs}, we give nice formulas for the numbers of (relative) generating sets in the combinatorial subfamily of well generated complex reflection groups. We wrap up this paper in Section~\ref{sec: gen_quest} with some final remarks and open questions. 

\section{Reflection groups, good generating sets, and parabolic subgroups}
\label{Sec: preliminaries}

In this section, we introduce the main supporting objects of our study and develop some preliminary properties.  Standard background material for real reflections groups may be found in \cite{Kane, Humphreys}, and for complex reflection groups in \cite{Kane, LehrerTaylor,broue-book}.

\subsection{Real reflection groups}
\label{Sec: real refl grps}

Given a finite-dimensional real inner product space $V$, a \defn{reflection} $t$ is an orthogonal map whose \defn{fixed space} $\defn{V^{t}} := \{v \in V \colon t(v) = v\}$ is a hyperplane, that is, $\codim(V^{t})=1$. Equivalently, an orthogonal map is a reflection if it is diagonalizable, with one eigenvalue equal to $-1$ and all others equal to $1$.  A finite subgroup $W\leq \GL(V)$ is called a \defn{real reflection group} if it is generated by reflections. 

The real reflection groups are precisely the finite \defn{Coxeter groups}, those generated by a set $S = \{s_1, \ldots, s_n\}$ of reflections subject to relations $s_i^2 = 1$ and $(s_i s_j)^{m_{ij}} = 1$ for some integers $m_{ij}>1$. 
The generating set $S$ is called a set of \defn{simple reflections} for $W$, and the pair $(W, S)$ of a Coxeter group together with such a set of simple reflections is called a \defn{simple system}. The size $n$ of the set $S$ is called the \defn{rank} of $W$.

We say that a real reflection group $W$ is \defn{irreducible} if there is no nontrivial subspace of $V$ stabilized by its action. It is easy to see that if $W_1$ acts on $V_1$, with reflections $\RRR_1$, and $W_2$ acts on $V_2$, with reflections $\RRR_2$, then $W = W_1 \times W_2$ acts on $V = V_1 \oplus V_2$, with reflections $\RRR = (\RRR_1 \times \{\id_2\}) \cup (\{\id_1\} \times \RRR_2)$.  In this case, $W$ fails to be irreducible (it stabilizes the subspace $V_1$ of $V$). Coxeter \cite{coxeter_annals} classified real reflection groups, as follows: every real reflection group is a product of irreducibles, and the irreducibles belong to four infinite families $A_n$ (the symmetric groups), $B_n$ (the hyperoctahedral groups of signed permutations), $D_n$ (index-$2$ subgroups of the hyperoctahedral groups), and $I_2(m)$ (the dihedral groups), and six exceptional types $H_3$, $H_4$, $F_4$, $E_6$, $E_7$, and $E_8$, where in all cases the indices correspond to ranks.

A real reflection group $W$ is completely determined by its set \defn{$\RRR$} of reflections.  Alternatively, it is determined by the associated \defn{reflection arrangement} \defn{$\mathcal{A}_W$}, namely, the arrangement of fixed hyperplanes of the reflections in $\RRR$. The reflection arrangement of $W$ is often encoded as part of a structure known as a \emph{root system}, which we define now.  For any reflection $t\in \RRR$, we may choose two vectors \defn{$\rho_t$} and \defn{$\widecheck{\rho}_t$}, both orthogonal to the fixed hyperplane $V^t$, that satisfy $\langle \rho_t,\widecheck{\rho}_t\rangle=2$.  They are known as the \defn{root} and \defn{coroot} associated to the reflection $t$, which they determine via the relation
\begin{equation}
t(v)=v-\langle v,\widecheck{\rho}_t\rangle\cdot \rho_t\label{eq: defn refln real}
\end{equation}
for all $v$ in $V$.  The collection $\defn{\mathrm{\Phi}}:=\{\pm \rho_t \colon t\in\RRR\}$ of roots and their negatives is the \defn{root system} of $W$; we choose the lengths of the roots $\rho_t$ so that $\Phi$ is $W$-invariant.  The \defn{simple (co)roots} $\defn{\alpha_i} := \rho_{s_i}$ and \defn{$\widecheck{\alpha}_i$}$\, := \widecheck{\rho}_{s_i}$ are the (co)roots associated to the simple reflections $s_i$.

\subsubsection*{Weyl groups} Although root systems may be defined for all real reflection groups, they behave particularly well for a subclass known as \defn{crystallographic} or \defn{Weyl} groups. Those are characterized by the existence of an essential, $W$-invariant \defn{root lattice} \defn{$\mathcal{Q}$}, the integer span of the root system $\Phi$.  When this exists, the coroots also span a lattice, the \defn{coroot lattice} \defn{$\widecheck{\mathcal{Q}}$}.  The simple roots and coroots form a lattice basis for $\mathcal{Q}$ and $\widecheck{\mathcal{Q}}$, respectively.

Weyl groups again decompose as products of irreducibles; the irreducible Weyl groups are the infinite families $A_n$, $B_n$, and $D_n$, and the five exceptional types $E_6$, $E_7$, $E_8$, $F_4$, and $G_2=I_2(6)$. In terms of the Coxeter presentation, these are exactly the finite Coxeter groups for which $m_{ij}\in\{2,3,4,6\}$ for all $i$ and $j$.

For any Weyl group, there is a natural inclusion of the root lattice $\mathcal{Q}$ inside the dual lattice $\mathcal{P}$ of the coroot lattice $\widecheck{\mathcal{Q}}$.  (The lattice $\mathcal{P}$, which will not play a major role in this paper, is called the \emph{weight lattice}.)  The \defn{connection index} $\defn{I(W)}:=[\mathcal{P}:\mathcal{Q}]$ is an important invariant of $W$. It is equal to the determinant of the \defn{Cartan matrix} $\defn{K}:=\big( \langle\alpha_i,\widecheck{\alpha}_j\rangle\big)_{i,j=1}^n$ of the Weyl group (where $\alpha_i$ and $\widecheck{\alpha}_j$ are the simple roots and coroots; see \cite[\S9-4]{Kane}). We give in Table~\ref{Table: I(W)} the values $I(W)$ for the irreducible Weyl types; for a reducible group $W=W_1\times\cdots\times W_r$, we have that $I(W)=\prod_{i=1}^r I(W_i)$.

\begin{table}[H]
\[
\begin{array}{|c|| c| c| c| c| c| c| c| c |c|} 
 \hline
 W & \Symm_{n+1}=A_n & B_n & D_n & E_6 & E_7 & E_8 & F_4 & I_2(6)=G_2 \\ 
 \hline
 I(W) & n+1 & 2 & 4 & 3 & 2 & 1 & 1 & 1 \\ 
 \hline
\end{array}
\]
\caption{The connection indices for irreducible Weyl groups.}
\label{Table: I(W)}
\end{table}

\subsubsection*{Parabolic Coxeter elements} In a real reflection group $W$, the product of the simple generators $s_i\in S$ of $W$, in any order, is called a Coxeter element.  They were introduced in \cite{coxeter_duke}, where it was shown that all such elements are conjugate to one another. We call \emph{any} element $c\in W$ that belongs to this conjugacy class a \defn{Coxeter element}; their common order $\defn{h} :=|c|$ is the \defn{Coxeter number} of $W$.  There are many choices of sets $S'$ of reflections that turn $(W,S')$ into a simple system. When $W$ is not crystallographic, different choices of $S$ may lead to different conjugacy classes of Coxeter elements.\footnote{For example, in the dihedral group $I_2(5)$, any pair of reflections determines a simple system for the group; but the product of a pair whose hyperplanes make an angle of $36^{\circ}$ (a rotation by $72^\circ$) is not conjugate to the product of a pair whose hyperplanes make an angle of $72^{\circ}$ (a rotation by $144^{\circ}$).} Following \cite{RRS}, the elements of $W$ that are Coxeter elements for \emph{some} choice of simple system are called \defn{generalized Coxeter elements}; they have the same order $h$ as the Coxeter elements.

For any subset $I\subseteq S$ of simple reflections, the subgroup $\defn{W_I}:=\langle I\rangle$ of $W$ generated by $I$ is called a \defn{standard parabolic subgroup}, and any subgroup conjugate to one of the $W_I$ is called a \defn{parabolic subgroup}. Each parabolic subgroup inherits a natural simple system from $W$.  The Coxeter elements in any parabolic subgroup are the \defn{parabolic Coxeter elements} of $W$.

In the following sections we introduce analogues of all the concepts we discussed here in the setting of complex reflection groups.

\subsection{Complex reflection groups}
\label{Sec: CX refl grps}

Let $V$ be a finite-dimensional complex vector space with a fixed Hermitian inner product.  A \defn{(unitary) reflection} $ t$ on $V$ is a unitary map whose fixed space $V^{t}$ is a hyperplane, that is, $\codim(V^{ t})=1$.  A finite subgroup $W\leq \GL(V)$ is called a \defn{complex reflection group} if it is generated by reflections.  As in the real case, we denote the subset of reflections by $\RRR$ and the associated reflection arrangement (that consists of all the fixed hyperplanes of reflections $t \in \RRR$) by $\mathcal{A}_W$, and we say that $W$ is irreducible if there is no nontrivial subspace of $V$ stabilized by its action.  Shephard and Todd \cite{ShephardTodd} classified the complex reflection groups, as follows: every complex reflection group is a product of irreducibles, and every irreducible either belongs to an infinite three-parameter family $G(m,p,n)$, described below, or is one of 34 exceptional cases, numbered $G_4$ to $G_{37}$.  

By extending scalars, every real reflection group may be viewed as a complex reflection group.  This operation preserves irreducibility, with all four infinite families of real reflection groups subsumed into to the infinite family of complex reflection groups.

\subsubsection*{The combinatorial family}

Let $m$, $p$, and $n$ be positive integers such that $p$ divides $m$.  Such a triple indexes a member $G(m, p, n)$ of the infinite family of complex reflection groups, which may be concretely described as 
\[
\defn{G(m, p, n)} := \left\{ \begin{array}{c} n \times n \textrm{ monomial matrices whose nonzero entries are} \\
\textrm{$m$th roots of unity with product an $\frac{m}{p}$th root of unity} \end{array} \right\}.
\]

It is natural to represent such groups combinatorially.
We may encode each element $w$ of $G(m, 1, n)$ by a pair $[u; a]$ with $u \in \Symm_n$ and $a = (a_1, \ldots, a_n) \in (\ZZ/m\ZZ)^n$, as follows: for $k = 1, \ldots, n$, the nonzero entry in column $k$ of $w$ is in row $u(k)$, and the value of the entry is $\exp(2\pi i a_k/m)$.  With this encoding, it's easy to check that
\[
[u; a] \cdot [v; b] = [uv; v(a) + b], \quad \textrm{ where } \quad v(a) := \left( a_{v(1)}, \ldots, a_{v(n)} \right).
\]
This shows that the group $G(m, 1, n)$ is isomorphic to the \defn{wreath product} $\ZZ/m\ZZ \wr \Symm_n$ of a cyclic group with the symmetric group $\Symm_n$.

The reflections in $G(m, p, n)$ come in two families: for $1 \leq i < j \leq n$ and $k \in \ZZ/m\ZZ$ there is the \defn{transposition-like reflection}
\[
\defn{[(i j); k]} := [(i j); k e_i - k e_j] = [(i j); (0, \ldots, 0, k, 0, \ldots, 0, -k, 0, \ldots, 0)],
\]
which fixes the hyperplane $x_j = \exp(2\pi i k/m) x_i$, and if $p < m$ then for $i = 1, \ldots, n$ and $k = 1, \ldots, \frac{m}{p} - 1$ there is the \defn{diagonal reflection}
\[
[\id; kp e_i] = [\id; (0, \ldots, 0, kp, 0, \ldots, 0)],
\]
which fixes the hyperplane $x_i = 0$.  Observe that with this notation, $[(i j); k] = [(j i); -k]$.

Given an element $w = [u; a] \in G(m, p, n)$ and a subset $S \subseteq \{1, \ldots, n\}$, we say that $\sum_{k \in S} a_k$ is the \defn{color} of $S$; this notion will come up particularly when the elements of $S$ form a cycle in $u$.  When $S = \{1, \ldots, n\}$, we call $a_1 + \ldots + a_n$ the color of the element $w$, and we denote it $\defn{\wt(w)}$.  In this terminology, $G(m, p, n)$ is the subgroup of $G(m, 1, n)$ containing exactly those elements whose color is a multiple of $p$, and $\wt$ is a surjective group homomorphism from $G(m, p, n)$ to $p\ZZ/m\ZZ \cong G(m, p, 1) \cong \ZZ / (m/p)\ZZ$.  Meanwhile the map $[u; a] \mapsto u$ that sends an element of $G(m, p, n)$ to its \defn{underlying permutation} is a surjective group homomorphism onto $\Symm_n$.

\subsection{Well generated groups and good generating sets}
\label{Sec: well-gend grps}

If $W$ is a complex reflection group acting on $V$, the \defn{rank} of $W$ is the codimension $\codim(V^W)$ of the fixed space $\defn{V^W} := \bigcap_{w \in W} V^w$ of $W$ (or, equivalently, the dimension of the orthogonal complement $(V^W)^\perp$).  Every irreducible complex reflection group of rank $n$ can be generated by either $n$ or $n+1$ reflections. The groups in the first category are called \defn{well generated}; they are better understood and share many properties with the subclass of real reflection groups. 

\begin{definition}[good generating sets]\label{Defn: good gen'g sets}
For a complex reflection group $W$ and a subset $S:=\{t_1,\ldots,t_n\}$ of reflections in $W$, we say that $S$ is a \defn{good generating set} for $W$ if $\langle t_i\rangle=W$ and the cardinality $n$ of $S$ equals the rank of $W$. In particular, it is precisely the well generated groups that have good generating sets.
\end{definition}

When $W$ is a \emph{real} reflection group, any set $S$ of simple generators is a good generating set for $W$, but there are many more good generating sets than simple systems.  For Weyl groups, good generating sets have a particularly nice characterization in terms of the root and coroot lattices.

\begin{proposition}[{\cite[Cor.~1.2]{BW}}]
\label{Prop: BW charn of gen sets}
For a Weyl group $W$ of rank $n$, a set $\{t_i\}_{i=1}^n$ of reflections generates $W$ if and only if the roots $\rho_{t_i}$ and coroots $\widecheck{\rho}_{t_i}$ form $\ZZ$-bases of the root and coroot lattices $\mathcal{Q}$ and $\widecheck{\mathcal{Q}}$ of $W$, respectively.
\end{proposition}

Next, we rephrase the criterion of Proposition~\ref{Prop: BW charn of gen sets} in terms of the connection index $I(W)$ of the Weyl group; this gives a computationally fast way to check whether a set of reflections is a good generating set or not.  To begin, suppose that we have a set $\{t_i\}_{i=1}^n$ of $n$ reflections in a Weyl group $W$, with associated roots $\rho_{t_i}$ and coroots $\widecheck{\rho}_{t_j}$.  Let $\{\alpha_i\}_{i=1}^n$ be a set of simple roots for $W$, with associated coroots $\widecheck{\alpha}_i$. If we denote by $U$ the matrix that expresses the $\rho_{t_i}$ as $\ZZ$-linear combinations of the $\alpha_i$ and similarly we write $\widecheck{U}$ for the matrix that expresses the $\widecheck{\rho}_{t_j}$ in terms of the $\widecheck{\alpha}_j$, then it follows easily that
\begin{equation}
    \label{eq:Cartan change of basis}
\det\big(\langle\rho_{t_i},\widecheck{\rho}_{t_j}\rangle\big)= 
\det(U) \cdot \det\big(\langle \alpha_i, \widecheck{\rho}_{t_j}\rangle\big) = 
\det(U) \cdot \det\big(\langle \alpha_i, \widecheck{\alpha}_j\rangle\big) \cdot \det\big(\widecheck{U}^\top\big) = \det(U) \cdot I(W) \cdot \det\big(\widecheck{U}^\top\big)
\end{equation}
(where the last equality uses the fact that $\big(\langle \alpha_i, \widecheck{\alpha}_j\rangle\big)$ is the Cartan matrix of $W$).  Moreover, since each coroot $\widecheck{\rho}_{t_i}$ is a positive multiple of the corresponding root $\rho_{t_i}$, and likewise for the $\alpha_i$, the two determinants $\det(U)$ and $\det(\widecheck{U})$ have the same sign.

\begin{corollary}
\label{Corol: gen set gdet=I(W)}
In a Weyl group $W$ of rank $n$, a set of reflections $\{t_i\}_{i=1}^n$ is a good generating set for $W$ if and only if
\[
\det \big( \langle \rho_{t_i},\widecheck{\rho}_{t_j}\rangle\big)_{i,j=1}^n=I(W),
\]
where $\rho_{t}$ and $\widecheck{\rho}_{t}$ are the root and coroot associated to $t$ and $I(W)$ is the connection index of $W$.
\end{corollary}
\begin{proof}
Since the roots $\{\rho_{t_i}\}$ (respectively, coroots $\{\widecheck{\rho}_{t_i}\}$) belong to the the lattice $\mathcal{Q}$ (resp., $\widecheck{Q}$) with $\ZZ$-basis $\{\alpha_i\}$ (resp., $\{\widecheck{\alpha}_i\}$), the matrix $U$ (resp., $\widecheck{U}$) is an integer matrix, with integer determinant.  Moreover, $\{\rho_{t_i}\}$ (resp., $\{\widecheck{\rho}_{t_i}\}$) forms a $\ZZ$-basis for $\mathcal{Q}$ (resp., $\widecheck{Q}$) if and only if $U$ (resp., $\widecheck{U}$) is invertible, i.e., if and only if it has determinant $\pm 1$.  By Proposition~\ref{Prop: BW charn of gen sets}, it follows that $\{t_i\}_{i=1}^n$ is a good generating set if and only if $\det(U)$ and $\det(\widecheck{U})$ are both $\pm 1$.  Then the result follows by \eqref{eq:Cartan change of basis}, after noting that $\det(U)$ and $\det(\widecheck{U})$ have the same sign.
\end{proof}

By the same methods, we can prove the following result (which we believe must be known, but for which we were not able to find a direct reference in the literature).

\begin{corollary}\label{Cor: W'<W means I(W')>I(W)}
If $W$ is a Weyl group and $W'\lneq W$ is a proper reflection subgroup of the same rank, then $I(W')>I(W)$.
\end{corollary}
\begin{proof}
Let $n$ be the common rank of $W$ and $W'$ and $\{t_i\}_{i=1}^n$ be a (good) generating set of reflections for $W'$, and let $\rho_{t_i}$ and $\widecheck{\rho}_{t_j}$ denote the associated roots and coroots. By Corollary~\ref{Corol: gen set gdet=I(W)} and \eqref{eq:Cartan change of basis}, we have that
\[
I(W')=\det\big(\langle\rho_{t_i},\widecheck{\rho}_{t_j}\rangle\big) = |\det(U)| \cdot I(W) \cdot |\det\big(\widecheck{U}^\top\big)|.
\]
Since the $t_i$ do not generate all of $W$, we have by Proposition~\ref{Prop: BW charn of gen sets} that at least one of $U$, $\widecheck{U}$ will have (integer) determinant different from $\pm 1$. This completes the proof.
\end{proof}

Unlike the case of simple systems, the product of the elements of a good generating set need not be a (generalized) Coxeter element.  For example, in $W := D_4 = G(2, 2, 4)$, the reflections $s_1 = [(12); 0]$, $s_2 = [(23); 0]$, $s_3 = [(34); 0]$ and $t= [(13); 1]$ form a good generating set for $W$ (since if $\widetilde{s_1}=[(12);1]$ is the fourth simple generator, we have $s_2ts_2 = \widetilde{s_1}$).  However, they have product $s_1s_2s_3 t = [(14)(23); (1, 0, 1, 0)]$, which is not a (generalized) Coxeter element (all generalized Coxeter elements in $D_4$ are of order $6$ and have as underlying permutation a $3$-cycle). We meet the products of arbitrary good generating sets again in \S\ref{sec.parabolic qce}, where they are called \emph{quasi-Coxeter} elements; they are the main objects of study in this work.

\subsubsection*{The combinatorial family}
In the infinite family, the groups $G(m, 1, n)$ and $G(m, m, n)$ are well generated, but the intermediate groups $G(m, p, n)$ for $1 < p < m$ are not \cite[\S 2.7]{LehrerTaylor}.  We now describe the good generating sets in $G(m, 1, n)$ and $G(m, m, n)$.

It is natural to represent collections of reflections in $G(m, p, n)$ as graphs on the vertex set $\{1, \ldots, n\}$: a transposition-like reflection whose underlying permutation is $(i j)$ can be represented by an edge joining $i$ to $j$, while a diagonal reflection whose matrix entry of nonzero color occurs at position $(i, i)$ can be represented by a loop at vertex $i$.  
We say that a collection of reflections in $G(m, p, n)$ is \defn{connected} if the associated graph is connected.  It is a basic fact of graph theory that any connected graph contains a spanning tree; consequently, any connected set of reflections in $G(m, p, n)$ contains a connected subset of $n - 1$ transposition-like reflections.  Such a subset always generates a subgroup of $G(m, p, n)$ isomorphic to the symmetric group $\Symm_n$ \cite[Lem.~2.7]{Shi2005}.  In particular, when $m = p = 1$ (so $G(m, p, n) = G(1, 1, n) = \Symm_n$), the good generating sets correspond to the trees on $\{1, \ldots, n\}$.  In order to characterize good generating sets when $m > 1$, we introduce some additional terminology.

Every connected graph on $n$ vertices with $n$ edges contains a unique cycle.  In the case that the cycle is a loop, we say that the graph is a \defn{rooted tree}.  Otherwise, the cycle contains at least two vertices, and in this case we call a graph a \defn{unicycle}.  Given a set $S$ of $n$ reflections whose associated graph is a unicycle, define a quantity $\delta(S)$, as follows: if the unique (graph) cycle has vertices $i_0, i_1, \ldots, i_{k - 1}, i_k = i_0$ in order,\footnote{
There are two possible cyclic orders; for each (graph) cycle, we fix an arbitrary choice of the order in which to read its vertices.} so that the associated reflections are $[(i_j \; i_{j + 1}); a_j]$ for some $a_0, \ldots, a_{k - 1} \in \ZZ/m\ZZ$, we set
\[
\delta(S) = a_0 + \ldots + a_{k - 1}.
\]
Shi showed that good generating sets in the combinatorial family can be characterized in terms of the objects just defined.

\begin{lemma}[{\cite[Lem.~2.1, Thm.~2.8, and Thm.~2.19]{Shi2005}}] 
\label{lem: good generating set G(m,1,n) and G(m,m,n)}
\

\begin{enumerate}[(i)]
    \item A set of $n$ reflections in $G(m, 1, n)$ generates the full group if and only if its graph is rooted tree and the color of the diagonal factor is a generator of $\ZZ/m\ZZ$.
    \item A set $S$ of $n$ reflections in $G(m, m, n)$ generates the full group if and only if its graph is a unicycle and $\delta(S)$ is a generator of $\ZZ/m\ZZ$.
\end{enumerate}
\end{lemma}

\begin{remark}
\label{rem: Shi's delta}
It is not difficult to see (for example, by following the argument in \cite[Prop.~3.30]{LW}, conjugating by a diagonal element of $G(m, 1, n)$ to send all but one of the elements of $S$ to true transpositions) that if $S$ is a unicycle then multiplying the factors of $S$ together in any order produces an element of $G(m, m, n)$ with exactly two (permutation) cycles, one of color $\delta(S)$ and the other of color $-\delta(S)$. This is compatible with the arbitrary choice in the definition of $\delta$: choosing the other cyclic order would replace $\delta$ with $-\delta$.
\end{remark}

\subsubsection*{Coxeter elements}

One of the most important characteristics of well generated groups is that they are precisely the reflection groups that have analogues of Coxeter elements. For an arbitrary complex reflection group $W$ of rank $n$, Gordon and Griffeth \cite{GG} defined the \defn{Coxeter number} $h$ as 
\[
\defn{h} := \dfrac{\#\RRR+\#\mathcal{A}_W}{n}.
\]
Following Springer \cite{springer}, we say that a vector $v\in V$ is a \defn{regular vector} if it does not belong to any reflection hyperplane $H\in\mathcal{A}_W$, and that an element $g\in W$ is a \defn{$\zeta$-regular element} if it has a regular $\zeta$-eigenvector. If we do not need to emphasize the value of $\zeta$ we may simply call $g\in W$ a \emph{regular element}. For a given $\zeta\in\CC^\times$, all $\zeta$-regular elements are conjugate and their order equals the multiplicative order of $\zeta$. 
A \defn{Coxeter element} is defined as an $e^{2\pi i/h}$-regular element of $W$ and a \defn{generalized Coxeter element} as any regular element of order $h$ \cite{RRS}. In the case of real reflection groups, this definition agrees with the one we gave in \S\ref{Sec: real refl grps} (see the discussion above \cite[Def.~3.1]{chapuy_theo}). It is an easy consequence of \cite[Prop.~4.2]{bessis-zariski} and standard properties of regular numbers that Coxeter elements exist solely for well generated groups.

\subsection{Reflection factorizations and reflection length}
\label{sec:reflen}

Starting with an arbitrary group $G$ and some generating set $S\subseteq G$,  the Cayley graph of $G$ with respect to $S$ determines a natural length function on the elements of the group: we may define the length of an element $g\in G$ as the size of the shortest path in the Cayley graph from the identity element to $g$.
In the case of a complex reflection group $W$, we pick as the distinguished generating set the set $\RRR$ of reflections.  Then the \defn{reflection length} \defn{$\lR(g)$} of an element $g\in W$ is the smallest number $k$ such that there exist reflections $t_1, \ldots, t_k$ whose product $t_1 \cdots t_k$ is equal to $g$. 

\begin{definition}\label{Defn: Red_W(g) and Fred(g)}
For an arbitrary number $N$, if $t_1,\ldots,t_N$ are reflections such that $t_1\cdots t_N=g$, we say that the tuple $( t_1,\ldots, t_N)$ is a \defn{reflection factorization} of $g$ of length $N$. If moreover $N=\lR(g)$, we say that the factorization is \defn{reduced}. We denote the set of reduced reflection factorizations of $g\in W$ by \defn{$\op{Red}_W(g)$}, i.e.,
\[
\op{Red}_W(g):=\left\{( t_1,\ldots, t_{\lR(g)})\in\RRR^{\lR(g)} \colon  t_1\cdots t_{\lR(g)}=g\right\},
\]
and we denote by \defn{$\Fred(g)$} their number (i.e., the cardinality of the set $\op{Red}_W(g)$).
\end{definition}

The reflection length is subadditive over products: for any $a, b$ one has $\lR(a b) \leq \lR(a) + \lR(b)$.  Thus, it determines a partial order $\defn{\leq_{\RRR}}$ (the \defn{absolute order}) on the elements of $W$ via
\[
u \leq_{\RRR} v \qquad \iff \qquad \lR(u)+\lR(u^{-1}v)=\lR(v).
\]
Equivalently, $u \leq_{\RRR} v$ if and only if every reduced reflection factorization of $u$ can be extended to give a reduced reflection factorization of $v$.

If $W$ is either a real reflection group or the wreath product $G(m, 1, n)$, there is a simple geometric formula for reflection length: one has in these cases that $\lR(w) = \codim(V^w)$ for all $w$ in $W$ \cite[Lem.~2]{carter}, \cite[Rem.~2.3]{Shi2007}.  For the other complex reflection groups, the situation is more complicated: it is always the case that $\lR(w) \geq \codim(V^w)$, but there exist elements where equality is not achieved \cite{FosterGreenwood}.  

The subgroups $W'\leq W$ that are generated by the reflections in a reduced reflection factorization of some element $g\in W$ are of special interest. For a Weyl group $W$ and a fixed $g\in W$, all such subgroups $W'$ share an important invariant: they have the same connection index.\footnote{There is a notion of connection index for general complex reflection groups (see \S\ref{sec:8:4}), but this statement is no longer valid in that generality and it is unclear how to extend it.} This follows from the next proposition, which appears in Wegener's thesis as part of the proof of \cite[Thm.~4.2.15]{Wegener_thesis}. Recall that the \defn{pseudo-determinant $\op{pdet}(A)$} of a linear transformation $A$ is defined as the product of the non-zero eigenvalues of $A$. 

\begin{proposition}[{see proof of Thm.~4.2.15 in \cite{Wegener_thesis}}]\label{Prop: pdet(g-id)=I(W')}
Let $W \subset \GL(V)$ be a Weyl group of rank $n$, $g\in W$ an element, and $(t_1,\ldots, t_k)\in \op{Red}_W(g)$ a reduced reflection factorization of $g$. If $W'$ is the reflection subgroup of $W$ generated by the factors $t_i$, then
\[
I(W') = \big|\op{pdet}(g-I_V)\big|,
\]
where $I_V$ is the identity on $V$ and $I(W')$ denotes the connection index of $W'$.
\end{proposition}

\subsubsection*{The combinatorial family}

In the infinite families $W=G(m,1,n)$ and $W=G(m,m,n)$, one may also give combinatorial formulas for reflection length.
\begin{theorem}[{\cite[Thm.~4.4]{Shi2007}}]
\label{thm:reflection length in G(m, 1, n) and G(m, p, n)} \;
\begin{enumerate}[(i)]
    \item Given an element $w$ in $G(m,1,n)$, its reflection length is $$\lR(w) = n-c_0(w),$$ where $c_0(w)$ is the number of cycles in $w$ of color $0$.
    \item Given an element $w$ in $G(m, m, n)$, its reflection length is $$\lR(w)=n + c(w) - 2v_m(w),$$ where $c(w)$ is the number of cycles in $w$ and $v_m(w)$ is the largest number of parts into which it is possible to partition the cycles of $w$ such that all parts have total color $0$.
\end{enumerate}
\end{theorem}
In the case that $m = 1$, so that $W = \Symm_n$, both formulas reduce to the usual formula $\lR(w) = n - c(w)$ for reflection length in the symmetric group, where $c(w)$ is the number of cycles of the permutation $w$.

\subsection{Parabolic subgroups}
\label{Sec: Par. subgrps}

As in the real case, the reflection arrangement $\mathcal{A}_W$ of a complex reflection group $W$ determines much of its combinatorics and group theory. Its \defn{flats} (arbitrary intersections of its hyperplanes) are linear subspaces; ordered by reverse-inclusion, they form the \defn{intersection lattice} of $W$. 

The pointwise stabilizer $W_U$ of any set $U\subseteq V$ will be called a \defn{parabolic subgroup} of $W$.  (For real reflection groups, this definition is a priori different than the one we gave in \S\ref{Sec: real refl grps}, but it is a standard theorem that they agree, see \cite[\S 5-2]{Kane}.) It is an easy observation, but so far based on the classification, that parabolic subgroups of well generated groups are themselves well generated. This means that (as discussed in \S\ref{Sec: well-gend grps}) they will have their own Coxeter elements; we call such elements \defn{parabolic Coxeter elements}. Arbitrary (i.e., not necessarily well generated) reflection groups may have well generated parabolic subgroups, but in this paper we only discuss parabolic Coxeter elements in well generated groups. The following theorem of Steinberg characterizes parabolic subgroups. 

\begin{theorem}[{Steinberg's theorem \cite[Thm.~1.5]{Steinberg}}]\label{Thm: steinberg}
Let $W$ be a complex reflection group acting on $V$ and $U\subseteq V$ a subset. Then the pointwise stabilizer of $U$ is a reflection group, generated by those reflections in $W$ whose reflection hyperplanes contain $U$.
\end{theorem}

It follows from Steinberg's theorem that, for any $U \subseteq V$, $W_U$ is the pointwise stabilizer of the (unique) minimal flat $X$ that contains $U$.  Thus, we index parabolic subgroups by their corresponding flat $X$.  For a flat $X$, the parabolic subgroup $W_X$ acts naturally as a complex reflection group on the orthogonal complement $X^{\perp}$, with rank equal to the codimension of $X$.

As a particular example, consider the case that $W_1$ and $W_2$ are irreducible complex reflection groups acting on $V_1$ and $V_2$, respectively, where then the product $W := W_1 \times W_2$ acts on $V := V_1 \oplus V_2$.  In this case the two factor subgroups $W_1 \cong W_1 \times \{\id_{W_2}\}  = W_{V_2} $ and $W_2 \cong \{\id_{W_1}\} \times W_2 = W_{V_1}$ are parabolic subgroups of $W$. Consequently, if $W$ is any complex reflection group that decomposes into irreducibles as $W_1 \times \cdots \times W_k$, then each of the $W_i$ is a parabolic for $W$ in a natural way. Furthermore, if $W_X\leq W$ is some parabolic subgroup, then its decomposition into irreducibles will be given as $W_X=W_{X_1}\times\cdots\times W_{X_k}$, where the $W_{X_i}$ are parabolic subgroups of $W$ and satisfy $\codim(X)=\sum_{i=1}^k\codim(X_i)$.

It follows from Steinberg's theorem that the collection of parabolic subgroups is closed under intersection. For any element $g\in W$ we define the \defn{parabolic closure} $W_g$ of $g$ to be the intersection of all parabolic subgroups that contain $g$. In this case the flat $X$ indexing $W_g$ is the fixed space $X=V^g$, and in particular we have $\rank(W_g)=\codim(V^g)$. The orthogonal complement $X^{\perp} = (V^g)^{\perp}$ is then known as the \defn{move space} of $g$ and is equal to $\defn{\Mov(g)}:=\op{Im}(g-\id)$.

In the real case, Bessis proved \cite[above Lem.~1.4.2]{bessis-dual-braid} that if some element $g\in W$ belongs to a parabolic subgroup $W_X\leq W$, then all its reduced reflection factorizations are realizable in $W_X$ (that is, all the reflections in these factorizations belong to $W_X$ as well). We extend this statement to complex groups $W$, but only for the family of elements $g\in W$ that satisfy $\lR(g)=\codim(V^g)$ (for a discussion of the general case, see \S \ref{sec: gen_quest}). This case is precisely what we will need later on (for instance for Proposition~\ref{Prop: gob-cycle-decomp}).

\begin{proposition} \label{Prop: Red_W=Red_(W_g)}
Let $W$ be a complex reflection group and $g\in W$ an element for which $\lR(g)=\codim(V^g)$. The following hold:
\begin{enumerate}
\item Every reflection $t$ that satisfies $t\leq_{\RRR} g$ belongs to the parabolic closure $W_g$ of $g$.
\item For any parabolic subgroup $W_X\leq W$ that contains $g$, we have the following equality of sets:
\[
\op{Red}_W(g)=\op{Red}_{W_X}(g).
\]
\end{enumerate}
\end{proposition}
\begin{proof}
Assume that $k=\lR(g)$ and that $t_1\cdots t_k=g$ is a reduced reflection factorization of $g$ in $W$. If we write $Y := \bigcap_i V^{t_i}$ for the common intersection of the fixed hyperplanes $V^{t_i}$, it is clear that $g$ fixes it pointwise; that is, $V^g\supset Y$. On the other hand, since $Y$ is the intersection of $k$ hyperplanes, we have that $\codim(Y)\leq k=\codim(V^g)$.  Thus $Y = V^g$.

Since $V^g = \bigcap_i V^{t_i}$, we have $V^{t_i}\supset V^g$ for all factors $t_i$. By Steinberg's theorem (Thm.~\ref{Thm: steinberg}) we must then have that $t_i\in W_g$. This means that all factors in every reduced reflection factorization of $g$ must belong to $W_g$, and this completes the proof of part (1). Since all parabolic subgroups $W_X$ that contain $g$ must, by definition, contain $W_g$ as well, the second part is also proven.
\end{proof}

Proposition~\ref{Prop: Red_W=Red_(W_g)} allows us to give a finer description of the reduced reflection factorizations of such elements in terms of their parabolic closures.

\begin{corollary}
\label{cor:finer description reduced refl fact}
Let $W$ be a complex reflection group and $g\in W$ an element for which $\lR(g)=\codim(V^g)$.  Let $W_g=W_1\times\cdots\times W_s$ be the decomposition of the parabolic closure $W_g$ into irreducibles and let $g=g_1\cdots g_s$ be the corresponding expression for $g$ with $g_i\in W_i$. Then $W_i=W_{g_i}$ and $\lR[W_i](g_i)=\lR(g_i)=\codim(V^{g_i})=\rank(W_i)$ for $i = 1, \ldots, s$.
\end{corollary}
\begin{proof}
For each $i$, $W_i$ is a parabolic subgroup of $W$ that contains $g_i$.  If there were a smaller parabolic subgroup $W'_i$ of $W$ that contained $g_i$, then $W_1 \times \cdots \times W'_i \times \cdots\times W_s$ would be a parabolic subgroup of $W$ containing $g$ that is smaller than $W_g$, a contradiction.  Consequently, $W_i=W_{g_i}$ and so also $\codim(V^{g_i})=\rank(W_i)$. 
Combining Proposition~\ref{Prop: Red_W=Red_(W_g)}~(2) with the fact that reflection length is additive over direct product decompositions gives
\begin{equation} \label{eq: sum of lengths}
\lR(g) = \lR[W_g](g) = \sum_{i=1}^s \lR[W_i](g_i),
\end{equation}
and therefore
\[
\lR(g) = \sum_{i=1}^s \lR[W_i](g_i) \geq \sum_{i=1}^s \lR(g_i) \geq \sum_{i=1}^s \codim(V^{g_i}) = \sum_{i=1}^s \rank(W_i) = \rank(W_g) = \lR(g),
\]
where the first inequality is direct from the definition of reflection length, the second inequality is discussed in \S \ref{sec:reflen}, rank-additivity comes from the discussion preceding Proposition~\ref{Prop: Red_W=Red_(W_g)}, and the last equality is the hypothesis on $g$. This completes the proof
.
\end{proof}

\begin{remark}\label{rem_par_1}
Proposition~\ref{Prop: Red_W=Red_(W_g)}~(2) is no longer true if we replace $W_X$ with an arbitrary reflection subgroup $W'$ (i.e., not necessarily parabolic), even in the real case.  For example, inside $G(2, 1, 2) = B_2$, consider the element $g = [\id;(1,1)]$ in $G(2, 1, 2)$ and the subgroup $W'$ of diagonal matrices, generated by the reflections $[\id;(1,0)]$ and $[\id;(0,1)]$.  Then $\op{Red}_{W'}(g)$ fails to contain the factors in the reduced $W$-reflection factorization $g=[(12);(0,0)]\cdot [(12);(1,1)]$.
\end{remark}

\begin{remark}\label{rem_par_2}
Another implication of Proposition~\ref{Prop: Red_W=Red_(W_g)}~(2) is that the reflection length of $g$ with respect to a parabolic subgroup $W_X$ is the same as with respect to $W$. In general, if we replace $W_X$ by an arbitrary reflection subgroup $W'$, this is no longer true. Unlike the situation of the Remark~\ref{rem_par_1}, however, this is solely a phenomenon of the complex types, since in the real case we always have $\lR(g)=\codim(V^g)$. For example, by Theorem~\ref{thm:reflection length in G(m, 1, n) and G(m, p, n)}, the element $[\id; (1, 1, 1)]$ has reflection length $3$ as an element of $G(3, 1, 3)$ (it is a product of three diagonal reflections), but it cannot be written as a product of fewer than four reflections in the non-parabolic subgroup $G(3, 3, 3)$. 
\end{remark}

\subsubsection*{Parabolic subgroups of the combinatorial family}

The parabolic subgroups of the group $G(m, p, n)$ are easy to describe.

\begin{theorem}[{essentially \cite[Thm.~3.11]{taylor}}]
\label{thm:parabolic subgroups of G(m, p, n)}
Every parabolic subgroup of $G(m, p, n)$ is either conjugate to a subgroup of the form 
\[
G(m, p, \lambda_0) \times \Symm_{\lambda_1} \times \cdots \times \Symm_{\lambda_k}
\]
for some partition\footnote{Note that $\lambda_0$ could be out of order in $\lambda$.} $(\lambda_0, \ldots, \lambda_k)$ of $n$ or conjugate by an element of $G(m, 1, n)$ to a subgroup of the form
\[
\Symm_{\lambda_1} \times \cdots \times \Symm_{\lambda_k}
\]
for some partition $(\lambda_1, \ldots, \lambda_k)$ of $n$.
\end{theorem}

\subsection{The Hurwitz action}
\label{sec: hurwitz action}

For any group $G$, there is a natural action of the $k$-strand braid group  
\[
\defn{\BBB_{k}} := \left\langle \sigma_1, \ldots, \sigma_{k - 1} \; \middle| \begin{array}{ll}
\sigma_i \sigma_{i + 1} \sigma_i = \sigma_{i + 1} \sigma_i \sigma_{i + 1} & \text{for } i = 1, \ldots, k - 2 \\
\sigma_i \sigma_j = \sigma_j \sigma_i & \text{if } |i - j| > 1
\end{array}
\right\rangle
\] 
on the set of $k$-tuples of elements of $G$. The generator $\sigma_i$ acts via
\[
\sigma_i(g_1,\cdots, \quad g_i, \quad g_{i+1}, \quad \cdots,g_k)=(g_1,\cdots, \quad g_{i+1}, \quad g^{-1}_{i+1}g_ig_{i+1}, \quad \cdots,g_k),
\]
swapping two adjacent elements and conjugating one by the other, preserving the product of the tuple. We call this the \defn{Hurwitz action} of $\BBB_k$ on $G^k$.

In addition to the product $g_1 \cdots g_k$, the Hurwitz action respects the subgroup $
\langle g_1,\ldots,g_k\rangle$ that is generated by the elements $g_i$ and the multiset of conjugacy classes they determine. This means that for a reflection group $W$, there is a well defined Hurwitz action on the set of length-$k$ reflection factorizations of an element $g\in W$.

The orbit structure of the Hurwitz action has interesting connections with other areas of mathematics.
In the symmetric group $\Symm_n$, for example, orbits often correspond to path-connected components of spaces of polynomials or holomorphic maps \cite[\S 5.4]{zdkh, LandoZvonkin}. In well generated complex reflection groups, the transitivity of the Hurwitz action in the following theorem was an important ingredient in Bessis' proof \cite{Bessis_Annals} of the $K(\pi,1)$ conjecture.

\begin{theorem}[case-free for real $W$ {\cite[Prop.~1.6.1]{bessis-dual-braid}}, case-by-case in general {\cite[Prop.~7.6]{Bessis_Annals}}]\label{Thm: Bessis Hurwitz trans pCox}
For a well generated group $W$ and a parabolic Coxeter element $g\in W$ of length $\lR(g)=k$, the Hurwitz
action of $\mathcal{B}_k$ on $\op{Red}_W(g)$ is transitive.\end{theorem}

The following result shows that every Hurwitz orbit of reflection factorizations in a real reflection group contains factorizations in a particularly simple form; it will be very useful in \S\ref{Sec: full factorizations and p-q-Cox} below.

\begin{lemma}[{\cite[Cor.~1.4]{LR}}]
\label{lem: LR lemma}
Let $W$ be a real reflection group and $g \in W$ an element of reflection length $\lR(w)=k$. Then any factorization of $g$ into $N$ reflections (with $N\geq k$) lies in the Hurwitz orbit of some tuple $( t_1,\ldots, t_N)$ such that 
\[
 t_1 = t_2, \qquad
 t_3 = t_4, \qquad
\ldots, \qquad
 t_{N-k-1}  = t_{N-k},
\]
and $( t_{N-k+1}, \ldots, t_N)$ is a reduced factorization of $g$.
\end{lemma}

\begin{remark} \label{Rem: WY lemma}
The proof of Lemma~\ref{lem: LR lemma} given in \cite{LR} relies on case-by-case calculations. Recently, Wegener and Yahiatene followed a different approach \cite{WY2} that furnishes a case-free proof for Lemma~\ref{lem: LR lemma}. This make the proofs of some of our statements below also case-free; see Remarks~\ref{Rem: uniform proof of lfull(id)} and~\ref{Rem: uniform proof lfull(g)}.
\end{remark}

\section{Parabolic quasi-Coxeter elements} \label{sec.parabolic qce}

In this section we introduce a family of elements in (well generated) complex reflection groups that we call \emph{parabolic quasi-Coxeter elements}, and that are the main objects of study in this paper.  In real reflection groups, these elements have appeared in multiple guises in the literature, with various definitions and names -- see Remark~\ref{rem:ToB}, Theorem~\ref{Prop: pqCox characterization for Weyl W}, and Corollary~\ref{Cor: pqCox characterization for real W}.  We begin this section with a historical review of these various appearances, which we hope will convince the reader that we are discussing natural objects that are of interest in a variety of mathematical areas.  After the overview, in \S\ref{ssec.expl.defn}, we give the formal definition that we use in the rest of the paper, and establish its equivalence with the earlier definitions in the real case.  We then establish a number of structural properties of parabolic quasi-Coxeter elements, including an analogue of the cycle decomposition of permutations (\S\ref{sec: cycle decomposition}) and the transitivity of the Hurwitz action on their reduced factorizations (\S\ref{ssec.pcqe.hurw.action}). We rely on the results we prove here as we present various characterizations for the parabolic quasi-Coxeter elements in the following \S\S\ref{Sec: relative gen sets} and~\ref{Sec: full factorizations and p-q-Cox}.

\subsection{Origins of (parabolic) quasi-Coxeter elements}
\label{sec.origins}

In a real reflection group $W$, the easiest way to describe the family of \emph{quasi-Coxeter} elements is as the elements that do not belong to any proper reflection subgroup of $W$. One could then define \emph{parabolic} quasi-Coxeter elements as the quasi-Coxeter elements of parabolic subgroups of $W$. This is the context underlying the earliest appearance of quasi-Coxeter elements, in Carter's classification of the conjugacy classes of Weyl groups.

\subsubsection{Conjugacy classes of Weyl groups}
\label{sec:origins_con_cl}
At the time of Carter's work \cite{carter}, the conjugacy classes of Weyl groups had been already described for all cases, but there was no case-free construction for them. The Borel--de Siebenthal algorithm (\cite{borel_de-sieb}, see also \cite[\S 12.1]{Kane}) gave a case-free way to build up all reflection subgroups of $W$, and Carter used it to reduce the description of all conjugacy classes of $W$ to the description of those that are not contained in any proper reflection subgroup. He didn't give a name for these latter classes, but he showed that they always contain the class of Coxeter elements and he described the rest via certain Dynkin-like graphs that he called \emph{admissible diagrams}. 

Carter's admissible diagrams can be constructed for any element $g$ in a Weyl group $W$. They encode the commutativity relations among reflections that form certain special reduced factorizations of $g$. Carter showed that such diagrams are sufficient to distinguish the quasi-Coxeter conjugacy classes and gave a complete list of them (see Table~\ref{Table: qCox elts Carter}). 

\begin{table}[ht]
\[
\begin{array}{|c|| c|} 
 \hline
 \text{Weyl group} & \text{list of quasi-Coxeter conjugacy classes} \\ 
 \hline
 \Symm_{n+1}=A_n & A_n \\
 B_n & B_n \\
 D_n & D_n,\ D_n(a_1),\ D_n(a_2),\ \ldots,\ D_n(a_{\lfloor(n-4)/2\rfloor}),\ D_n(a_{\lfloor(n-2)/2\rfloor})=D_n(b_{\lfloor(n-2)/2\rfloor})\\
 G_2 & G_2\\
 F_4 & F_4,\ F_4(a_1) \\
 E_6 & E_6,\ E_6(a_1),\ E_6(a_2)\\
 E_7 & E_7,\ E_7(a_1),\ E_7(a_2)=E_7(b_2),\ E_7(a_3),\ E_7(a_4)\\
 E_8 & E_8, E_8(a_1), E_8(a_2), E_8(a_3)\!=\!E_8(b_3), E_8(a_4), E_8(a_5)\!=\!E_8(b_5), E_8(a_6), E_8(a_7), E_8(a_8)\\
 \hline
\end{array}
\]
\caption{Carter's notation for conjugacy classes of Weyl groups $W$ that do not live in proper reflection subgroups of $W$ \cite{carter}. Some classes have two distinct associated diagrams, indexed $a_i$ and $b_i$. 
}
\label{Table: qCox elts Carter}
\end{table}

\subsubsection{Distinguished bases in singularity theory}
\label{sec:origins_dist_bas}
In the 1970s, Arnold \cite{Arnold_Congress} and his school developed a very rich mathematical area which studies the geometry and topology of hypersurface singularities. The second (chronologically) appearance of quasi-Coxeter elements took place in this singularity theory. An important object in this setting is the \emph{Milnor lattice} of the singularity; it is defined via the Milnor fibration and deformation theory, and encodes many invariants associated to the singularity. There is a natural geometric construction for certain bases of the Milnor lattice: vaguely stated, the Picard--Lefschetz transformations associated to a collection of paths around the critical values of a deformation of the singularity always determine a basis. In general, such bases are called \emph{weakly distinguished}, while if one further requires that the paths are non self-intersecting, the resulting bases are called \emph{distinguished}.\footnote{Unfortunately, many authors have used different, conflicting names for this pair of properties (distinguished, weakly distinguished). Two common variants are (\emph{geometric}, \emph{distinguished}), used by Voigt \cite{voigt_thesis}, and (\emph{strongly distinguished},  \emph{weakly distinguished}).} There is significant interest in the enumeration and study of structural properties of these bases \cites{zade_milnor,ebeling_milnor,brieskorn_milnor}.

Part of the theory developed by Arnold and his school describes an infinite hierarchy of singularities that starts with the so-called \emph{simple singularities}, which are in natural correspondence with the simply laced Weyl groups $A_n, D_n, E_6, E_7, E_8$. In particular, the Milnor lattice of a singularity is precisely the root lattice of the corresponding Weyl group $W$, and its distinguished bases correspond to the ordered sets of roots appearing in reduced reflection factorizations of some Coxeter element of $W$ (see \cite{brieskorn_milnor}). 

In his thesis \cite{voigt_thesis}, Voigt extends the previous correspondence to weakly distinguished bases. He defines \emph{quasi-Coxeter} elements as those whose reduced reflection factorizations determine a basis of the root lattice and he shows that all such bases are weakly distinguished. A major ingredient in Voigt's proof is that the Hurwitz action is transitive on the reduced factorizations of quasi-Coxeter elements.

\subsubsection{Dual Coxeter systems}
\label{sec:origins_dual_cox}

A third, more recent, appearance of quasi-Coxeter elements is in the context of the \emph{dual approach} to Coxeter systems. Initiated in the seminal work of Bessis \cite{bessis-dual-braid}, the dual approach to a reflection group $W$ and its generalized braid group $B(W)$ involves considering the whole set of reflections $\RRR$ as a generating set of $W$, instead of just the simple reflections, and studying the (more symmetric) presentations that arise in this way. Bessis proved that the Artin and dual presentations of $B(W)$ are equivalent \cite{bessis-dual-braid, Bessis_Annals}; an important component of his proofs is that the Hurwitz action is transitive on the set of reduced reflection factorizations of a Coxeter element (see also the exposition in \cite[\S 6]{chapuy_theo_lapl}).

Bessis' work sparked considerable interest in the Hurwitz action on reflection factorizations (for instance in \cite{michel-hurwitz-tuples} and \cite{BDSW}).  In~\cite{BGRW}, Baumeister et al.\ set out to describe its orbit structure in real reflection groups $W$. They found that the elements $g\in W$ for which the Hurwitz action is transitive on their set of reduced reflection factorizations are precisely the elements whose reduced factorizations form good generating sets for parabolic subgroups of $W$. They called such elements \emph{parabolic quasi-Coxeter} elements.

It is this last definition of the quasi-Coxeter property that is most relevant in our work. In the following section, we extend it to the case of well generated complex reflection groups. Although the result of \cite{BGRW} (involving Hurwitz orbits) does not extend to this setting (see \S\ref{ssec.pcqe.hurw.action} below), it turns out that it is precisely their definition that allows us to relate the quasi-Coxeter property with full reflection factorizations, as we do \S\S\ref{Sec: relative gen sets} and~\ref{Sec: full factorizations and p-q-Cox} below.

\begin{remark}[Tower of Babel]
\label{rem:ToB}

We hope that the preceding discussion makes clear that the class of quasi-Coxeter elements is natural, and in fact prominent, when viewing the theory of reflection groups from several different perspectives. However, because of the distance separating the various relevant research areas, this same class of elements often appears with different names in the literature, and, worse, sometimes the same name has been used to refer to distinct properties.

Carter \cite{carter} did not give a name to the conjugacy classes of Table~\ref{Table: qCox elts Carter} but he called their associated diagrams \emph{admissible diagrams}. In a later paper \cite{carter_semi-cox}, Carter and Elkington gave an explicit definition for a class of elements that they called \emph{semi-Coxeter elements}.  Taken literally, their definition accidentally allows more classes than those of Table~\ref{Table: qCox elts Carter}, but subsequent authors (e.g., \cite{BGRW,enhanced_dynkin}) have used ``semi-Coxeter'' to mean exactly the classes of the table.

As discussed in \S\ref{sec:origins_dist_bas}, Voigt \cite{voigt_thesis} used the term \emph{quasi-Coxeter elements} for the elements with a reduced reflection factorization whose corresponding roots form a basis of the root lattice.  In the simply laced Weyl groups (types $A_n$, $D_n$, $E_n$), this definition agrees with the definition given by Baumeister et al.~\cite{BGRW} (discussed above in \S\ref{sec:origins_dual_cox}; and see Theorem~\ref{Prop: pqCox characterization for Weyl W} below for the proof).  However, in other (non-simply laced) types, the two definitions are not equivalent.  Both versions have appeared in the literature: Hertling and Balnojan \cite{BalnojanHertling} used Voigt's definition unchanged for (non-simply laced) Weyl groups, while Proposition~\ref{Prop: BW charn of gen sets} (of Baumeister and Wegener \cite{BW}) suggests that it is better to add to the definition the condition that the corresponding coroots form a basis of the coroot lattice as well.

In many sources (e.g., \cite{bouwknegt, kac_peterson, delduc-feher}), the conjugacy classes of quasi-Coxeter elements in a Weyl group $W\leq \GL(V)$ are called \emph{primitive classes}. They are usually defined as the conjugacy classes of elements $g\in W$ such that the determinant of $g-I_V$ equals the determinant of the Cartan matrix of $W$ up to sign (where $I_V$ is the identity on $V$). In these references, it is often observed that primitive classes are precisely those that do not live in any proper reflection subgroup of $W$ (again, see Theorem~\ref{Prop: pqCox characterization for Weyl W} for the equivalence of these definitions).

A different source of potential confusion comes from the labels used to denote the quasi-Coxeter conjugacy classes of a Weyl group $W$ (as in Table~\ref{Table: qCox elts Carter}). Carter's notation is similar to the later Bala--Carter notation \cite{BC_I,BC_II}, which encodes conjugacy classes of nilpotent elements of the Lie algebra associated to $W$. This is not an accident; Carter had already observed in \cite[\S 10]{carter} that admissible diagrams may form good combinatorial models also for such Lie algebra classes.
\end{remark}

\subsection{Explicit definition of (parabolic) quasi-Coxeter elements}
\label{ssec.expl.defn}

We are finally ready to define parabolic quasi-Coxeter elements for well generated complex reflection groups, extending the definition of \cite{BGRW}. 

\begin{definition}[parabolic quasi-Coxeter elements]\label{Defn: p-q-Cox elts}
Let $W$ be a well generated complex reflection group. We say that an element $g\in W$ is a \defn{parabolic quasi-Coxeter} element if it has a reduced reflection factorization $g=t_1 \cdots t_{\lR(g)}$ whose factors $\{t_i\}$ form a good generating set (as in Definition~\ref{Defn: good gen'g sets}) for some parabolic subgroup of $W$. We say it is \defn{quasi-Coxeter} if the factors $\{t_i\}$ form a good generating set for $W$.
\end{definition}

Definition~\ref{Defn: p-q-Cox elts} does not specify \emph{which} parabolic subgroup should be generated by the factors $t_i$. In the next proposition we see that, in fact, the subgroup is completely determined by the parabolic quasi-Coxeter element.

\begin{proposition} \label{Prop: X=V^g and l_R(g)=codim(V^g)}
Let $W$ be a complex reflection group and $g\in W$ an arbitrary element. If the factors of a reduced reflection factorization of $g$ form a good generating set for some parabolic subgroup $W_X$, then $X = V^g$ and $W_X$ is the parabolic closure of $g$. Moreover, in this case $\lR(g)=\codim(V^g)$.
\end{proposition}
\begin{proof}
Assume that $t_1 \cdots t_k=g$ is a reduced reflection factorization of $g$ as in the statement and that $W_X=\langle t_i\rangle$. Since the $t_i$ form a good generating set for $W_X$, we have that $\rank(W_X)=k$, or equivalently that $k=\codim(X)$. 
Consider now the parabolic closure $W_g$ of $g$, which is indexed by the flat $Y=V^g$. Since $g\in W_X$, we have $Y\supseteq X$, and therefore $\codim(Y)\geq k$. On the other hand, $g$ can be written as a product of $k$ reflections and hence has to fix the intersection of their fixed spaces; this forces $\codim(Y)\leq k$. Putting the two inequalities together, we have that $\codim(Y)=k$ and thus $Y=X$ and $W_g = W_Y = W_X$.
\end{proof}

With this preliminary result in hand, we are ready now to prove that the various definitions considered in \S\ref{sec.origins} are equivalent for Weyl groups. The following arguments are case-free and constitute a combination of results from \S\ref{Sec: preliminaries} that are mostly due to \cite{BGRW} and \cite{Wegener_thesis}.

\pqcoxWeylchar

\begin{proof}
It is sufficient to prove the case of parabolic quasi-Coxeter elements. The following relations prove the equivalence of all given statements.
\begin{enumerate}
    \item[$(i)\Rightarrow (ii)$:] For $g$ parabolic quasi-Coxeter, we have by definition that there is a reduced factorization $g=t_1\cdots t_k$ whose factors form a good generating set for some parabolic subgroup $W_X$. By Proposition~\ref{Prop: X=V^g and l_R(g)=codim(V^g)} this subgroup is $W_g$, and the conclusion follows by Proposition~\ref{Prop: BW charn of gen sets}.
    \item[$(ii)\Rightarrow (i)$:] Take a reduced reflection factorization $g=t_1\cdots t_k$ as in $(ii)$. Since $W$ is a real reflection group, we have $k=\codim(V^g)$, and since $\codim(V^g)=\rank(W_g)$, Proposition~\ref{Prop: BW charn of gen sets} implies that the factors $t_i$ generate $W_g$. Since $W_g$ is a parabolic subgroup of $W$, this means that $(i)$ holds. 
    \item[$(i)\Rightarrow (iii)$:] Pick a reduced reflection factorization $t_1\cdots t_k=g$ for which the factors $t_i$ generate $W_g$ as in the step $[(i)\Rightarrow (ii)]$. The required property $(iii)$ follows by Proposition~\ref{Prop: pdet(g-id)=I(W')}.
    \item[$(iii)\Rightarrow (iv)$:] Choose a reflection group $W'$ such that  $g \in W'\leq W_g$, and consider a $W'$-reduced reflection factorization $t_1\cdots t_k=g$ and the rank-$k$ group $W''=\langle t_i\rangle\leq W'$ generated by its factors. Since $W$ is a real reflection group, this factorization is also $W$-reduced and we have that that $k=\lR(g)=\codim(V^g)=\rank(W_g)$; in other words, the group $W''$ is a \emph{maximal rank} reflection subgroup of $W_g$. By Proposition~\ref{Prop: pdet(g-id)=I(W')}, we must have that $|\op{pdet}(g - I_V)|=I(W'')$, which forces by the assumption in $(iii)$ that $I(W'')=I(W_g)$. Since $W''$ is a maximal rank reflection subgroup of $W_g$, it follows from Corollary~\ref{Cor: W'<W means I(W')>I(W)} that $W''=W_g$, and therefore that $W' = W_g$.  Thus $g$ satisfies $(iv)$. 
    \item[$(iv)\Rightarrow (i)$:] Take an arbitrary reduced reflection factorization $t_1\cdots t_k=g$ and consider the group $W'=\langle t_i\rangle$ generated by its factors. Since $W$ is a real reflection group, Proposition~\ref{Prop: Red_W=Red_(W_g)} implies that all reflections $t_i$ belong to the parabolic closure $W_g$. Thus $W'\leq W_g$.  Since clearly $g$ belongs to $W'$, the assumption of $(iv)$ implies that $W'=W_g$. This proves $(i)$.  \qedhere
\end{enumerate}
\end{proof}

We end this section with one further preliminary result about the complex case.  The collection of reduced reflection factorizations of an element $g$ depends on the ambient group $W \ni g$.  This suggests that whether $g$ is parabolic quasi-Coxeter should depend on $W$. In the next statement we prove that, to the contrary, the property is hereditary inside the family of parabolic subgroups of $W$.

\begin{corollary}\label{cor: p-q-Cox in W_X implies in W}
Let $W$ be a well generated group and $W_X\leq W$ a parabolic subgroup. Then an element $g\in W_X$ is parabolic quasi-Coxeter in $W_X$ if and only if it is parabolic quasi-Coxeter in $W$.
\end{corollary}
\begin{proof}
First, suppose $g \in W_X \leq W$ is parabolic quasi-Coxeter in $W_X$.  By Proposition~\ref{Prop: X=V^g and l_R(g)=codim(V^g)}, we have that $\lR[W_X](g) = \codim(V^g)$.  Since $W_X \leq W$, we have $\lR(g) \leq \lR[W_X](g)$.  On the other hand, as mentioned in \S\ref{sec:reflen}, we have for all $w \in W$ that $\lR(w) \geq \codim(V^w)$, and thus $\lR(g) = \lR[W_X](g)$.  Therefore, any reduced $W_X$-reflection factorization $g= t_1 \cdots t_k$ is also reduced as a $W$-reflection factorization.  Since $g$ is parabolic quasi-Coxeter in $W_X$, the subgroup $W'$ generated by the $t_i$ is parabolic in $W_X$.  It is easy an easy consequence of Steinberg's theorem (Thm.~\ref{Thm: steinberg}) that a parabolic subgroup of a parabolic subgroup is parabolic in the parent group, and therefore $g$ has a reduced $W$-factorization that generates a parabolic subgroup of $W$, as claimed.

Conversely, assume that $g$ is parabolic quasi-Coxeter in $W$ and $g=t_1\cdots t_k$ is a reduced $W$-reflection factorization of $g$.  We have by Proposition~\ref{Prop: X=V^g and l_R(g)=codim(V^g)} that  
$\lR(g)=\codim(V^g)$. Therefore, by Proposition~\ref{Prop: Red_W=Red_(W_g)}~(2), $g=t_1\cdots t_k$ is also a reduced $W_X$-reflection factorization of $g$.  Again by Proposition~\ref{Prop: X=V^g and l_R(g)=codim(V^g)}, the group $\langle t_i \rangle$ is the parabolic closure in $W$ of $g$, which (by Steinberg's theorem) is also a parabolic subgroup of $W_X$.  Thus $g$ has a reduced $W_X$-reflection factorization whose factors are a good generating set for a parabolic subgroup of $W_X$, as needed.
\end{proof}

\subsection{Interaction with the Hurwitz action.}
\label{ssec.pcqe.hurw.action}

In the real case, parabolic quasi-Coxeter elements may be characterized in terms of the Hurwitz action on their reduced reflection factorizations.

\begin{proposition}[{\cite[Thm.~1.1]{BGRW}}]
\label{Prop: Hurw. trans for par.q-Cox}
Let $W$ be a real reflection group and $g\in W$ an element that has reflection length $\lR(g)=k$. Then the Hurwitz action of the braid group $\BBB_k$ on the set of reduced reflection factorizations of $g$ is transitive if and only if $g$ is a parabolic quasi-Coxeter element.
\end{proposition}

This allows us to extend the equivalence $(i)\Leftrightarrow (iv)$ of Theorem~\ref{Prop: pqCox characterization for Weyl W} to the set of all real reflection groups.

\begin{corollary}\label{Cor: pqCox characterization for real W}
Let $W$ be a real reflection group and $g$ an element of $W$. Then the following statements are equivalent. 
\begin{enumerate}[(i)]
    \item $g$ is a quasi-Coxeter element (respectively, parabolic quasi-Coxeter element).
    \item $g$ does not belong to any proper reflection subgroup of $W$ (resp., of $W_g$).
\end{enumerate}
\end{corollary}
\begin{proof}
Again, it is sufficient to prove the parabolic case. The direction $(ii)\Rightarrow (i)$ is proven exactly as in Theorem~\ref{Prop: pqCox characterization for Weyl W}.  For the other direction, pick a reduced reflection factorization $g=t_1\cdots t_k$ whose factors generate a parabolic subgroup $W_X$. Since $W$ is a real reflection group, we have $k=\codim(V^g)$.  Then Proposition~\ref{Prop: X=V^g and l_R(g)=codim(V^g)} implies $W_X=W_g$. Since the group generated by the factors is invariant under the Hurwitz action, Proposition~\ref{Prop: Hurw. trans for par.q-Cox} implies that for \emph{every} reduced reflection factorization of $g$, the factors generate $W_g$.  Now suppose that $W'$ is a reflection subgroup such that $g \in W' \leq W_g$.  Since $W$ is a real reflection group, a reduced factorization in $W'$ must also be $W$-reduced (because $\lR[W'](g) = \codim(V^g) = \lR(g)$), and therefore must generate $W_g$ (because every reduced factorization of $g$ does).  Thus $W' = W_g$, and so $g$ does not belong to any proper reflection subgroup of $W_g$.
\end{proof}

In a real reflection group $W$, for every reduced reflection factorization $g=t_1\cdots t_k$, the element $g$ is a quasi-Coxeter element of the reflection subgroup $W'=\langle t_i \rangle$ generated by the factors. This observation, together with Corollary~\ref{Cor: pqCox characterization for real W}, leads to a finer understanding of the whole collection of Hurwitz orbits on the set $\op{Red}_W(g)$ of reduced reflection factorizations of $g$. The following is a slight rephrasing of \cite[Cor.~3.11]{gobet-cycle}. 

\begin{proposition}\label{Prop:gob.char.hurw.orb}
Let $W$ be a real reflection group and $g\in W$ an arbitrary element of $W$.  There is a one-to-one correspondence between the Hurwitz orbits on $\op{Red}_W(g)$ and the minimal (under inclusion) reflection subgroups $W'\leq W$ that contain $g$.
\end{proposition}
\begin{proof}
Gobet showed \cite[Cor.~3.11]{gobet-cycle} that the Hurwitz orbits on $\op{Red}_W(g)$ are in bijection with the reflection subgroups $W'$ of $W$ that contain $g$ as a quasi-Coxeter element. The statement now follows from the previous Corollary~\ref{Cor: pqCox characterization for real W}.
\end{proof}

Proposition~\ref{Prop: Hurw. trans for par.q-Cox} is not valid for all complex reflection groups -- for example, \cite[Ex.~4.3]{LW} shows that it fails in the exceptional group $G_{16}$.  However, some weaker forms of the result hold; we will need two of these in what follows.

\begin{proposition}[{\cite[Cor.~5.4]{LW}}]
\label{Prop: LW G(m, p, n) transitivity}
Let $W = G(m, p, n)$ and $g\in W$ an element of reflection length $\lR(g)=k$ with a reduced factorization that generates $W$. Then the Hurwitz action of the braid group $\BBB_k$ on the set of reduced reflection factorizations of $g$ is transitive.
\end{proposition}

\begin{proposition}
\label{Prop. strong=weak}
Let $W$ be a complex reflection group and $g \in W$ a parabolic quasi-Coxeter element, and let $g=t_1\cdots t_k$ be \emph{any} reduced reflection factorization of $g$. Then the set $\{ t_1,\ldots,t_k\}$ of factors is a good generating set for the parabolic closure $W_g$ of $g$.
\end{proposition}

(For arbitrary \emph{Coxeter} groups, the analogous statement has been conjectured by Gobet \cite[Conj.~2.4]{gobet-cycle}.)

\begin{proof}
In \cite[Thm.~5.6]{LW}, it is proved that if $G$ is any finite complex reflection group and $h \in G$ satisfies $\lR[G](h) = \rank(G)$ and has a reduced reflection factorization that generates $G$, then \emph{every} reduced reflection factorization of $h$ generates $G$.  Since the given element $g$ is parabolic quasi-Coxeter, by definition, $g$ has a reduced reflection factorization that generates a parabolic subgroup of $W$; by Proposition~\ref{Prop: X=V^g and l_R(g)=codim(V^g)}, this subgroup is its parabolic closure $W_g$, and $g$ has reflection length $\rank(W_g)$.  Applying the former result in the case $G = W_g$ completes the proof.
\end{proof}

\begin{remark}
The cited result \cite[Thm.~5.6]{LW} was proved by combining Proposition~\ref{Prop: LW G(m, p, n) transitivity} and an exhaustive computer check for exceptional groups $W$ that compared all Hurwitz orbits in $\op{Red}_W(g)$ for all conjugacy classes of elements $g\in W$. For Weyl groups, however, Proposition~\ref{Prop. strong=weak} has an elegant case-free proof in \cite[Thm.~4.3.1]{Wegener_thesis}.
\end{remark}

\begin{remark}
Proposition~\ref{Prop:gob.char.hurw.orb} also fails in the case of complex reflection groups, but here it is difficult even to conjecture an appropriate geometric object that indexes the Hurwitz orbits. For the combinatorial family, see \cite[Thm.~3.2 and~\S 6.1]{LW}.  
\end{remark}

\subsection{Unique cycle decomposition}
\label{sec: cycle decomposition}

In \cite[Thm.~1.3]{gobet-cycle}, Gobet describes an extension of the cycle decomposition of permutations in $\Symm_n$ to the collection of parabolic quasi-Coxeter elements in real reflection groups. Much of his work extends essentially verbatim to complex reflection groups, although the proofs of some supporting statements (our Propositions~\ref{Prop: Red_W=Red_(W_g)} and~\ref{Prop: X=V^g and l_R(g)=codim(V^g)}) need to be rewritten for the complex case. We sketch his argument below, with the necessary adjustments.

\begin{proposition}[generalized cycle decomposition]
\label{Prop: gob-cycle-decomp}
Let $W$ be a well generated complex reflection group and $g\in W$ a parabolic quasi-Coxeter element. Then there exists a unique (up to reordering the factors) decomposition $g=g_1\cdots g_r$ with $g_i\in W$ non-identity elements such that
\begin{enumerate}
\item $g_ig_j=g_jg_i$ for all $i$ and $j$,
\item $\lR(g)=\lR(g_1)+\cdots+\lR(g_r)$, and
\item no $g_i$  can be further decomposed; that is, there do not exist elements $a,b\in W$ such that $g_i=ab$ and the factorization $g_i=ab$ satisfies (1) and (2).
\end{enumerate}
Moreover, if $W_g=W_1\times\cdots\times W_s$ is the decomposition of the parabolic closure $W_g$ into irreducibles, then $r=s$ and (after an appropriate reordering) each factor $g_i$ is a quasi-Coxeter element of $W_i$.
\end{proposition}

\begin{proof}[Proof sketch]
Let us start with the decomposition $W_g = W_1 \times \cdots \times W_s$ of $W_g$ into irreducible parabolic subgroups, as in \S\ref{Sec: Par. subgrps}, and consider the resulting factorization
\begin{equation}
\label{EQ: g=g1..gs}
g=g_1\cdots g_s
\end{equation}
with $g_i \in W_i$ for each $i$.  Part (1) follows from the fact that the decomposition of $W_g$ is a direct product, while part (2) follows from \eqref{eq: sum of lengths} in the proof of Corollary~\ref{cor:finer description reduced refl fact}. We now consider part (3).

Any reduced reflection factorization of $g$ determines for each $i$ a $W_i$-factorization of $g_i$, and therefore each $g_i$ is a quasi-Coxeter element of $W_i$ according to Definition~\ref{Defn: p-q-Cox elts}.  Thus, it is sufficient to deal with the case that $g$ is quasi-Coxeter in an irreducible group $W$. We assume that the factorization $g=ab$ satisfies parts (1) and (2) and show that this contradicts the irreducibility of $W$. By condition (2), there is a reduced factorization $g=t_1\cdots t_n$ and an index $k$ such that $a=t_1\cdots t_k$ and $b=t_{k+1}\cdots t_n$. Consider the two reflection subgroups $W_1:=\langle t_1,\ldots,t_k\rangle$ and $W_2:=\langle t_{k+1},\ldots,t_n\rangle$. 
Then
\begin{multline*}
    \codim(V^{ab}) \overset{\text{ Prop.~\ref{Prop: X=V^g and l_R(g)=codim(V^g)} }}{=} 
    \lR(ab) \overset{\text{(2)}}{=} 
    \lR(a) + \lR(b) \geq \\ \codim(V^a)+\codim(V^b) \geq \codim(V^a\cap V^b) \geq \codim(V^{ab}).
\end{multline*}
This forces the inequalities to be equalities, which implies that the move spaces (defined in \S\ref{Sec: Par. subgrps}) of $a$ and $b$ intersect trivially. Pick an arbitrary vector $u\in \Mov(b)$, say $u = b(v) - v$.  Since $ab=ba$ we have $a(u) = b(a(v)) - a(v) \in \Mov(b)$.  Hence $a(u)-u \in \Mov(a)\cap\Mov(b) = \{0\}$. This means that $u\in V^a$ and thus that $\Mov(b)\subseteq V^a$.  By Proposition~\ref{Prop: Red_W=Red_(W_g)}~(1), for every reflection $t\leq_{\RRR}a$ we have $\Mov(b)\subseteq V^a\subseteq V^t$. Taking orthogonal complements, $\Mov(t) \subseteq V^b$.  Repeating the argument with the roles of $a$ and $b$ reversed, we have for any $t' \leq_{\RRR} b$ that $\Mov(t') \subseteq V^a$.  This implies that all reflections $t_1,\ldots,t_k$ commute with all reflections $t_{k+1},\ldots,t_n$, and therefore that $W=W_1\times W_2$, which contradicts the irreducibility of $W$.

The only remaining point is to show that  \eqref{EQ: g=g1..gs} is the unique factorization satisfying the three conditions. For any decomposition $g=g'_1\cdots g'_r$ as in the statement, any reflection $t\leq_{\RRR} g'_i$ must belong to $W_g$ (by Proposition~\ref{Prop: Red_W=Red_(W_g)}), and so we have $g'_i\in W_g$.  Each $g'_i$ must belong to a single irreducible component of $W_g$ since otherwise, the decomposition would fail (3). Finally, if there is a component -- without loss of generality $W_1$ -- that contains more than one of the factors $g'_i$, then $g_1$ from \eqref{EQ: g=g1..gs} must equal their product; but this contradicts part (3) for the factorization \eqref{EQ: g=g1..gs}. This concludes the argument. 
\end{proof}

\begin{remark}\label{Rem: what is a cycle}

The concept of a ``generalized cycle" in Proposition~\ref{Prop: gob-cycle-decomp} is best captured by a \emph{pair} $(g,G)$ of an irreducible (well generated) reflection group $G$ and a quasi-Coxeter element $g$ of $G$. A priori, this suggests that the cycle decomposition should depend on the choice of ambient group.  For similar decompositions of non-quasi-Coxeter elements, this is indeed the case; for example, consider the  element $p:=[(12)(34);(0,1,0,1)]$, which belongs to both $W = G(2, 2, 4) \cong D_4$ and $W' = G(2, 1, 2) \times G(2, 1, 2) \cong B_2 \times B_2$.  In $W$, the decomposition of $p$ coming from the parabolic closure is just $p = p$ (the closure $W_p$ is the whole group), while in $W'$ the natural decomposition is
\[
p=[(12);(0,1,0,0)]\cdot [(34);(0,0,0,1)].
\]

However, the decompositions described in Proposition~\ref{Prop: gob-cycle-decomp} are hereditary in some settings. For example, if we have real reflection groups\footnote{It is unclear to what extent this holds in the complex case; see \S\ref{sec: gen_quest} for more on this question.} $W'\leq W$ with $w\in W'$ a parabolic quasi-Coxeter element of $W$, then $w$ is parabolic quasi-Coxeter also in $W'$, and its two generalized cycle decompositions (in $W$ and $W'$) are identical.  (This comes down to the transitivity of the Hurwitz action and the fact that reduced factorizations in $W'$ are also reduced in $W$.)
\end{remark}

\begin{remark}
We will not make use of the strong uniqueness properties of the generalized cycle decomposition in this paper or the conclusion \cite{DLM3}. For us, the most important consequence of Proposition~\ref{Prop: gob-cycle-decomp} is its last statement: that for parabolic quasi-Coxeter elements $g\in W$, the factors in this natural decomposition are all quasi-Coxeter elements of the irreducible components of $W_g$.
\end{remark}

\subsection{Parabolic quasi-Coxeter elements in the infinite family}
\label{subsec:pqC in the infinite family}

In \cite[Cor.~5.7]{LW}, the following characterization was given for the quasi-Coxeter elements in the well generated groups in the infinite family:\footnote{
When comparing the present work with \cite{LW}, one should be aware of an important difference in terminology: in \cite[Def.~5.1]{LW}, quasi-Coxeter elements are not required to have reflection length equal to the rank of $W$ (unlike our Definition~\ref{Defn: p-q-Cox elts}), but may also have larger reflection lengths.  The statement of Corollary~\ref{cor:pqC elements in m1n and mmn} has been adjusted to our more restrictive definition.}
in $G(m, 1, n)$, an element $w$ is quasi-Coxeter if and only if its underlying permutation is an $n$-cycle and its color generates $\ZZ/m\ZZ$, while in $G(m, m, n)$, an element is quasi-Coxeter if and only if its underlying permutation has exactly two cycles and their colors (which necessarily sum to $0$) generate $\ZZ/m\ZZ$.  Combining this result with Theorem~\ref{thm:parabolic subgroups of G(m, p, n)} and Proposition~\ref{Prop: gob-cycle-decomp} immediately yields the following characterization of parabolic quasi-Coxeter elements in $G(m, 1, n)$ and $G(m, m, n)$.

\begin{corollary}
\label{cor:pqC elements in m1n and mmn}
Suppose that $g$ is an element of $G(m, 1, n)$.  Then
\begin{enumerate}[(i)]
\item $g$ is parabolic quasi-Coxeter for a subgroup of type $\Symm_{\lambda_1} \times \cdots \times \Symm_{\lambda_k}$ if and only if $g$ has cycles of lengths $\lambda_1, \ldots, \lambda_k$, all of color~$0$, and
\item $g$ is parabolic quasi-Coxeter for a subgroup of type $G(m, 1, \lambda_0) \times \Symm_{\lambda_1} \times \cdots \times \Symm_{\lambda_k}$ if and only if $g$ has one cycle of length $\lambda_0$ whose color generates $\ZZ/m\ZZ$ and $k$ cycles of lengths $\lambda_1, \ldots, \lambda_k$ and color $0$. 
\end{enumerate}
Suppose instead that $g$ is an element of $G(m, m, n)$.  Then 
\begin{enumerate}[(i)]
\addtocounter{enumi}{2}
\item $g$ is parabolic quasi-Coxeter for a subgroup of type $\Symm_{\lambda_1} \times \cdots \times \Symm_{\lambda_k}$ if and only if $g$ has cycles of lengths $\lambda_1, \ldots, \lambda_k$, all of color~$0$, and
\item $g$ is parabolic quasi-Coxeter for a subgroup of type $G(m, m, \lambda_0) \times \Symm_{\lambda_1} \times \cdots \times \Symm_{\lambda_k}$ if and only if $g$ has two cycles whose colors generate $\ZZ/m\ZZ$ and sum to $0$ and whose lengths add to $\lambda_0$, and has cycles of lengths $\lambda_1, \ldots, \lambda_k$ of color $0$.
\end{enumerate}

In all cases except the last, the generalized cycles of $g$ are the cycles of $g$; in the case that $W_g \cong G(m, m, \lambda_0) \times \Symm_{\lambda_1} \times \cdots \times \lambda_k$, the two cycles of nonzero color in $g$ together form a generalized cycle.
\end{corollary}

\begin{remark}
Parabolic quasi-Coxeter classes make up a small percentage of all conjugacy classes in the combinatorial family.  In the exceptional types, by contrast, there are often more classes that are parabolic quasi-Coxeter than the other way around. In the following table we give the relevant numbers for the real case.  (These are easy to count by a computer, but for Weyl groups all necessary information is also available in the literature: the quasi-Coxeter classes are listed in Carter's work \cite{carter}, the parabolic subgroups up to conjugacy are given in \cite[Appendix~C]{OT-book}, and these may be combined using Proposition~\ref{Prop: gob-cycle-decomp}.)
\begin{table}[H]
\[
\begin{array}{|c|| c| c| c| c| c| c| |c|} 
 \hline
 W & H_3 & H_4 & F_4 & E_6 & E_7 & E_8  \\ 
 \hline
 \text{number of parabolic Coxeter classes} & 6 & 10  & 12 & 17 & 32 & 41 \\
 \hline
 \text{number of parabolic quasi-Coxeter classes} & 9 & 24 & 13 & 21 & 44 &  67 \\ 
 \hline
 \text{number of all conjugacy classes} & 10 & 34 & 25 & 25 & 60 & 112 \\
 \hline
\end{array}
\]
\caption{Comparison of types of conjugacy classes in exceptional real reflection groups}
\label{Table: pqCox classes ratios}
\end{table}

\end{remark}

\section{Characterization via relative reflection generating sets}
\label{Sec: relative gen sets}

In this section, we give a characterization of parabolic quasi-Coxeter elements in terms of the existence of certain sets of reflections that we call \emph{relative reflection generating sets}. These new objects are combinatorially appealing on their own (see \S \ref{sec: counting rrgs}) but they will also become a fundamental ingredient of the uniformly-stated formulas for $W$-Hurwitz numbers in the third instalment \cite{DLM3} of this series of papers.  Our definitions are analogous to the notion of good generation in Definition~\ref{Defn: good gen'g sets}, but \emph{relative} to a given element or subgroup. 

\begin{definition}[relative generating sets for reflection subgroups]
\label{Defn: relative generating set (W,W')}
Let $W$ be a well generated complex reflection group of rank $n$ and let $W'$ be a reflection subgroup of $W$ of rank $k$.  We say that a set $\{t_1, \ldots, t_{n - k}\}$ of $n - k$ reflections is a \defn{generating set for $W$ relative to $W'$} (or \defn{relative generating set} for short) if $W = \langle W', t_1, \ldots, t_{n - k}\rangle$. We denote by \defn{$\RGS(W,W')$} the set of all generating sets for $W$ relative to $W'$.
\end{definition}

\begin{example}
For any $W$, one has that $\RGS(W, \{\id\})$ is the set of good generating sets for $W$, while $\RGS(W, W) = \{\varnothing\}$.
\end{example}

\begin{definition}[relative generating sets for elements]
\label{Defn: relative generating set (W,g)}
Let $W$ be a well generated complex reflection group of rank $n$ and let $g$ be an element of $W$ of reflection length $\lR(g) = k$.  We say that a set $\{t_1, \ldots, t_{n - k}\}$ of $n - k$ reflections is a \defn{generating set for $W$ relative to $g$} (or \defn{relative generating set} for short) if there exists a reduced reflection factorization $g = t_{n - k + 1} \cdots t_{n}$ such that $\{t_1, \ldots, t_n\}$ is a generating set for $W$.  We denote by \defn{$\RGS(W, g)$} the set of all generating sets for $W$ relative to $g$.
\end{definition}

\begin{example}
For any $W$, one has that $\RGS(W, \id)$ is equal to the set of good generating sets for $W$.  If $g$ is a quasi-Coxeter element for $W$, then $\RGS(W, g) = \{\varnothing\}$.
\end{example}

\subsection{Compatibility with parabolic subgroups and parabolic quasi-Coxeter elements.}

Not every element or reflection subgroup of a reflection group has relative generating sets.  Indeed, the next several results (culminating with Corollary~\ref{Prop: RGS(W,W') exist iff parabolic} and Theorem~\ref{Prop: RGS(W,g) exist iff p-q-Cox}) show that relative generating sets exist precisely for parabolic subgroups and for parabolic quasi-Coxeter elements.

\begin{proposition} \label{Prop. extensions <W',t_i>=W}
Let $W$ be a well generated complex reflection group of rank $n$ and $W_X\leq W$ a parabolic subgroup of rank $k$. Then there exist reflections $\{t_1,\ldots,t_{n-k}\}$ in $W$ such that $\langle W_X,t_1,\ldots,t_{n-k}\rangle=W$.
\end{proposition}
We give two proofs of this result.  The first is computational.
\begin{proof}[First proof.]
In the real case, any parabolic subgroup is conjugate (by some element $g$) to a standard parabolic, generated by a subset $T$ of the simple reflections $S$.  Then the conjugates (by $g^{-1}$) of the simple reflections in $S \smallsetminus T$ are the required $t_i$. In the complex case, the result can be derived on a case-by-case basis using a similar argument, relying on the tables of \cite[App.~A]{broue-book} for generators and on \cite[Fact~1.7]{BMR-Hecke} for the relationship between arbitrary and ``standard" parabolic subgroups with respect to these generators.
\end{proof}
Since Proposition~\ref{Prop. extensions <W',t_i>=W} is a very natural statement, we also briefly sketch a case-free\footnote{
   This argument is case-free in the context of \emph{duality groups}, which are known (by a case-by-case argument  \cite[Thm.~5.5]{OrlikSolomon1980}) to coincide with well generated groups.
} topological argument.
\begin{proof}[Second proof.]
In \cite{bessis-zariski}, Bessis constructed geometric generators for every complex reflection group $W$, as follows.  The Shephard--Todd--Chevalley theorem \cite{ShephardTodd, Chevalley} identifies the space of orbits $V/W$ with a complex affine space $\CC^n$. When $W$ is well generated, there is a line $L_0$ through the origin that is the image in $V/W\cong\CC^n$ of the $e^{2\pi i/h}$-eigenspace of every Coxeter element. A generic affine line $L$, parallel to $L_0$, intersects the discriminant hypersurface $\mathcal{H}$ of $W$ at $n$ points.  Loops around these points define generators of the \defn{generalized braid group} $\defn{B(W)}:=\pi_1(\CC^n-\mathcal{H})$ that, under a canonical surjection $B(W) \twoheadrightarrow W$, map to a set of reflections that generates $W$.  

If we choose the line $L$ to be close to the stratum $X/W$ of the discriminant, the intersection $L\cap\mathcal{H}$ will consist of $\rank(W_X)=k$ points near $X/W$ and $n-k$ more away from it. This is because the multiplicity of $X/W$ as a stratum of $\mathcal{H}$ is equal to $\rank(W_X)$ \cite[Lem.~5.4]{Bessis_Annals} and because the direction of $L$ is always transversal to $\mathcal{H}$ \cite[Prop.~37]{theo_thesis}. These $n$ points define a generating set of reflections for $W$; moreover, the first $k$ of them generate $W_X$ because there is a local embedding of braid groups $B(W_X)\hookrightarrow B(W)$. The remaining $n-k$ reflections are exactly what is asked for in the statement. 
\end{proof}

The next two propositions provide complementary information to Proposition~\ref{Prop. extensions <W',t_i>=W}. Their original proofs by Taylor \cite{taylor} rely on case-by-case arguments, but Proposition~\ref{Prop: Taylor_subset_good} does have a case-free proof (easily deducible from \cite[Thm.~4.3.9]{Wegener_thesis}) for Weyl groups.

\begin{proposition}[{\cite[Thm.~4.2]{taylor}}]
\label{Prop: Taylor_subset_good}
Fix a well generated complex reflection group $W$ of rank $n$ and a good generating set $\mathcal{G}:=\{t_1,\ldots,t_n\}$ for $W$. Then any $k$-subset $\{t_{i_1},\ldots,t_{i_k}\}\subseteq \mathcal{G}$ is a good generating set for some parabolic subgroup $W'\leq W$ of rank $k$.
\end{proposition}

\begin{proposition}[{\cite[Thm.~4.1]{taylor}}]
\label{Prop: Taylor_K=<H,t>}
Assume that $W$ is a complex reflection group and that $K$ and $H$ are reflection subgroups of $W$ such that $K=\langle H,t\rangle$ for some reflection $t\in W$. If $K$ is parabolic and its rank is greater than that of $H$, then $H$ is also parabolic.
\end{proposition}

We are now ready for two corollaries that characterize the existence of relative generating sets. 

\begin{corollary}
\label{Prop: RGS(W,W') exist iff parabolic}
For a well generated complex reflection group $W$, a reflection subgroup $W'\leq W$ has relative generating sets if and only if it is parabolic.
\end{corollary}
\begin{proof}
If $W'\leq W$ is a parabolic subgroup, then a relative generating set is furnished by Proposition~\ref{Prop. extensions <W',t_i>=W}. 

Conversely, assume that $\rank(W)=n$, $\rank(W')=k$, and that $\{t_1,\ldots,t_{n-k}\}$ is a relative generating set for $W'$. Consider the chain of groups 
\[
W_i:=\langle W',t_1,\ldots,t_i\rangle \quad  \text{ for } i = 0, \ldots, n - k.
\]
Recall that the rank of a reflection group $G$ is the codimension of its fixed space $V^G$. Since $W_{i + 1}$ is generated over $W_i$ by a single reflection, we have $\rank(W_{i+1}) \leq \rank(W_i)+1$.  Since further $\rank(W_0) = \rank(W') = k$ and $\rank(W_{n - k}) = \rank(W) = n$, the equality $\rank(W_i)=k+i$ is forced for all $i$. Finally, since $W=W_{n-k}$ itself is parabolic, Proposition~\ref{Prop: Taylor_K=<H,t>} implies by induction that all $W_i$, including $W_0 = W'$, must be parabolic as well.
\end{proof}

\begin{theorem}[first characterization of parabolic quasi-Coxeter elements] \label{Prop: RGS(W,g) exist iff p-q-Cox}
Let $W$ be a well generated complex reflection group and $g \in W$.  Then there exists a generating set for $W$ relative to $g$ if and only if $g$ is a parabolic quasi-Coxeter element.
\end{theorem}
\begin{proof}
Let's first fix the notation, writing $n$ for the rank of $W$ and $k$ for the reflection length of $g$. For the forward direction, assume that $\{t_1,\ldots,t_n\}$ is a good generating set for $W$ such that the tuple $(t_1,\ldots,t_k)$ is a reduced  reflection factorization of $g$. By Proposition~\ref{Prop: Taylor_subset_good}, the set of reflections $\{t_1,\ldots,t_k\}$ forms a good generating set for a parabolic subgroup $W'\leq W$ of rank $k$. This means, by definition, that $g$ is parabolic quasi-Coxeter.

For the reverse direction, suppose $g$ is parabolic quasi-Coxeter and choose a reduced reflection factorization $g=t_1\cdots t_k$. By Proposition~\ref{Prop. strong=weak}, the reflections $t_1, \ldots, t_k$ generate the parabolic closure $W_g$ of $g$, and by Proposition~\ref{Prop: X=V^g and l_R(g)=codim(V^g)}, we have $\rank(W_g) = k$. Then Proposition~\ref{Prop. extensions <W',t_i>=W} guarantees the existence of a relative generating set for $W_g$, which is also a relative generating set for $g$.
\end{proof}

Finally, we show that the two notions of relative generating sets can be used interchangeably.

\begin{proposition}
\label{Prop: RGS(W,g)=RGS(W,W_g)}
Let $W$ be a well generated complex reflection group, $g\in W$ a parabolic quasi-Coxeter element, and $W_g$ the parabolic closure of $g$. Then we have a concordance of relative generating sets: \[
\RGS(W,g)=\RGS(W,W_g).
\]
\end{proposition}
\begin{proof}
This is immediate from the definition (Definition~\ref{Defn: p-q-Cox elts}) of parabolic quasi-Coxeter elements and Proposition~\ref{Prop. strong=weak}.
\end{proof}

\subsection{Structure of relative generating sets for the combinatorial family}
\label{sec:characterize RGS}

Next, we give a combinatorial description for relative generating sets when the group $W$ belongs to the infinite family $G(m, p, n)$.  We begin with some necessary terminology.

\begin{definition}
Given a partition $\Pi$ of the set $K$, we say that a graph $\Gamma$ with vertex set $K$ is a \defn{tree relative to $\Pi$} if no edge of $\Gamma$ connects two vertices in the same block of $\Pi$ and, furthermore, contracting each block of $\Pi$ to a single point leaves a tree.  We say that $\Gamma$ is a \defn{rooted tree relative to $\Pi$} if $\Gamma$ is the union of a tree with respect to $\Pi$ and a single loop (at any vertex of $K$).  Finally, we say that $\Gamma$ is a \defn{unicycle relative to $\Pi$} if $\Gamma$ is the union of a tree with respect to $\Pi$ and a single non-loop edge.
\end{definition}

The following properties are straightforward; we omit the proofs.
\begin{proposition}
\label{prop:relative trees}
Suppose that $\Pi$ is a partition of a set $K$, and $\Gamma$ is a graph with vertex set $K$.
\begin{enumerate}[(i)]
\item If $\Gamma$ is a tree relative to $\Pi$, then taking the union of $\Gamma$ with a spanning tree on each block of $\Pi$ gives a tree on $K$.
\item If $\Gamma$ is a rooted tree relative to $\Pi$, then contracting each component of $\Pi$ to a single vertex leaves a rooted tree.  Moreover, in this case, the union of $\Gamma$ with a spanning tree on each block of $\Pi$ gives a rooted tree on $K$.
\item If $\Gamma$ is a unicycle relative to $\Pi$, then contracting each component of $\Pi$ to a single vertex leaves either a unicycle (if no edge in $\Gamma$ connects two vertices in the same component of $\Pi$) or a rooted tree (otherwise).  Moreover, in this case, the union of $\Gamma$ with a spanning tree on each block of $\Pi$ gives a unicycle on $K$.
\end{enumerate}
\end{proposition}

\begin{definition}
Suppose that $W = G(m, 1, n)$ or $G(m, m, n)$ and that $g \in W$ is parabolic quasi-Coxeter.  Let \defn{$\Pi_g$} be the partition of $\{1, \ldots, n\}$ whose blocks are the support of the \emph{generalized} cycles of $g$ (as described in Corollary~\ref{cor:pqC elements in m1n and mmn}).  We say that $\Pi_g$ is the partition \defn{induced} by $g$.
\end{definition}

We are now prepared to describe the relative generating sets of parabolic quasi-Coxeter elements in $G(m, 1, n)$ and $G(m, m, n)$.

\begin{proposition}
\label{prop:RGS characterization}
Suppose that $W$ is either $G(m, 1, n)$ or $G(m, m, n)$ and that $g$ is a parabolic quasi-Coxeter element for $W$.  Let $\Pi_g$ be the partition of $\{1, \ldots, n\}$ induced by $g$.
\begin{enumerate}[(i)]
\item If either $W = G(m, 1, n)$ and $g$ is a quasi-Coxeter element for the subgroup $G(m, 1, \lambda_0) \times \Symm_{\lambda_1} \times \cdots \times \Symm_{\lambda_k}$, or $W = G(m, m, n)$ and $g$ is a quasi-Coxeter element for the subgroup $G(m, m, \lambda_0) \times \Symm_{\lambda_1} \times \cdots \times \Symm_{\lambda_k}$, then a set $S$  of reflections is a relative generating set for $g$ if and only if the graph associated to $S$ is a tree relative to $\Pi_g$.
\item If $W = G(m, 1, n)$ and $g$ is a quasi-Coxeter element for the subgroup $\Symm_{\lambda_1} \times \cdots \times \Symm_{\lambda_k}$, then a set $S$ of reflections is a relative generating set for $g$ if and only if the graph associated to $S$ is a rooted tree relative to $\Pi_g$ and the color of the unique diagonal reflection in $S$ generates $\ZZ/m\ZZ$.
\item  If $W = G(m, m, n)$ and $g$ is a quasi-Coxeter element for the subgroup $\Symm_{\lambda_1} \times \cdots \times \Symm_{\lambda_k}$, then a set of reflections is a relative generating set for $g$ if and only if the graph associated to $S$ is a unicycle relative to $\Pi_g$ and a certain color $c$ is a primitive generator for $\ZZ/m\ZZ$, where $c$ is defined as follows: if contracting the cycles of $g$ leaves a rooted tree, then $c$ is the color of the loop edge, whereas if contracting the cycles of $g$ leaves a unicycle, then $c$ is the value of the statistic $\delta$ on the cycle in the contracted graph.
\end{enumerate}
\end{proposition}

Note that in case (iii), the color condition is \emph{not} preserved under conjugacy; therefore, if one wishes to consider relative generating sets for parabolic subgroups of $G(m, m, n)$ conjugate in $G(m, 1, n)$ to $\Symm_{\lambda_1} \times \cdots \times \Symm_{\lambda_k}$, it is necessary to take the conjugation into account when applying Proposition~\ref{prop:RGS characterization}.

\begin{proof}
In all three cases, a simple calculation (using the facts $\rank(\Symm_i) = i - 1$ and $\rank(G(m, 1, j)) = \rank(G(m, m, j)) = j$ for $m > 1$) establishes that $\rank(W_g) = n - k$, and therefore that each relative generating set must have exactly $k$ reflections.
We consider the three parts in order.  

In case (i), consider a set $S$ of $k$ reflections, and let $\Gamma$ be the associated graph.  Contracting the generalized cycles of $g$ to a single point leaves a graph on $k + 1$ vertices with $k$ edges.  It is a standard result from graph theory (e.g., \cite[Lem.~5.3]{Bona}) that such a graph is either a tree or is disconnected.  If the graph is disconnected, then both the group $W_g$ and the reflections in $S$ respect the partition of $\{1, \ldots, n\}$ into components.  But the group $W$ does not respect any nontrivial partition of $\{1, \ldots, n\}$.  Thus, it is necessary that $S$ be a tree relative to $\Pi_g$.

Conversely, suppose that $S$ is a tree relative to $\Pi_g$; we must show that it is a relative generating set.  By Proposition~\ref{Prop: Red_W=Red_(W_g)}, Lemma~\ref{lem: good generating set G(m,1,n) and G(m,m,n)}, and the fact that the good generating sets for the symmetric group correspond to trees, we have that any reduced factorization of $g$ corresponds to a forest whose components are the blocks of $\Pi_g$, together with an additional edge in the first ($\lambda_0$) component.  Taking the union of such a graph with a tree relative to $\Pi_g$ produces a tree together with an additional edge.  Moreover, the unique graph cycle in this graph (which is a loop if $W = G(m, 1, n)$ and a cycle with at least two vertices if $W = G(m, m, n)$) belongs entirely to the reduced factorization.  Because that component of the graph corresponds to a generating set for either $G(m, 1, \lambda_0)$ or $G(m, m, \lambda_0)$ (whichever is relevant), the graph cycle satisfies the color condition of Lemma~\ref{lem: good generating set G(m,1,n) and G(m,m,n)}.  Thus, the union meets the hypotheses of Lemma~\ref{lem: good generating set G(m,1,n) and G(m,m,n)} and so generates the full group $W$. 

We omit the proof of (ii), which is similar to (but less technical than) case (iii).

In case (iii), consider a set $S$ of $k$ reflections, and let $\Gamma$ be the associated graph and $\overline{\Gamma}$ be the result of contracting $\Gamma$ so that each of the cycles of $g$ contracts to a single point.  As in case (i), in order for $S$ to be a relative generating set, it is necessary that $\overline{\Gamma}$ be connected.  Assume that $S$ meets this condition.  The contracted graph $\overline{\Gamma}$ has $k$ vertices and $k$ edges, so it is either a rooted tree or a unicycle.  Fix a reflection that corresponds to an edge in the unique graph cycle in $\overline{\Gamma}$ (the loop, if $\overline{\Gamma}$ is a rooted tree), and remove it.  By definition, the remaining reflections in $S$ correspond to a tree relative to the cycle partition $\Pi_g$ induced by $g$.  Since in this case $W = G(m, m, n)$, the removed reflection is transposition-like, so does not correspond to a loop in $\Gamma$.  Therefore, by definition, $\Gamma$ is a unicycle relative to $\Pi_g$.

We now consider which relative unicycles $\Gamma$ with respect to $\Pi_g$ are actually relative generating sets.  Since $W_g \subseteq \Symm_n$, we have by Proposition~\ref{Prop: Red_W=Red_(W_g)} that every reduced $W$-factorization of $g$ consists of a set $T$ of true transpositions (i.e., reflections of the form $(i j) = [(i j); 0]$) whose associated graph is a forest with components $\Pi_g$.  Taking the union of this forest associated with $T$ and a relative unicycle $\Gamma$ produces a unicycle.  By Lemma~\ref{lem: good generating set G(m,1,n) and G(m,m,n)}, the set $S\cup T$ of corresponding reflections is a good generating set for $W$ if and only if $\delta(S \cup T)$ is a primitive generator for $\ZZ/m\ZZ$.

Next, observe that the contraction of $\Gamma$ that produces $\overline{\Gamma}$ (collapsing each cycle of $g$ to a single point) may be performed step-wise, starting with the larger graph associated to $S \cup T$ and successively contracting the edges that correspond to $T$ (since these form a spanning forest with the correct components).  Moreover, it follows immediately from the definition of $\delta$ that if a set $S'$ of reflections corresponds to a unicycle, then contracting a unicycle along an edge $[(ij); 0]$ produces a new unicycle with the same $\delta$-value, unless the unique graph cycle is a $2$-cycle and containing the edge to be contracted.  If the second case never arises while contracting away the edges that correspond to $T$, then we have that $\delta(S \cup T)$ is a primitive generator for $\ZZ/m\ZZ$ (and so $S$ is a relative generating set) if and only if the value of $\delta$ on the contracted unicycle is a primitive generator, as needed; if the second case does arise, then the cycle in the graph associated to $S \cup T$ contains only a single edge $[(ij); c]$ from $S$, whose image under contraction is a loop, so $\delta(S \cup T) = c$, and so $S$ is a relative generating set if and only if the color $c$ of the loop in the contraction is a primitive generator, as needed.  This completes the proof.
\end{proof}

\section{Characterization via full reflection factorizations}
\label{Sec: full factorizations and p-q-Cox}

In this section we give two more characterizations of parabolic quasi-Coxeter elements $g$.  One of these (Thm.~\ref{Cor: p-q-cox means below q-cox}) is in terms of the absolute order $\leq_{\RRR}$.  The other (Thm.~\ref{Prop: g p-q-Cox iff ltr(g)=2n-lR(g)}) is given in terms of the \emph{full reflection length} of $g$, whose definition we recall now.  As in the first part \cite{DLM1} of this series, we say that a reflection factorization $t_1 \cdots t_k = g$ of an element $g \in W$ is \defn{full} (relative to $W$) if the factors generate the full group, i.e., if $W = \langle t_1, \ldots, t_k \rangle$.  
The \defn{full reflection length} \defn{$\ltr(g)$} of $g$ is the minimum length of a full reflection factorization:
\[
\ltr(g) := \min \Big\{ k : \exists t_1, \ldots, t_k \in \RRR \text{ such that } t_1 \cdots t_k = g \text{ and } \langle t_1, \ldots, t_k \rangle = W \Big\}.
\]

In Part I, we gave the following formula for the full reflection length of an arbitrary element $g$ in the group $G(m, p, n)$. Notice that the \emph{cycles} here are the cycles of the underlying permutation of $g$ and should not be confused with the \emph{generalized cycles} of \S \ref{sec: cycle decomposition}.
\begin{proposition}[{\cite[Cor.~5.2]{DLM1}}]
\label{cor:full length}
Let $W = G(m, p, n)$.
For an element $g\in W$ with $k$ cycles, of colors $a_1, \ldots, a_k$, let $d = \gcd(a_1, \ldots, a_k, p)$. If $m=p$, we have 
    \[
\ltr(g) = \begin{cases}
    n + k - 2, & \text{ if } d = 1 \\
    n + k , & \text{ if } d \neq 1,
    \end{cases}
    \]
while if $m\neq p$, we have
    \[
    \ltr(g) = \begin{cases}
    n + k - 1, & \text{ if } \gcd(\wt(g),m) = p \text{ and } d = 1 \\
    n + k, & \text{ if } \gcd(\wt(g),m)\neq p \text{ and }  d = 1\\
    n + k + 1, & \text{ if } \gcd(\wt(g),m) = p \text{ and }  d \neq 1\\
    n + k + 2, & \text{ if } \gcd(\wt(g),m) \neq p \text{ and }  d \neq 1.
    \end{cases}
    \]
\end{proposition}

We also need the following formula for the full reflection length of the identity in any complex reflection group.

\begin{lemma}
\label{lem:identity full reflection length}
In any complex reflection group $W$, the full reflection length $\ltr(\id)$ of the identity element is equal to twice the size of a minimum reflection generating set for $W$.
\end{lemma}
\begin{proof}
One inequality is straightforward: if $\{t_1, \ldots, t_k\}$ is a minimum reflection generating set for $W$ then 
\begin{equation}
\label{eq:dumb factorizations of id}
\id = t_1 \cdot t_1^{-1} \cdot t_2 \cdot t_2^{-1} \cdots t_k \cdot t_k^{-1}
\end{equation}
is a reflection factorization of $\id$ that generates $W$.  Therefore, it is left to show that there is no shorter reflection factorization of $\id$ that generates $W$.

Both full reflection length and the size of a minimum reflection generating set are obviously additive over direct products.  Therefore, the claim is valid for all reflection groups if and only if it is valid for irreducible groups, so it suffices to consider irreducible $W$.  We proceed in various cases.

First, suppose that $W$ is real.  By Lemma~\ref{lem: LR lemma}, every reflection factorization of the identity is Hurwitz-equivalent to a factorization of the form \eqref{eq:dumb factorizations of id}.  Such a factorization is full if and only if the set $\{t_1, \ldots, t_k\}$ is a generating set for $W$.  Since the Hurwitz action preserves the group generated by the factors, we conclude that in real reflection groups, the full reflection length of the identity is at least twice the minimum size of a reflection generating set.

Second, suppose that $W = G(m, m, n)$ for $m > 1$.  Taking $g = \id$ in Proposition~\ref{cor:full length}, we have $k = n$, $a_1 = \ldots = a_n = 0$, and $d = m \neq 1$, and therefore $\ltr(\id) = 2n$ in this case, as claimed.  

Third, suppose that $W = G(m, p, n)$ for some $p < m$.  Taking $g = \id$ in Proposition~\ref{cor:full length}, we have $k = n$, $a_1 = \ldots = a_n = 0$, $d = p$, and $\gcd(\wt(g), m) = m \neq p$.  Therefore, if $p = 1$, we have $d = 1$ and consequently $\ltr(\id) = 2n$, as claimed.  On the other hand, if $p > 1$ then $d \neq 1$ and so $\ltr(\id) = 2n + 2$, as claimed.

Finally, for the non-real exceptional groups, the claim has been verified by computer calculations following the representation-theoretic approach described in \cite[Rem.~3.2]{DLM1}.
\end{proof}

\begin{remark}
\label{Rem: uniform proof of lfull(id)}
For real reflection groups, the proof of Lemma~\ref{lem:identity full reflection length} is case-free after Remark~\ref{Rem: WY lemma}.
\end{remark}

We give now several statements that are of use at various points in the rest of the paper; they are all corollaries of Lemma~\ref{lem:identity full reflection length} and previous discussions. 

\begin{corollary}
\label{Corol: lr(g)+ltr(g)>=2m}
If $W$ is a complex reflection group and $r$ denotes the minimum size of a reflection generating set for $W$, then for any element $g\in W$, we have 
\[
\lR(g)+\ltr(g)\geq 2r.
\]
\end{corollary}
\begin{proof}
Consider two reflection factorizations of $g$: one that is reduced $ t_1\cdots t_k=g$ (with $k=\lR(g)$), and one that is full $ t'_1\cdots t'_l=g$ (with $l=\ltr(g)$). By combining them, we get a full factorization
\[
\id =  t'_1\cdots t'_l\cdot t_k^{-1}\cdots t_1^{-1}
\]
of the identity, having length $l+k=\lR(g)+\ltr(g)$. It follows from Lemma~\ref{lem:identity full reflection length} that $l + k \geq 2r$.
\end{proof}

\begin{corollary}
\label{Cor: prod_t_i is q-Cox}
In a well generated complex reflection group $W$ of rank $n$ with a given good generating set $\{ t_1,\ldots, t_n\}$, the product of the generators $ t_i$ in any order is a quasi-Coxeter element.
\end{corollary}
\begin{proof}
Let $g$ be the product of the reflections $ t_i$ in some order. This factorization is clearly full, which implies that $\lR(g) \leq \ltr(g) \leq n$. Since $W$ is well generated, we have by Corollary~\ref{Corol: lr(g)+ltr(g)>=2m} that $\lR(g) + \ltr(g) \geq 2n$.  Combining the inequalities gives $\lR(g) = \ltr(g) = n$, and so the given factorization of $g$ is reduced.  Then it follows from Definition~\ref{Defn: p-q-Cox elts} that $g$ is quasi-Coxeter.
\end{proof}

\begin{remark}
In the infinite family, Corollary~\ref{Cor: prod_t_i is q-Cox} could be alternatively checked by comparing the descriptions of the good generating sets in Lemma~\ref{lem: good generating set G(m,1,n) and G(m,m,n)} with the description of the quasi-Coxeter elements in Section~\ref{subsec:pqC in the infinite family}, using the observation of Remark~\ref{rem: Shi's delta} in the case of $G(m, m, n)$.
\end{remark}

One characterization of parabolic \emph{Coxeter} elements is that they are precisely those which are below some Coxeter element in the absolute order $\leq_{\RRR}$ (for instance, see \cite[Lem.~1.4.3]{bessis-dual-braid} for the real case).
The following statement generalizes this property to the \emph{quasi-Coxeter} setting, extending \cite[Cor.~6.11]{BGRW} (which covers the case of real $W$).

\begin{theorem}[second characterization of parabolic quasi-Coxeter elements]
\label{Cor: p-q-cox means below q-cox}
Let $W$ be a well generated complex reflection group and let $g\in W$. Then $g$ is a parabolic quasi-Coxeter element if and only if there exists a quasi-Coxeter element $w\in W$ such that $g\leq_{\RRR}w$.
\end{theorem}
\begin{proof}
Let $n = \rank(W)$.  First, choose a parabolic quasi-Coxeter element $g$ and a reduced reflection factorization $g=t_1\cdots t_k$ (with $k=\lR(g)$).  By Theorem~\ref{Prop: RGS(W,g) exist iff p-q-Cox}, there exists a relative generating set $\{t_{k+1},\ldots,t_n\}$ with respect to $g$, that is, one which satisfies $\langle t_1,\ldots,t_n\rangle=W$. By Corollary~\ref{Cor: prod_t_i is q-Cox}, the product $w:=t_1\cdots t_n$ must be a quasi-Coxeter element; in particular, $\lR(w)=n$. We have now by definition that $g\leq_{\RRR}w$, and the forward direction is proven.

For the reverse direction, suppose that $w\in W$ is a quasi-Coxeter element for which $g\leq_{\RRR}w$.  By definition of $\leq_{\RRR}$, there is a reduced reflection factorization $w=t_1\cdots t_n$ that starts with a reduced factorization of $g$, that is, such that $g=t_1\cdots t_k$ for $k=\lR(g)$. Since $w$ is quasi-Coxeter, the set $\{t_1,\ldots, t_n\}$ is a good generating set for $W$  (Proposition~\ref{Prop. strong=weak}), and then, by Proposition~\ref{Prop: Taylor_subset_good}, the set $\{t_1,\ldots,t_k\}$ must be a good generating set for some parabolic subgroup of $W$. By definition, this means that $g$ is a parabolic quasi-Coxeter element.
\end{proof}

We now extend Lemma~\ref{lem:identity full reflection length} to give another characterization of parabolic quasi-Coxeter elements in well generated groups $W$. 

\begin{theorem}[third characterization of parabolic quasi-Coxeter elements]
\label{Prop: g p-q-Cox iff ltr(g)=2n-lR(g)}
In a well generated complex reflection group $W$ of rank $n$, an element $g\in W$ is a parabolic quasi-Coxeter element if and only if 
\[
\ltr(g)=2n-\lR(g).
\]
\end{theorem}
\begin{proof}
Let us start with the forward implication and assume that $g\in W$ is a parabolic quasi-Coxeter element.
Consider a reduced reflection factorization $ t_1\cdots t_k=g$ (so $k=\lR(g)$).  By Proposition~\ref{Prop. strong=weak}, the $t_i$ generate the parabolic closure $W_g$ of $g$ and $\rank(V^g) = k$. By Proposition~\ref{Prop. extensions <W',t_i>=W}, there are reflections $ t_{k+1},\ldots, t_n$ such that $\langle W_g, t_{k+1},\ldots, t_n\rangle=W$. Therefore
\[
g= t_1\cdots t_k\cdot t_{k+1} \cdot  t_{k+1}^{-1}\cdots t_n \cdot  t_n^{-1}
\]
is a \emph{full} factorization of $g$ in $W$, of length $2n - k$, which in turn implies that $\ltr(g)\leq 2n-\lR(g)$. By Corollary~\ref{Corol: lr(g)+ltr(g)>=2m}, actually $\ltr(g) = 2n-\lR(g)$, as needed.

Conversely, suppose that $W$ is a well generated complex reflection group and the element $g$ in $W$ satisfies $\lR(g) + \ltr(g) = 2n$.  Since rank, reflection length, and full reflection length are all additive over direct products, using Corollary~\ref{Corol: lr(g)+ltr(g)>=2m}, we have that equality holds for each irreducible component of $W$.  Since (by Corollary~\ref{cor: p-q-Cox in W_X implies in W}) the product of parabolic quasi-Coxeter elements from the irreducible factors is a parabolic quasi-Coxeter element for the direct product, it suffices to consider irreducible $W$.  We proceed in various cases.

First, suppose that $W$ is real of rank $n$ and let $k = \lR(g)$.  By hypothesis, $\ltr(g) = 2n - k$.  Consider a minimum-length full reflection factorization $\ttt = (t_1, \ldots, t_{2n - k})$ of $g$.  By Lemma~\ref{lem: LR lemma}, $\ttt$ has in its Hurwitz orbit a factorization of the form $\ttt' = (t'_1, t'_1, \ldots, t'_{n - k}, t'_{n - k}, t'_{n - k + 1}, \ldots, t'_n)$ such that $(t'_{n - k + 1}, \ldots, t'_n)$ is a reduced reflection factorization of $g$.  Since $\ttt$ is full, $\ttt'$ is also full, and therefore $\{ t'_1, \ldots, t'_{n - k}\}$ is a relative generating set for $g$.  By Theorem~\ref{Prop: RGS(W,g) exist iff p-q-Cox}, it follows that $g$ is parabolic quasi-Coxeter.

Second, suppose that $W = G(m, 1, n)$ for some $m > 1$ and that $g$ has $k$ cycles, of which $j$ have nonzero color.  By Theorem~\ref{thm:reflection length in G(m, 1, n) and G(m, p, n)}~(i), we have $\lR(g) = n - k + j$.  By Proposition~\ref{cor:full length}, we have $\ltr(g) \geq n + k - 1$, and consequently $j \leq 1$.  If $j = 0$ then all cycles of $g$ have color $0$, and therefore by Proposition~\ref{cor:pqC elements in m1n and mmn}~(i) $g$ is parabolic quasi-Coxeter for a conjugate of a subgroup of type $\Symm_{\lambda_1} \times \cdots \times \Symm_{\lambda_k}$ for some partition $\lambda$ of $n$.  On the other hand, if $j = 1$ then $\ltr(g) = n + k - 1$.  Therefore, by Proposition~\ref{cor:full length}, we have that $\gcd(\wt(g), m) = 1$, and so the color of $g$ is primitive modulo $m$.  Since $g$ has only one cycle of nonzero color, the color of that cycle generates $\ZZ/m\ZZ$ and we have by Proposition~\ref{cor:pqC elements in m1n and mmn}~(ii) that $g$ is parabolic quasi-Coxeter for a conjugate of a subgroup of type $G(m, 1, \lambda_1) \times \Symm_{\lambda_2} \times \cdots \times \Symm_{\lambda_k}$ for some partition $\lambda$ of $n$.

Third, suppose that $W = G(m, m, n)$ for some $m > 1$ and that $g$ has $k$ cycles.  By Theorem~\ref{thm:reflection length in G(m, 1, n) and G(m, p, n)}, we have $\lR(g) = n + k - 2v_m(g)$, where $v_m(g)$ is the largest number of parts into which one can partition the cycles of $g$ so that all parts have color $0$.  By Proposition~\ref{cor:full length}, we have $\ltr(g) \geq n + k - 2$.  Consequently, $v_m(g) \geq k - 1$.  If $v_m(g) = k$ then the associated partition is the trivial partition, with every cycle in its own part, and so all cycles are of color $0$.  Therefore, by Proposition~\ref{cor:pqC elements in m1n and mmn}~(iii), $g$ is parabolic quasi-Coxeter a subgroup of type $\Symm_{\lambda_1} \times \cdots \times \Symm_{\lambda_k}$ for some partition $\lambda$ of $n$. On the other hand, if $v_m(g) = k - 1$, then the partition must have one part of size $2$ (with two cycles of nonzero colors that sum to $0$) and $k - 2$ parts containing a single cycle (necessarily of color $0$).  Moreover, in this case we have $\ltr(g) = n + k - 2$, so $d = 1$, and consequently the nontrivial cycle colors must be primitive modulo $m$.  Therefore, by Proposition~\ref{cor:pqC elements in m1n and mmn}~(iv), $g$ is parabolic quasi-Coxeter a subgroup of type $G(m, m, \lambda_1) \times \Symm_{\lambda_2} \times \cdots \times \Symm_{\lambda_{k - 2}}$ for some partition $\lambda$ of $n$.

Finally, suppose that $W$ is of non-real exceptional type.  In these cases, the statement has been confirmed by computer calculations, as follows.  On one hand, the property of being parabolic quasi-Coxeter can be verified most easily by the characterization of Theorem~\ref{Prop: RGS(W,g) exist iff p-q-Cox}: for each conjugacy class representative $g$, we pair every $(n-\lR(g))$-element subset of $\RRR$ with the reflections in a fixed reduced factorization of $g$, and check whether the whole collection ever generates $W$. On the other hand, the full reflection length can be computed by representation-theoretic techniques, as in \cite[Rem.~3.2]{DLM1}.
\end{proof}

\begin{remark}
\label{Rem: uniform proof lfull(g)}
For real reflection groups, the proof of Theorem~\ref{Prop: g p-q-Cox iff ltr(g)=2n-lR(g)} is case-free after Remark~\ref{Rem: WY lemma}.
\end{remark}

The previous theorem is the final component of our main structural 
theorem announced in the introduction. Its proof is a direct combination of Theorems~\ref{Prop: RGS(W,g) exist iff p-q-Cox}, \ref{Cor: p-q-cox means below q-cox}, and~\ref{Prop: g p-q-Cox iff ltr(g)=2n-lR(g)}.

\pqCoxChar

\section{Reduced reflection factorizations of parabolic quasi-Coxeter elements}
\label{sec: Fred numbers of qCox}

In this section, we study the number $\Fred(g)$ of reduced reflection factorizations of a quasi-Coxeter element $g$ in a well generated complex reflection group $W$.  It is well known \cite{Denes, looijenga, J88} that, in the case of a long cycle $c$ in the symmetric group $\Symm_n$, the number of reduced reflection factorizations has a simple product formula: $\Fred[\Symm_n](c) = n^{n - 2}$.  As we see below, it is also the case for every quasi-Coxeter element $c$ in every well generated group $W$ that the number $\Fred(c)$ factors as a product of small primes.  Our discussion proceeds in several stages: first, in \S\ref{ssec: Fred of qcox_Coxeter case}, we review the case of Coxeter elements, including the connection with the \emph{discriminant hypersurface} of $W$ and the \emph{Lyashko--Looijenga morphism}.  Second, in \S\ref{Sec: Frobenius mflds}, we consider the case of a general quasi-Coxeter element $g$ in a \emph{real} reflection group $W$.  We see in Table~\ref{Table: Fred(g)} that in all cases, $\Fred(g)$ is a product of small primes.  While we do not have a complete understanding this numerological phenomenon, we present a conjectural geometric explanation (see Conjecture~\ref{conj alg frob mfld}) that generalizes the situation of Coxeter elements discussed in \S\ref{ssec: Fred of qcox_Coxeter case}.  In \S\ref{ssec: Fred of qcox_infinite family}, we give explicit formulas for $\Fred(g)$ when $g$ is a quasi-Coxeter element and $W$ belongs to the infinite family of complex reflection groups (Corollary~\ref{cor: Fred for infinite families}), and discuss the case of the exceptional groups (Remark~\ref{remark:Freds for exceptionals}).  Finally, in \S\ref{ssec:reduction from pqc to qc}, we show that $\Fred(g)$ is a product of small primes whenever $g$ is a \emph{parabolic} quasi-Coxeter element, by reducing to the case of quasi-Coxeter elements.

Apart from being of substantial interest on its own, the study of the numbers $\Fred(g)$ of reduced reflection factorizations of parabolic quasi-Coxeter elements $g$ will also play an important role for our main result in \cite{DLM3}, which gives a formula for the numbers $\Ftr_W(g)$ of minimum-length \emph{full} reflection factorizations of $g$. The connection is two-fold: first, for a quasi-Coxeter element $g \in W$, we have that $\Ftr_W(g)=\Fred(g)$ (see Proposition~\ref{Prop: Fred=Ffull for qCox} below; in the case of the symmetric group, this is a direct computation using the Hurwitz formula \eqref{eq: formula H_0} that $H_0((n)) = n^{n - 2}$).  Second, we will show in \cite{DLM3} that the numbers $\Fred(g)$ are always (i.e., for any \emph{parabolic} quasi-Coxeter element $g\in W$) factors in the formula for $\Ftr_W(g)$ (see also \S\ref{sec:intro:towards_uniform_formulas}). We end this introductory discussion with the proof that $\Ftr_W(g)=\Fred(g)$ for quasi-Coxeter elements $g$.

\begin{proposition}\label{Prop: Fred=Ffull for qCox}
For a quasi-Coxeter element $g$ of a well generated group $W$, the number of full reflection factorizations of $g$ equals the number of reduced reflection factorizations of $g$.
\end{proposition}
\begin{proof}
Let $g$ be a quasi-Coxeter element of $W$.  By Definition~\ref{Defn: p-q-Cox elts}, $g$ has a reduced reflection factorization that is also full. Then Proposition~\ref{Prop. strong=weak} implies that all reduced factorizations of $g$ must be full for the parabolic closure $W_g$, which is equal to $W$ by Proposition~\ref{Prop: X=V^g and l_R(g)=codim(V^g)}.
\end{proof}

\subsection{Reflection discriminants and reduced factorizations of Coxeter elements.}
\label{ssec: Fred of qcox_Coxeter case}

We deal briefly in this section with the case of a Coxeter element $c$ in an irreducible well generated group $W$. For such elements, the numbers $\Fred(c)$ have a well known product structure: if $h$ is the Coxeter number of $W$ and $n$ is its rank, we have that
\begin{equation}
\Fred(c)=\dfrac{h^n \cdot n!}{\#W}.\label{eq: deligne-arnold-bessis formula}    
\end{equation}
This formula has many proofs \cite{chapoton,chapuy_Stump,michel-deligne-lusztig-derivation,chapuy_theo_lapl}, each indicating connections to different areas and each explaining the product structure in a different way. In the next few paragraphs, we review a geometric approach, which we extend in the next section (\S\ref{Sec: Frobenius mflds}) to the case of quasi-Coxeter elements. For any terminology that is not explained, or for proofs of statements that are only claimed here, consult the common references \cite{Kane,LehrerTaylor,broue-book}.

Recall first that, as in the proof of Proposition~\ref{Prop. extensions <W',t_i>=W}, for a complex reflection group $W\leq \GL(V)$ acting on some space $V\cong \CC^n$, the Shephard--Todd--Chevalley theorem identifies the quotient space $V/W$ with the complex affine space $\CC^n$ whose coordinates are given by the fundamental invariants $\bm f:=(f_i)_{i=1\ldots n}$ of $W$. The $f_i$ are homogeneous polynomials and we order them according to their degrees $d_i:=\deg f_i$, so that $d_i\leq d_{i+1}$. The \defn{discriminant hypersurface $\mathcal{H}$} of $W$ is defined as the quotient variety $\mathcal{H} := \mathcal{A}_W/W \subset V/W\cong \CC^n$ of the reflection arrangement $\mathcal{A}_W$. 

The highest-degree invariant $f_n$ plays a special role in this setting.  The hypersurface $\mathcal{H}$ is the zero set of a polynomial $\Delta(W;\bm f)$ in $\CC[f_1,\ldots, f_n]$ that is monic and of degree $n$ in $f_n$. In fact, there is always a suitable choice of fundamental invariants such that the discriminant polynomial can be written as
\begin{equation}
    \Delta(W;\bm f)= f_n^n+a_2(\bm f)\cdot f_n^{n-2}+\cdots+a_n(\bm f),\label{eq: Delta(W;f)}
\end{equation}
where the $a_i$ are weighted-homogeneous polynomials in $f_1,\ldots,f_{n-1}$ of weighted degrees $\op{wt}(a_i)=h\cdot i$ (where the weight of $f_i$ is $d_i$). Moreover, the \emph{braid monodromy} of the polynomial $\Delta(W;\bm f)$ in the $f_n$-direction encodes reduced factorizations of Coxeter elements, as we explain next.

The braid monodromy of an algebraic function such as $\Delta(W;\bm f)$ (viewed as a polynomial in $f_n$ with coefficients $a_i$) is a refinement of its usual monodromy group: it keeps track of \emph{how} the function values move around each other when we vary the coefficients $a_i$, as opposed to just recording their final permutation (see \cite{arnold_braids}). The braid monodromy of an algebraic function is encoded via its \emph{coefficient map} (as in \cite{hansen}, \cite[\S2]{cohen-suciu}), which in the case of the discriminant polynomial $\Delta(W;\bm f)$ is known also as the \defn{Lyashko--Looijenga morphism $LL$} (see \cite{looijenga,Bessis_Annals}):
\begin{equation}
   \CC^{n-1}\ni \big(f_1,\ldots,f_{n-1}\big)\overset{LL}{\longmapsto}\big(a_2(f_1,\ldots,f_{n-1}),\ldots,a_n(f_1,\ldots,f_{n-1})\big)\in\CC^{n-1}.\label{eq: LL map}
\end{equation}
The $LL$ map is weighted-homogeneous and its degree may be calculated by a version of Bezout's theorem (for instance \cite[Thm.~5.1.5]{LandoZvonkin}) in terms of the weights of the coordinates in its image ($\op{wt}(a_i)=h\cdot i$) and its domain ($\op{wt}(f_i)=d_i$):
\begin{equation}
    \deg(LL)=\dfrac{\prod_{i=2}^n\op{wt}(a_i)}{\prod_{i=1}^{n-1}\op{wt}(f_i)}=\dfrac{h^{n-1} \cdot n!}{d_1\cdots d_{n-1}}=\dfrac{h^n \cdot n!}{\#W},\label{eq: deg LL calculation}
\end{equation}
where the last equality is due to the facts that $d_n=h$ for well generated $W$ and $\#W=d_1\cdots d_n$. 

A natural geometric construction \cite[\S7]{Bessis_Annals}, which goes back to Looijenga's paper \cite[\S3]{looijenga} but is also part of the general theory of braid monodromies \cite[\S3]{ignacio_braid}, associates a reduced factorization of the Coxeter element to each point in a generic fiber of the $LL$ map. Bessis was able to rely on certain properties of the $LL$ map (essentially, that it is a finite morphism) and on  the transitivity of the Hurwitz action on $\op{Red}_W(c)$ to prove the following lemma. His proof is conceptual and case-free (assuming the Hurwitz transitivity) and, in combination with Equation \eqref{eq: deg LL calculation}, it gives an a priori justification of the numerological properties of the numbers $\Fred(c)$ (that they are products of small primes) discussed at the beginning of \S\ref{sec: Fred numbers of qCox}.

\begin{lemma}\label{lemma deg LL k-1 covering}
For a well-generated complex reflection group $W$ and a Coxeter element $c\in W$, there is for some $k$ a $k$-to-$1$ correspondence between elements in a generic fiber of the $LL$ map and the set $\op{Red}_W(c)$ of reduced reflection factorizations of $c$. In particular, we have
\[
\deg(LL)=k\cdot \Fred(c)
\]
for some positive integer $k$.
\end{lemma}

\begin{remark}
Comparing the formulas in \eqref{eq: deligne-arnold-bessis formula} and \eqref{eq: deg LL calculation}, it is evident that $k=1$ in the Lemma~\ref{lemma deg LL k-1 covering}.  
Many of the constructions in Bessis' work \cite{Bessis_Annals} rely on the numerological coincidence of the formulas \eqref{eq: deligne-arnold-bessis formula} and \eqref{eq: deg LL calculation}; in particular, the proof of the dual braid presentation of the generalized braid group $B(W)$ (see the exposition in \cite{chapuy_theo_lapl} for more). For those applications, the statement of Lemma~\ref{lemma deg LL k-1 covering} is not sufficient; one really needs to know that $k=1$. A conceptual proof of this, relying on the geometry of the $LL$ map, is highly desirable but currently seems out of reach. A very promising approach is in the work of Hertling and Roucairol, who do the case of the simply laced Weyl groups (see the proof of \cite[Thm.~7.1]{hertling_roucairol}).
\end{remark}

\subsection{Frobenius manifolds and reduced factorizations of quasi-Coxeter elements.}
\label{Sec: Frobenius mflds}

That the counts $\Fred(g)$ are always products of small prime numbers when $g$ is a quasi-Coxeter element (not just a Coxeter element) 
was first observed by Kluitmann and Voigt in the case of the simply laced Weyl groups \cite{KV}. It was rediscovered by Christian Stump, who further confirmed it in the broader set of real reflection groups $W$ 
(personal communication) and asked for an explanation or even just a uniformly stated formula for the numbers $\Fred(g)$. 

We give in Table~\ref{Table: Fred(g)} the numbers $\Fred(g)$ for every quasi-Coxeter element $g$ and irreducible real reflection group $W$. The conjugacy classes of quasi-Coxeter elements in Weyl groups are indexed according to Carter's notation \cite{carter}, apart from $D_n$ where we write $D_n(a,b)$ for Carter's $D_n(a_{b-1})$ to have a compatible notation with the classification of Corollary~\ref{cor:pqC elements in m1n and mmn} and with Corollary~\ref{cor: Fred for infinite families}~(iii). The conjugacy classes of quasi-Coxeter elements of the groups $H_3$ and $H_4$ are indexed according to increasing Coxeter length (the smallest length of a factorization of a representative $g$ of the class as a product of \emph{simple} reflections). The groups $\Symm_{n+1}=A_n$ and $B_n$ only have a single quasi-Coxeter class, while in the dihedral group $I_2(m)$ all $\frac{\varphi(m)}{2}$ of the quasi-Coxeter classes have the same number $\Fred[I_2(m)](g) = m$ of reduced reflection factorizations (recall that $\varphi$ denotes Euler's totient function).

\begin{table}[H]
\[
\begin{array}{|l l | l l | l l| l l |} 
 \hline
 & & & & & & & \\[-0.3cm]
 g\phantom{\quad\quad\quad} & \Fred(g) & g\phantom{\quad\quad\quad} & \Fred(g) & g\phantom{\quad\quad\quad} & \Fred(g) & g\phantom{\quad\quad\quad} & \Fred(g) \\[0.1cm] 
 \hline
 & & & & & & & \\[-0.3cm]
 A_n & (n+1)^{n-1} & E_7(a_3) & 2\cdot 3^4\cdot 5^6 & E_8(a_7) & 2^{13}\cdot 3^6\cdot 5\cdot 7 & H_4(2) & 3^4\cdot 5^2 \\
 B_n & n^n & E_7(a_4) & 2^4\cdot 3^8\cdot 5\cdot 7 & E_8(a_8) & 2^7\cdot 3^9\cdot 5^2\cdot 7 & H_4(3) & 2^6\cdot 3^3 \\
 I_2(m) & m & E_8 & 2\cdot3^5\cdot 5^7 & F_4 & 2^4\cdot 3^3 & H_4(4) & 2^3\cdot 3\cdot 5^3 \\
 E_6 & 2^9\cdot 3^4 & E_8(a_1) & 2^{18}\cdot 3^5 & F_4(a_1) & 2^3\cdot 3^4 & H_4(5) & 2\cdot 3^2\cdot 5^3 \\
 E_6(a_1) & 3^{10} & E_8(a_2) & 2^{10}\cdot 5^7 & H_3 & 2\cdot 5^2 & H_4(6) & 3^4\cdot 5^2 \\
 E_6(a_2) & 2^6\cdot 3^5\cdot 5 & E_8(a_3) & 2^{12}\cdot 3^6\cdot 5\cdot 7 & H_3(1) & 2\cdot 3^3 & H_4(7) & 2^6\cdot 5^2 \\
 E_7 & 2\cdot 3^{12} & E_8(a_4) & 2\cdot 3^{13}\cdot 5\cdot 7 & H_3(2) & 2\cdot 5^2 & H_4(8) & 2^3\cdot 3^4\cdot 5 \\
 E_7(a_1) & 2\cdot 7^7 & E_8(a_5) & 3^5\cdot 5^7\cdot 7 & H_4 & 2\cdot 3^3\cdot 5^2 & H_4(9) & 2\cdot 3^3\cdot 5^2 \\
 E_7(a_2) & 2^9\cdot 3^6\cdot 5 & E_8(a_6) & 2^3\cdot 3^2\cdot 5^8\cdot 7 & H_4(1) & 2^6\cdot 5^2 & H_4(10) & 2^3\cdot 3\cdot 5^3 \\
 \hline
 \multicolumn{8}{|c|}{ }\\[-0.3cm] 
 \multicolumn{8}{|c|}{ \Fred[D_n](D_n(a,b))=2\cdot (n-1)\cdot\binom{n-2}{a-1,b-1}\cdot a^a\cdot b^b \text{ with } a + b = n}\\[0.2cm]
 \hline
\end{array}
\]
\caption{The numbers $\Fred(g)$ for quasi-Coxeter elements $g$ of real reflection groups $W$.}
\label{Table: Fred(g)}
\end{table}

\subsubsection{A conjecture in terms of geometric invariants of Frobenius manifolds}\label{sssec:conj.frob}

We do not have a complete explanation for the numerological phenomenon illustrated in Table~\ref{Table: Fred(g)}, but we present here a conjectural interpretation, generalizing the geometric approach for Coxeter elements in Section~\ref{ssec: Fred of qcox_Coxeter case}.  This interpretation relies on the theory of \emph{Frobenius manifolds}, pioneered by Dubrovin \cite{dubrovin_ICM}, for which we give below a (very) brief introduction. 

Dubrovin's original goal was to give a coordinate-free formulation of the Witten--Dijkgraaf--Verlinde--Verlinde (WDVV) equations from 2D topological field theory (see \cite{dubrovin_encyclopedia}). He showed that solutions of the WDVV equations, which he called \emph{prepotentials} or \emph{free energies} \cite[Lec.~1]{dubrovin_2dtfts}, are encoded by the structure coefficients of   
smoothly varying Frobenius algebra structures on the tangent planes $T_xM$ of a manifold $M$. Dubrovin called manifolds with such structures \defn{Frobenius manifolds}.

In his study \cite[Lect.~4]{dubrovin-frobenius} of massive Frobenius manifolds, Dubrovin encoded the local algebra structure in a Stokes matrix, or equivalently a tuple of Euclidean reflections $\bm t:=(t_1,\cdots,t_n)$, while he described its analytic continuation via the Hurwitz action (as in \S\ref{sec: hurwitz action}) of the braid group $\BBB_n$ on $\bm t$ (see also the detailed presentation in \cite[\S1.4]{dub_maz} for the rank-$3$ case). Under this interpretation, {\it algebraic} prepotentials correspond to tuples $\bm t$ with finite Hurwitz orbits, and Dubrovin asked for the construction of the corresponding Frobenius manifolds. By work of Michel \cite{michel-hurwitz-tuples}, $\bm t$ has a finite Hurwitz orbit if and only if it generates a finite reflection group, so that the problem of algebraic Frobenius manifolds in some sense lives entirely in the world of finite Coxeter groups and their quasi-Coxeter elements. Indeed, a reduced tuple $\bm t$ always determines a quasi-Coxeter element $g:=\prod_{i=1}^n t_i$ of the group $W' :=\langle\bm t\rangle$, so that we can index the possible corresponding Frobenius manifold by the conjugacy class of $g$ and denote it $\mathcal{F}_g$.

In the case of a \emph{Coxeter} element $c$ in a real reflection group $W\leq \GL(V)$, Dubrovin constructed the manifold $\mathcal{F}_c$ by adding a Frobenius algebra structure on the quotient variety $V/W$ itself \cite{dubrovin_coxeter}. This algebra structure was defined via a special choice of fundamental invariants, known as {\it Saito flat coordinates}, that provide a Euclidean metric for the orbit space $V/W$. Dubrovin conjectured \cite{dubrovin-frobenius} and Hertling later proved \cite[\S5.4]{hertling_book}  that, in fact, these are the only examples of Frobenius manifolds with associated {\it polynomial} prepotentials.

On the other hand, we saw in \S\ref{ssec: Fred of qcox_Coxeter case} that for a Coxeter element $c\in W$, one can relate the number $\Fred(c)$ with the degree of the weighted-homogeneous Lyashko--Looijenga map; in fact, they are equal. Many of the relevant geometric objects of this approach also exist in the world of Frobenius manifolds. In particular, there is an analogue of the $LL$ map \cite[\S3.5]{hertling_book}, which relates two natural coordinate systems of $\mathcal{F}_g$. It sends the flat coordinates, on which the prepotential is given, to (the elementary symmetric polynomials of) the canonical coordinates, which are the eigenvalues in the algebra structure of the multiplication by the Euler field. The $LL$ map is also weighted-homogeneous and one can easily calculate its degree, as in \eqref{eq: deg LL calculation}, from the prepotential associated to $\mathcal{F}_g$. The degree calculations in already constructed algebraic Frobenius manifolds (see  Example~\ref{example:frob_mfld_H3}) and Dubrovin's analysis of their local structures suggest the following conjecture. 

\begin{conjecture}
\label{conj alg frob mfld}
Let $g$ be a quasi-Coxeter element in an irreducible real reflection group $W$. Assuming the Frobenius manifold $\mathcal{F}_g$ exists, then the degree of the map $LL(\mathcal{F}_g)$ equals $\Fred(g)$.
\end{conjecture}

Because of the weighted-homogeneity of the $LL$ map, Conjecture~\ref{conj alg frob mfld} may be seen as a \emph{justification} of the nice numerological properties of the numbers $\Fred(g)$ in Table~\ref{Table: Fred(g)}. Indeed, assuming the conjecture, Bezout's theorem (as in \eqref{eq: deg LL calculation}) would imply that $\deg\big(LL(\mathcal{F}_g)\big)$ is a product of small primes (those involved in the weights and algebraicity degree of the prepotential).

\begin{example}[the algebraic prepotential associated to $g=H_3(1)$]\label{example:frob_mfld_H3} 
We illustrate here the the calculations described in the previous discussion for a special case in $H_3$. Dubrovin and Mazocco \cite{dub_maz} constructed Frobenius manifolds associated to all quasi-Coxeter elements of $H_3$ and Kato--Mano--Sekiguchi \cite[\S7]{KMS_chapter} gave the corresponding prepotentials. In terms of the flat coordinates $t_1,t_2,t_3$, the prepotential for the element $g=H_3(1)$ (denoted in these references by $(H_3)''$) is
\begin{equation}\label{eq:example:prepotential}
F(t_1,t_2,t_3)=\dfrac{4063}{1701}t_1^7+\dfrac{1}{2}(t_2^2t_3+t_1t_3^2)+\dfrac{19}{135}t_1^5z^2-\dfrac{73}{27}t_1^3z^4+\dfrac{11}{9}t_1z^6-\dfrac{16}{35}z^7,
\end{equation}
where $z$ is given in terms of the $t_i$ via the \emph{algebraic} equation
\begin{equation}\label{eq:example:z_eqn}
z^2-t_1^2+t_2=0.
\end{equation}
It is easy to see that the algebraic function $F(t_1,t_2,t_3)$ is weighted-homogeneous with weights
\[
\op{wt}(t_1)=\frac{1}{3}, \qquad
\op{wt}(t_2)=\frac{2}{3}, \qquad
\op{wt}(t_3)=1.
\]
From this data, we calculate the degree of the corresponding $LL$ map as in \eqref{eq: deg LL calculation}: 
\[
\deg\big(LL(\mathcal{F}_g)\big)=\dfrac{3!}{\prod_{i=1}^3\op{wt}(t_i)}\cdot 2=\dfrac{6}{2/9}\cdot 2=54,
\]
where the ``extra" factor of $2$ is the algebraicity degree (i.e., the number of branches) of \eqref{eq:example:z_eqn}. This agrees with the value $\Fred[H_3](g)$ in Table~\ref{Table: Fred(g)} and thus confirms Conjecture~\ref{conj alg frob mfld} for the case $H_3(1)$.
\end{example}

\subsubsection{The case of regular quasi-Coxeter elements in Weyl groups}\label{sssec.reg_qcox_freds}

An obstacle to using the geometry of Frobenius manifolds to study the combinatorics of the sets $\op{Red}_W(g)$ for quasi-Coxeter elements $g$ is that in most cases the corresponding manifolds $\mathcal{F}_g$ have not been constructed yet. In fact their very existence is generally conjectural (this is implicit in Dubrovin's work, see for example the remark above Example~5.5 in \cite[pp.~402--404]{dubrovin-frobenius}). However, for the subclass of \emph{regular} (as in \S\ref{Sec: well-gend grps}) quasi-Coxeter elements, the manifolds $\mathcal{F}_g$ are known to exist and enough data about their structure is available so that we can give an \emph{almost} explicit formula for the counts $\Fred(g)$ (in Proposition~\ref{Obs. Fred regular qCox} and Remark~\ref{Rem. numbers dg are hard} below). 

The most important work towards constructing algebraic Frobenius manifolds has been carried out by Dinar \cite{dinar_first,dinar_polynomial,yassir-subregular,dinar_regular_classes}. He also seems to have been the first to explicitly state Dubrovin's refined conjecture, namely, that there exists a Frobenius manifold for every conjugacy class of quasi-Coxeter elements and that this family constitutes the totality of Frobenius manifolds with algebraic prepotentials.

Dinar focused on Weyl groups, where quasi-Coxeter classes correspond in some sense to certain nilpotent orbits in the Lie algebra. In \cite{dinar_polynomial}, he gave a construction of polynomial Frobenius manifolds associated to the Coxeter class (corresponding to the regular nilpotent orbit) that was different from Dubrovin's but that produced equivalent structures. In \cite{dinar_regular_classes}, he constructed Frobenius manifold structures on Slodowy-type slices associated to regular quasi-Coxeter classes. Dinar's work not only produced new algebraic Frobenius manifolds (even a single new example is considered valuable in the area), but his approach is a conceptual one and his proofs are case-free.

For the case of regular quasi-Coxeter elements $g$ in a Weyl group $W$, Dinar gave an explicit description of the weights of the flat coordinates of $\mathcal{F}_g$. The \defn{exponents} $\{e_j(g)\}_{j=1}^n$ of an element $g$ are the integers in the interval $\big[0,|g|-1\big]$ such that the eigenvalues of $g$ (counted with multiplicity) are $\big\{e^{2\pi i e_j(g)/|g|}\big\}_{j = 1}^n$. Dinar showed\footnote{When comparing \cite[Table~1]{dinar_regular_classes} with Table~\ref{Table: Fred(g)}, the reader should be aware that Dinar uses the Bala--Carter notation for nilpotent orbits in the Lie algebra, not the Carter notation for the corresponding classes in the Weyl group; see also Remark~\ref{rem:ToB}.} \cite[Thm.~1.1]{dinar_regular_classes} that the weights of $\mathcal{F}_g$ are given by the numbers $(e_j(g)+1)/|g|$ (see also the prior work of Pavlyk \cite{oleksandr}). Then, analogously to \eqref{eq: deg LL calculation}, the weighted-homogeneity of the $LL$ map would give its degree as 
\begin{equation}
\deg\big(LL(\mathcal{F}_g)\big)\,=\,\left(\prod_{j=1}^n\dfrac{|g|\cdot j}{e_j(g)+1}\right)\cdot d_g\,=\,\dfrac{|g|^n\cdot n!}{\prod_{j=1}^n(e_j(g)+1)}\cdot d_g,\label{EQ: deg(LL) Frob}
\end{equation}
where $d_g$ is the algebraicity degree of the Frobenius prepotential (i.e., the number of branches it has as an algebraic function on the flat coordinates). 

Dinar's work is not sufficient to prove Conjecture~\ref{conj alg frob mfld}, even for regular elements, because it does not allow for a direct computation of the numbers $d_g$.  However, when taken with the next result (a natural generalization of the Arnold--Bessis--Chapoton formula \eqref{eq: deligne-arnold-bessis formula}), it should be viewed as significant evidence in favor of the conjecture.  

\FrobFredNums

\begin{proof}
In the case that $g$ is a Coxeter element, since $\#W = \prod_{j=1}^n (e_j(g) + 1)$ \cite[Thm.~3.19]{Humphreys}, we have $|g|^n \cdot n!/\prod_{j=1}^n(e_j(g)+1) = h^n \cdot n!/\#W$ and the result is immediate by \eqref{eq: deligne-arnold-bessis formula}.  

In types $A_n$, $B_n$, and $G_2$, all quasi-Coxeter elements are Coxeter elements, so there is nothing more to check.  It is easy (e.g., by using the classification in Corollary~\ref{cor:pqC elements in m1n and mmn} and explicitly computing the eigenspaces) to check that in type $D_{2n + 1}$, the only regular quasi-Coxeter elements are the Coxeter elements, and that in type $D_{2n}$, there is a second class of regular quasi-Coxeter elements; in terms of the notation in Table~\ref{Table: Fred(g)}, these are the elements in conjugacy class $D_{2n}(n, n)$.  Treating $D_{2n}$ as a subgroup of $B_{2n}$, these elements are Coxeter elements for the subgroup $B_n \times B_n$ and so they have order $|g|=2n$ and exponents $\{1, 1, 3, 3, \ldots, 2n -1, 2n - 1\}$.  Therefore, using Table~\ref{Table: Fred(g)}, we can calculate the number $\delta_g$ for these elements as
\[
\delta_g = \frac{2 \cdot (2n - 1) \cdot \binom{2n - 2}{n - 1} \cdot n^n \cdot n^n}{(2n)^{2n} \cdot (2n)! / (2 \cdot 4 \cdots 2n)^2} = n.
\]
Finally, for the remaining types $F_4$, $E_6$, $E_7$, and $E_8$, it is a straightforward computer calculation to test regularity, compute exponents, and compare the results with the data in Table~\ref{Table: Fred(g)}. 
\end{proof}

\begin{table}[H]
\[
\begin{array}{|l||c|c|c|c|c|c|}
\hline
g\in W & D_{2n}(n,n) & F_4(a_1) & E_6(a_1) & E_6(a_2) & E_7(a_1) & E_7(a_4) \\ \hline
 & & & & & &  \\[-0.35cm]
\delta_g & n & 3 & 2 & 5 & 2 & 2\cdot 3^2 \\ \hline\hline
 & & & & & &  \\[-0.35cm]
g\in W& E_8(a_1) & E_8(a_2) & E_8(a_3) & E_8(a_5) & E_8(a_6) & E_8(a_8)  \\ \hline
 & & & & & &  \\[-0.35cm]
\delta_g & 2 & 3 & 2^3 & 7 & 2^2\cdot 5 & 3^3\cdot 5 \\
\hline
\end{array}
\]
\caption{The numbers $\delta_g$ for regular quasi-Coxeter elements $g$ of Weyl groups $W$,  using Carter's notation as in Table~\ref{Table: Fred(g)}.}
\label{Table: the numbers d_g}
\end{table}

\begin{remark}[interpretation of the numbers $\delta_g$]\label{Rem. numbers dg are hard}
In all the cases that the prepotentials of the Frobenius manifolds $\mathcal{F}_g$ have been explicitly given ($D_4(2,2)$ in \cite[\S6]{yassir-subregular}; $F_4(a_1)$ in \cite[Ex.~5.4]{dinar_first}; $E_8(a_1)$ in \cite{dinar_sekiguchi}; and $E_6(a_1)$, $E_7(a_1)$ in \cite[\S7]{seki_constr}), Conjecture~\ref{conj alg frob mfld} is confirmed (i.e., $\delta_g=d_g$). This includes the case of Coxeter elements $c$, when the prepotential is polynomial and hence $d_c=1$. Applied in the opposite direction, this enumerative data can be exploited to guess solutions to the WDVV equations; see \S\ref{sec:8:wdvv} for more. 
\end{remark}

\begin{remark}\label{remark:Fred_regulars_fail}
Proposition~\ref{Obs. Fred regular qCox} does not extend verbatim to non-Weyl groups. For example, in $H_3$, all quasi-Coxeter elements $g\in H_3$ are regular, but the numbers $\Fred[H_3](g)$ are not all integer multiples of the expression 
\begin{equation}
\frac{|g|^n \cdot n!}{\prod_{j=1}^n(e_j(g)+1)}.\label{Eq: g^nn!/prod}
\end{equation}
See Remark~\ref{remark:Fred_regulars_extend} and \S\ref{sec:8:wdvv} for more about this.
\end{remark}

\begin{remark}\label{Rem: Frobenius for non Weyl}
In this Section~\ref{Sec: Frobenius mflds} we have only discussed \emph{real} reflection groups, although, as we will see in Section~\ref{ssec: Fred of qcox_infinite family}, the property that the numbers $\Fred(g)$ are products of small primes remains true in the complex setting as well. The theory of Frobenius manifolds does have extensions in the setting of well generated groups. We mention especially the work \cite{KMS}, where the authors describe an extension of the WDVV equations and construct solutions for it, one associated to each well generated complex reflection group. The solutions are certain potential vector fields with polynomial entries (but they do not always integrate to a prepotential, as is the case in real types).  Moreover, the analytic behavior of these differential-geometric structures is encoded by tuples of unitary reflections \cite[Rem.~2.1]{KMS}, generalizing the case of Euclidean reflections in Frobenius manifolds. The potential vector fields constructed in \cite{KMS} seem to be related to Coxeter elements, but there are also \emph{algebraic} solutions of these extended WDVV equations, and the corresponding tuples of reflections determine \emph{quasi-Coxeter} elements; see for example \cite[\S8]{boalch} for the group $G_{27}$. We hope that the development of these theories will allow in the future a uniform geometric interpretation for the structure and enumeration of reduced factorizations of quasi-Coxeter elements in every well generated group, along the lines of Conjecture~\ref{conj alg frob mfld}.
\end{remark}

\subsection{The non-real types}
\label{ssec: Fred of qcox_infinite family}

In this section we discuss the case of non-real reflection groups, and in particular give formulas for the number of reduced reflection factorizations of quasi-Coxeter elements in the infinite families $G(1,1,n)$, $G(m,1,n)$ and $G(m,m,n)$. Our main tool for the latter is a special case of a result in the first part of this series \cite{DLM1} that gives a formula for the number of minimum-length full factorizations of an arbitrary element in $G(m,p,n)$. In this special case, the answer is a multiple of the number of full factorizations of the underlying permutation, which are counted by the  classical Hurwitz formula \eqref{eq: formula H_0}. 

\begin{theorem}[{\cite[Thm.~5.4 for $d=1$]{DLM1}}]
\label{thm:main lead coeff part 1}
Fix an element $g\in G(m,p,n)$ with $k$ cycles of colors $a_1, \ldots, a_k$ such that $\gcd(a_1, \ldots, a_k, p)=1$, and let $a=\gcd(\col(g),m)/p$. If $m=p$, we have 
    \[
\Ftr_{m,m,n}(g) = 
    m^{k-1} \cdot H_0(\lambda),
    \]
while if $m\neq p$, we have
    \[
    \Ftr_{m,p,n}(g) = \begin{cases}
    n(n+k-1)\cdot m^{k-1}  \cdot H_0(\lambda), & \text{ if } a=1  \\[6pt]
    \dfrac{n^2 (n+k)(n+k-1)m^{k} }{2} \cdot \dfrac{\varphi(a)}{pa} \cdot H_0(\lambda), & \text{ if } a\neq 1,
    \end{cases}
    \]
where  $\lambda$ is the cycle type of the underlying permutation of $g$ and $\varphi$ is Euler's totient function.
\end{theorem}

The cases in the next result use the characterization of quasi-Coxeter elements in the infinite families (see \S\ref{subsec:pqC in the infinite family}). 

\begin{corollary}
\label{cor: Fred for infinite families}
Let $g$ be a quasi-Coxeter element in the complex reflection group $W$.
\begin{enumerate}[(i)]

\item For $W=\Symm_n = G(1,1,n)$, we have that 
\[
\Fred(g) = \Ftr_W(g) = n^{n-2}.
\]
\item For $W=G(m,1,n)$ with $m > 1$, we have that 
\[
\Fred(g) = \Ftr_W(g) = n^n.
\]
\item For $W=G(m,m,n)$ with $m > 1$, if $g$ consists of two cycles of lengths $a$ and $b$ with colors that generate $\ZZ/m\ZZ$, then we have that
\[
\Fred(g) = \Ftr_W(g) = m(n-1)\cdot \binom{n-2}{a-1,b-1}\cdot a^a \cdot b^b.
\]
\end{enumerate}
\end{corollary}

\begin{proof}
By Proposition~\ref{Prop: Fred=Ffull for qCox}, for a quasi-Coxeter element $g$, $\Fred(g)=\Ftr_W(g) $. We now compute the latter for each of the cases.

We start with the case (i) for $W=G(1,1,n)=\Symm_n$. The element $g$ is a long cycle and the statement is precisely D{\'{e}}nes' result \cite{Denes} $\Ftr_{1,1,n}(g)=H_0(n)=n^{n-2}$.

For $W=G(m, 1, n)$, the underlying permutation of $g$ is an $n$-cycle and its color generates $\ZZ/m\ZZ$. By Theorem~\ref{thm:main lead coeff part 1} (case $a=\gcd(\col(g),m)/1=1$) we have that $\Ftr_{m,1,n}(g)=n^2\cdot H_0(\lambda)$, where $\lambda=(n)$. The result then follows by D{\'{e}}nes' result  $H_0(n)=n^{n-2}$.

For $W=G(m, m, n)$, the colors $c, -c$ of the cycles of $g$ generate $\ZZ/m\ZZ$ (i.e., $\gcd(c, -c, m)=1$).  Therefore, by Theorem~\ref{thm:main lead coeff part 1} (case $p=m$) we have that $\Ftr_{m, m, n}(g)=m\cdot H_0(\lambda)$ where $\lambda = (a, b)$. The result follows by applying the formula in \eqref{eq: formula H_0} for this $\lambda$.
\end{proof}

\begin{remark}[exceptional complex reflection groups]
\label{remark:Freds for exceptionals}
Using SageMath \cite{sagemath} and CHEVIE \cite{chevie}, we have computed the numbers $\Fred(g)$ for every quasi-Coxeter element $g$ in every non-real well generated exceptional complex reflection group, using Proposition~\ref{Prop. strong=weak} to identify the quasi-Coxeter elements and applying the representation-theoretic techniques described in \cite[\S3]{DLM1}. The resulting output is attached to the arXiv version of this submission as an auxiliary file.  The numbers $\Fred(g)$ all factor as products of small primes.  For example, in the group $G_{34}$ (the largest of the exceptional non-real complex reflection groups), there are $16$ conjugacy classes of quasi-Coxeter elements.  Of these classes, two consist of Coxeter elements, with multiplicative order $h = 42$ and $2 \cdot 3 \cdot 7^5$ reduced reflection factorizations; two consist of the squares $c^2$ of Coxeter elements $c$, with $3 \cdot 5 \cdot 7^5$ factorizations; one consists of the cubes $c^3$ of Coxeter elements, with $2^2 \cdot 5 \cdot 7^5$ factorizations; the remaining classes consist of non-regular elements, of orders
\begin{itemize}
\item $30$, with $2^2 \cdot 3 \cdot 5^6$ factorizations (two classes);
\item $24$, with $2^{16} \cdot 3$ factorizations (two classes);
\item $18$, with $2^2 \cdot 3^9 \cdot 5$ factorizations (two classes);
\item $18$, with $2 \cdot 3^9 \cdot 5$ factorizations (one class);
\item $12$, with $2^{13} \cdot 3^2 \cdot 5$ factorizations (two classes);
\item $12$, with $2^{10} \cdot 3^4 \cdot 5$ factorizations (one class); and
\item $6$, with $2^5 \cdot 3^8 \cdot 5$ factorizations (one class).
\end{itemize}
\end{remark}

\begin{remark}[regular quasi-Coxeter elements in complex types]\label{remark:Fred_regulars_extend}

As we mentioned in Remark~\ref{remark:Fred_regulars_fail}, the formula of Proposition~\ref{Obs. Fred regular qCox} does not directly extend to non-Weyl groups. However, there is a variation of it that does extend. 

Let $W$ be a well generated complex reflection group of rank $n$, let $g\in W$ be a \emph{regular} quasi-Coxeter element, and write $|g|$ for the order of $g$. Then, if we denote by $(e_i)_{i=1}^n$  the \emph{exponents} of $W$ (in the terminology of Proposition~\ref{Obs. Fred regular qCox}, these are the exponents of the inverse $c^{-1}$ of some Coxeter element $c\in W$; equivalently \cite[Def.~10.24 and Thm.~11.56(iii)]{LehrerTaylor}, they are given by $e_i = d_i-1$ where the $d_i$'s are the fundamental degrees \cite[Ch.~3, \S5]{LehrerTaylor} of $W$), we always have that
\begin{equation}
    \Fred(g)=\dfrac{|g|^n\cdot n!}{\prod_{i=1}^n\Big(\big(e_i\!\!\!\!\mod |g|\big)+1\Big)}\cdot\delta_g,\label{EQ:reg_freds}
\end{equation}
where $\delta_g$ is a small integer.  (For the exceptional non-Weyl groups, the numbers $\delta_g$ are listed in an auxiliary file attached to the arXiv version of this paper.) The proof of this formula is completely analogous to the proof of Proposition~\ref{Obs. Fred regular qCox}, but the underlying geometric theory (as in Remark~\ref{Rem: Frobenius for non Weyl}) is not sufficiently built-up to support a geometric interpretation for them (see \S\ref{sec:8:wdvv} for more).

On the other hand, it is easy to see that \eqref{EQ:reg_freds} is a direct generalization of Proposition~\ref{Obs. Fred regular qCox}. Indeed for any regular quasi-Coxeter element $g$ in a Weyl group $W$, the multiset of exponents $\{e_i(g)\}_{i=1}^n$ agrees precisely with the multiset $\{e_i\!\!\mod |g|\}_{i=1}^n$ (this is essentially \cite[Thm.~11.56(iii)]{LehrerTaylor}).
\end{remark}

\subsection{Reduction to the parabolic quasi-Coxeter case}
\label{ssec:reduction from pqc to qc}

The enumerative results in the preceding subsections concern quasi-Coxeter elements.  The following result extends this to the enumeration of reduced reflection factorizations of all parabolic quasi-Coxeter elements.  In particular, it 
implies that for any parabolic quasi-Coxeter element $g\in W$, the number $\Fred(g)$ is always a product of 
small prime numbers.

\begin{corollary}
\label{Prop: Fred of pqcox to Fred of qcox}
Let $W$ be a well generated complex reflection group, $g\in W$ a parabolic quasi-Coxeter element, and $g = g_1\cdots g_s$ the generalized cycle decomposition of $g$ in the sense of Proposition~\ref{Prop: gob-cycle-decomp}.  Then the number $\Fred(g)$ of reduced reflection factorizations of $g$ is given by
\begin{equation}
    \label{eq:fred of pqcox}
    \Fred(g)=\binom{\lR(g)}{\lR(g_1),\ldots,\lR(g_s)}\cdot \prod_{i=1}^s\Fred(g_i).
\end{equation}
\end{corollary}

\begin{proof}
By Proposition~\ref{Prop: Red_W=Red_(W_g)}, all reduced reflection factorizations of $g$ in $W$ consist of reflections that belong to $W_g$.  Let $W_g$ decompose into irreducibles as $W_g = W_1 \times \cdots W_s$, where for $i = 1, \ldots, s$ we have (possibly after a reordering) that $g_i$ is quasi-Coxeter for $W_i$ by Proposition~\ref{Prop: gob-cycle-decomp}.  Since reflections in different $W_i$ commute, all such factorizations are shuffles of (necessarily reduced) reflection factorizations of the $g_i$ in the components $W_i$.  By Proposition~\ref{Prop: Red_W=Red_(W_g)}, the number of reduced $W_i$-factorizations of $g_i$ is $\Fred(g_i)$.  Since, by Corollary~\ref{cor:finer description reduced refl fact}, the length of the reduced $W_i$-factorization of $g_i$ is $\lR(g_i)$, the number of ways to shuffle the factorizations of the $g_i$ is $\binom{\lR(g)}{\lR(g_1),\ldots,\lR(g_s)}$, and the result follows immediately.
\end{proof}

\begin{remark}
Since $\lR(g_i) = \lR[W_i](g_i)$ and $\Fred(g_i) = \Fred[W_i](g_i)$ for all $i$, one could rewrite \eqref{eq:fred of pqcox} using the ``local'' lengths and factorization counts (i.e., with respect to the parabolic subgroup $W_i$) rather than their ``global'' counterparts in $W$.
\end{remark}

\begin{remark}
\label{rem:pC in terms of Cox num}
In the case that $c$ is a parabolic \emph{Coxeter} element in a well generated group $W$, with $W_X$ the parabolic subgroup of $W$ in which $c$ is a Coxeter element and $W_X=W_1\times\cdots\times W_s$ the decomposition of $W_X$ into irreducibles, we have by Proposition~\ref{Prop: Fred of pqcox to Fred of qcox} and \eqref{eq: deligne-arnold-bessis formula} that
\[
\Fred(c)=\binom{n}{n_1,\ldots,n_s}\cdot\prod_{i=1}^s\dfrac{h_i^{n_i} \cdot n_i!}{\#W_i}=\dfrac{n!}{\#W_X}\cdot\prod_{i=1}^sh_i^{n_i}
\]
where $n$ and $n_i$ are respectively the ranks of $W$ and $W_i$ and where $h_i$ is the Coxeter number of $W_i$.
\end{remark}

\section{Counting relative generating sets}
\label{sec: counting rrgs}

In the previous section, we showed that the number $\Fred(g)$ of reduced reflection factorizations of a parabolic quasi-Coxeter element $g$ in a well generated complex reflection group $W$ are well behaved (they are always products of small primes) and are often given by nice formulas (as in Proposition~\ref{Obs. Fred regular qCox} and  Corollaries~\ref{cor: Fred for infinite families},~\ref{Prop: Fred of pqcox to Fred of qcox}). In this section, we show that there are also nice formulas for the numbers of (relative) generating sets in the combinatorial subfamily of well generated complex reflection groups.  In particular, we see in \S\ref{sec: rrgs combinatorial family} that the relative generating sets preserve the tree-like enumeration from the case of the symmetric group (\S\ref{sec: couting rrgs in S_n}). Leaving the combinatorial family, in \S\ref{sec: counting RGS with laplacian} we see that the number of (relative) generating sets does not seem to have nice uniform formulas, except for a special case in Weyl groups via the $W$-Laplacian of \cite{chapuy_theo}. Lastly, in \S\ref{sec: comparison DPS cacti}, we relate the relative generating sets in the symmetric group and certain cacti used by Duchi--Poulalhon--Schaeffer \cite{DPS_II} to compute Hurwitz numbers bijectively.

\subsection{Counting relative generating sets in the symmetric group} \label{sec: couting rrgs in S_n}

In the case of the symmetric group, the relative generating sets have a tree-like structure. Thus, it is unsurprising that their enumeration will be aided by the following weighted version of Cayley's theorem.

\begin{theorem}[{weighted Cayley theorem \cite[Thm.~5.13]{Bona}}] \label{thm: weighted cayley}
We have
\[
 \sum_{\substack{T \textrm{ a tree on vertex} \\ \textrm{set } \{0, 1, \ldots, k\}}} \ \prod_{i = 0}^{k} x_i^{\deg_T(i)} = x_0x_1\cdots x_k(x_0+x_1+\cdots + x_k)^{k-1}.
\]
 \end{theorem}
 
\begin{proposition} \label{prop: RGS for symmetric group}
Let $g$ in $W=G(1,1,n)=\mathfrak{S}_n$ have cycle type $\lambda=(\lambda_1,\ldots,\lambda_k)$.  Then  
\begin{equation} \label{eq:num rgs in Sn}
\# \RGS(W, g) = n^{k - 2} \cdot \prod_{i = 1}^k \lambda_i.
\end{equation}
\end{proposition}

\begin{proof}
By Proposition~\ref{prop:RGS characterization}~(i), the elements  in the relative generating set of $g$ correspond to a tree relative to the partition $\Pi_g$ of $\{1, \ldots, n\}$ into the $k$ cycles of $g$. In such a tree, there are $ \lambda_i \cdot \lambda_j$ choices of an edge joining cycle $i$ to cycle $j$ for $i, j$ in $\{1, \ldots, k\}$, and thus 
\[
\# \RGS(G(1,1,n), g) = \sum_{\substack{T \text{ a tree } \\ \text{on } \{1, \ldots, k\}}} \prod_{i = 1}^{k} \lambda_i ^{\deg(i)} \overset{\left(\substack{\text{Thm.} \\ \ref{thm: weighted cayley}}\right)}{=}(\lambda_1 + \ldots + \lambda_{k})^{k - 1}
\cdot \prod_{i = 1}^k \lambda_i  = n^{k - 2} \cdot \prod_{i = 1}^k \lambda_i.  \qedhere
\]
\end{proof}

\subsection{Counting relative generating sets in the combinatorial family} 

\label{sec: rrgs combinatorial family}

In the combinatorial family, the relative generating sets retain the tree-like structure apparent in the case of symmetric groups.  The cases in the following theorem correspond to those of Proposition~\ref{prop:RGS characterization}.

\thmcountrgs

\begin{proof}
In case (i), we have by Proposition~\ref{prop:RGS characterization}~(i) that the elements in the relative generating set correspond to a tree relative to the partition $\Pi_g$ of $\{1, \ldots, n\}$ into the $k + 1$ generalized cycles of $g$. In such a tree, there are $m \cdot \lambda_i\cdot  \lambda_j$ choices of an edge joining cycle $i$ to cycle $j$ (for $i, j$ in $\{0, 1, \ldots, k\}$), and thus 
\[
\# \RGS(W, g) = m^{k} \sum_{\substack{T \text{ a tree on } \\\{0, 1, \ldots, k\}}} \prod_{i = 0}^{k} \lambda_i ^{\deg(i)}.
\]
By the weighted Cayley theorem (Theorem~\ref{thm: weighted cayley}), this simplifies to
\[
\# \RGS(G(m, 1, n), g) = m^{k} \cdot (\lambda_0 + \ldots + \lambda_{k})^{k - 1}
\cdot \prod_{i = 0}^k \lambda_i  = m^{k} \cdot n^{k - 1} \cdot \prod_{i = 0}^k \lambda_i,
\]
as claimed.

In case (ii), we have by Proposition~\ref{prop:RGS characterization}~(ii) that the elements in the relative generating set correspond to a rooted tree relative to the partition $\Pi_g$ of $\{1, \ldots, n\}$ into the $k$ (generalized) cycles of $g$.  There are $\varphi(m) \cdot n$ choices for the diagonal reflection, as its nontrivial eigenvalue may be any of $\varphi(m)$ primitive roots of unity and may be located in any of $n$ diagonal positions.  The remaining elements form a tree relative to $\Pi_g$, so may be chosen in $ m^{k - 1} \cdot n^{k - 2} \cdot \prod_{i = 1}^k \lambda_i$ ways (as in case (i)).  Multiplying these independent choices together gives the result.

Finally, in case (iii), we have by Proposition~\ref{prop:RGS characterization}~(iii) that the elements in the relative generating set correspond to a unicycle relative to the partition $\Pi_g$ of $\{1, \ldots, n\}$ into the $k$ (generalized) cycles of $g$.  Therefore, by Proposition~\ref{prop:relative trees}~(iii),  given any relative generating set of reflections for $g$, taking the associated graph $\Gamma$ and contracting the cycles of $g$ to single points leaves either a rooted tree or a unicycle.

Suppose first that the contraction is a rooted tree; we compute the number of relative generating sets that map to this particular graph, as follows: if (after contraction) the loop is on the $i$th cycle, its endpoints (before contraction) may be chosen in $\binom{\lambda_i}{2}$ ways, and its color may be chosen in $\varphi(m)$ ways (to make $\delta$ a primitive generator).  The remaining $k - 1$ edges form a tree relative to the cycle partition $\Pi_g$, so (as in cases~(i) and~(ii)) may be chosen in $m^{k - 1} \cdot n^{k - 2} \cdot \prod_{i = 1}^k \lambda_i$ ways.  Thus the total contribution to $\#\RGS(W, g)$ in this case is 
\[
\varphi(m)\cdot m^{k - 1} n^{k - 2} \prod_{i = 1}^k \lambda_i \cdot \sum_{i = 1}^k \binom{\lambda_i}{2} 
=                                 
\frac{1}{2}\varphi(m) \cdot m^{k - 1}n^{k - 2}\left(-n + \sum_{i = 1}^k \lambda_i^2\right) \prod_{i = 1}^k \lambda_i.
\] 
Now suppose instead that the contraction of $\Gamma$ is a unicycle.  In this case, the length $j$ of the contracted graph cycle is at least $2$.  Let the vertices in the cycle be indexed by $S = \{s_1, \ldots, s_j \} \in \binom{[k]}{j}$.  Performing a second contraction on this set leaves a tree on $k - j + 1$ vertices, with weights $\{\lambda_i \colon i \not \in S\} \cup \{\lambda_{s_1} + \ldots + \lambda_{s_j}\}$. 
We now consider how many relative generating sets contract to such trees.  By the weighted Cayley theorem (Theorem~\ref{thm: weighted cayley}), the number of ways to choose the reflections corresponding to the edges of the tree is
\[
m^{k - j} \cdot (\lambda_{s_1} + \ldots + \lambda_{s_j}) \cdot \prod_{i \not \in S} \lambda_i \cdot (\lambda_1 + \ldots + \lambda_k)^{k - j - 1} 
=
m^{k - j} \cdot n^{k - j - 1} \cdot (\lambda_{s_1} + \ldots + \lambda_{s_j}) \cdot \prod_{i \not \in S} \lambda_i,
\]
where the factor $m^{k - j}$ corresponds to the choice of colors.  Independently, the number of ways to choose the reflections that correspond to the graph cycle is 
\[
\frac{\varphi(m) \cdot m^{j - 1} (j - 1)!}{2} \cdot (\lambda_{s_1} \cdots \lambda_{s_j})^2,
\]
where the first factor accounts for the number of different ways of inserting a cycle (including colors) while reversing the second contraction, and the second factor account for the number of ways to reverse the first contraction (since the contracted vertices have weights $\lambda_{s_1}, \ldots, \lambda_{s_j}$ and a graph cycle is $2$-regular).
Combining all the different values of $j$ (including the $j=1$ loop case), we conclude that
\begin{align}
\notag
\frac{\# \RGS(W, g)}{\varphi(m)} &=  
\frac{m^{k - 1}}{2}  \prod_{i = 1}^k \lambda_i  \cdot
\Big(-n^{k - 1}  +
\sum_{j = 1}^k (j - 1)! \cdot n^{k - j - 1} \sum_{S \in \binom{[k]}{j}} (\lambda_{s_1} + \ldots + \lambda_{s_j}) \lambda_{s_1} \cdots \lambda_{s_j}
\Big)\\
\label{eq:RGS 2(b)}
&=  \frac{m^{k - 1}}{2} \prod_{i = 1}^k \lambda_i \cdot  
\Big(-n^{k - 1}  + \sum_{j = 1}^k (j - 1)! \cdot n^{k - j - 1} m_{(2,1^{j-1})}(\lambda)\Big),
\end{align}
where $m_{\mu}(\lambda)$ is the monomial symmetric polynomial in the variables $\lambda=(\lambda_1,\ldots,\lambda_k)$. We may rewrite this symmetric polynomial when $\mu=(2,1^{j-1})$ in terms of the elementary symmetric polynomials as $m_{(2,1^{j-1})}= e_{j,1}-(j+1) e_{j+1}$. Since $e_1(\lambda)=n$, \eqref{eq:RGS 2(b)} becomes
\[
 \frac{\# \RGS(W, g)}{\varphi(m)} =  \frac{m^{k - 1}}{2} \prod_{i = 1}^k \lambda_i  \cdot
\Big(-n^{k - 1}  +  \sum_{j=1}^k (j-1)! \cdot n^{k-j} e_j(\lambda) - \sum_{j=1}^k (j-1)! \cdot (j+1)
  n^{k-j-1} e_{j+1}(\lambda)\Big).
\]
The first term of the first sum on the right side is
$n^{k-1}e_1(\lambda)=n^k$. Separating this term and combining the
remaining terms with the second sum gives
\begin{align*}
\frac{\# \RGS(W, g)}{\varphi(m)}  &= \frac{m^{k - 1}}{2} \prod_{i = 1}^k \lambda_i  \cdot
\Big(n^k - n^{k-1} -\sum_{j=2}^k  (j-2)! \cdot n^{k-j} e_j(\lambda)\Big),
\end{align*}
as claimed.
\end{proof}

\begin{remark}
The RHS of case (iii) of Theorem~\ref{thm:count rgs} is related to the formula for genus-$1$ Hurwitz numbers in the symmetric group (see \cites{GJ99,V01}). This will play a role in \cite{DLM3}.
\end{remark}

\begin{remark}
It will be important in the sequel \cite{DLM3} that the factor $\varphi(m)$ in cases (ii) and (iii) of Theorem~\ref{thm:count rgs} corresponds to an equidistribution: in case (ii), each primitive $m$th root of unity occurs equally often as an eigenvalue for the diagonal reflection in the elements of $\RGS(W, g)$, while in case (iii), each primitive $m$th root of unity occurs equally often as the value of $\delta(F \cup S)$ where $F$ is the set of factors in a fixed reflection factorization of $g$ and $S$ varies over $\RGS(W, g)$. 
\end{remark}

\subsection{Counting good generating sets via the $W$-Laplacian} \label{sec: counting RGS with laplacian}

Outside the combinatorial family, the collections $\RGS(W)$ and $\RGS(W,g)$ of (relative) generating sets do not seem to have good numerological properties. There is however a closed formula, at least when $W$ is a Weyl group and $g=\id$. If $W$ is an irreducible Weyl group of rank $n$ and $h$ is its Coxeter number, it was shown in \cite[proof of Thm.~8.11]{chapuy_theo} that 
\[
h^n=\sum_{W'\leq W} \# \RGS(W')\cdot I(W'),
\]
where the sum is over all rank-$n$ reflection subgroups $W'$ of $W$. M\"obius inversion in the poset of reflection subgroups (ordered by reverse inclusion) gives
\begin{equation}
\# \RGS(W) =\dfrac{1}{I(W)} \cdot \sum_{W'\leq W} \mu(W,W') \prod_{i=1}^nh_i(W'),\label{eq: RGS(W) via laplacia}
\end{equation} 
where again the sum is over all rank-$n$ reflection subgroups of $W$, and where $\{h_i(W')\}$ denotes the \defn{multiset of Coxeter numbers} of $W'$. This is defined as the multiset where each irreducible component $W_i$ of $W'$ contributes $\rank(W_i)$-many times its Coxeter number $h(W_i)$ (as in Remark~\ref{rem:pC in terms of Cox num}); in particular, for an irreducible $W$ we have $\{h_i(W)\}=\{\underbrace{h,h,\dots,h}_{n\text{ times}}\}$.

\begin{example}
The case of $E_6$ is a particularly nice example where the calculation above can be rendered by hand (see also \cite[Ex.~8.13]{chapuy_theo}). In \cite{DPR}, the authors have compiled tables with the reflection subgroups of finite Coxeter groups. In particular, for $E_6$, the poset of \emph{maximum rank} reflection subgroups contains a minimal
element ($E_6$ itself) and only atoms: $36$ groups of type $A_1\times A_5$ and $40$ groups of type $A_2\times A_2 \times A_2$. The multisets of Coxeter numbers of the groups $E_6$, $A_1A_5$, and $A_2^3$ are respectively $\{12^6\}$, $\{2, 6^5\}$, and $\{3^6\}$ (where exponents denote multiplicities).  Then Equation~\eqref{eq: RGS(W) via laplacia} for the cardinality of the family $\RGS(E_6)$ evaluates to
\[
\#\RGS(E_6)=\dfrac{1}{3}\cdot \left( 1\cdot 12^6 -1\cdot 36\cdot (2\cdot 6^5)-1\cdot 40\cdot(3^6)\right)=798984.
\]
\end{example}

The results of performing the same calculation for the other exceptional Weyl groups are given in Table~\ref{table:RGS for exceptional Weyl groups}.

\begin{table}
\[
    \begin{array}{|c||l|}
    \hline
        W & \hspace{.7in} \#\RGS(W) \\\hline
        G_2 & \phantom{12345678910}6 = 2 \cdot 3 \\
        F_4 & \phantom{12345678}2160 = 2^4 \cdot 3^3 \cdot 5\\
        E_6 & \phantom{123456}798984 = 2^3 \cdot 3^6 \cdot 137 \\
        E_7 & \phantom{123}196800768 = 2^8 \cdot 3^2 \cdot 229 \cdot  373 \\
        E_8 & 273643237440 = 2^6 \cdot 3^3 \cdot 5 \cdot 31671671 \\\hline
    \end{array}
\]
    \caption{Number of good generating sets of reflections for the exceptional Weyl groups.}
    \label{table:RGS for exceptional Weyl groups}
\end{table}

\subsection{Comparison with the cacti of Duchi--Poulalhon--Schaeffer} \label{sec: comparison DPS cacti}

The relative trees appearing in the proof of Proposition~\ref{prop: RGS for symmetric group} are similar to certain objects called  {\em Cayley cacti} that were introduced by Duchi--Poulalhon--Schaeffer \cite[\S 3]{DPS_II}. The latter were part of a construction to bijectively compute Hurwitz numbers in the symmetric group (a special case of a more general construction for \emph{double Hurwitz numbers}). This is arguably the only bijective proof of the formula for genus-$0$ Hurwitz numbers, although there are other combinatorial proofs \cite{strehl,GJ97,B-MS}.

Cayley cacti are defined as follows. Given positive integers $k\leq n$ and a weak composition ${\bf m}=(m_1,\ldots,m_n)$ of $k$ with $\sum_i i\cdot m_i=n$, a \defn{Cayley cactus} of type ${\bf m}$ is a tree consisting of $m_i$ oriented $i$-polygons connected by $k-1$ labeled edges.  Let $DPS(m_1,\ldots,m_n)$ be the number of Cayley cacti of type ${\bf m}$.  The authors in \cite{DPS_II} gave the following formula for this number:
\begin{equation} \label{eq:CayleyCacti}
DPS(m_1,\ldots,m_n) = \frac{1}{k} \binom{k}{m_1,\ldots,m_n} \cdot n^{k-2}.
\end{equation}
Comparing this formula with the formula \eqref{eq:num rgs in Sn} for the number $RT(\lambda)$ of relative trees, we see that 
\begin{equation} \label{eq:rel 2 formulas}
\binom{k}{m_1,\ldots,m_n} \cdot RT(\lambda) \,=\, k\cdot \prod_i \lambda_i \cdot DPS(m_1,\ldots,m_n),
\end{equation}
where $\lambda=(1^{m_1},2^{m_2},\ldots,n^{m_n})$.
This relation can be explained as follows. From the proof of Proposition~\ref{prop:RGS characterization}~(i), underlying each relative tree there is a tree $T$ with labeled vertices $\{1,\ldots,k\}$, and so (as in the proof of Theorem~\ref{thm:count rgs}~(i)) we have
\begin{equation} \label{eq: expression RT}
   RT(\lambda) =  \sum_{\substack{T \text{ a tree on vertex} \\ \text{set } \{1, \ldots, k\}}} \ \prod_{i = 1}^{k} \lambda_i ^{\deg(i)}.
\end{equation}
Similarly, underlying each Cayley cactus there is a tree $T'$ with labeled edges $\{1,\ldots,k-1\}$;  see Figure~\ref{fig:comparing cacti}.  Thus
\begin{equation}
\label{eq: expression DPS}   
DPS(m_1,\ldots,m_n) = \sum_{\Pi} \sum_{\substack{T' \text{ a tree with } \\ \text{edges } \{1, \ldots, k-1\}}} \prod_{i = 1}^{k} \lambda_i ^{\deg(i)-1},
\end{equation}
where $\Pi$ is a multiset in $\binom{[k]}{m_1,\ldots,m_n}$ encoding the distribution of oriented $i$-polygons, and the power $\lambda_i^{\deg(i)-1}$  is due the polygons in the cacti being oriented. Next, there is a standard $k$-to-$1$ map from trees with edges labeled $\{1,\ldots,k-1\}$ to trees with vertices labeled $\{1,\ldots,k\}$: choose a vertex of the tree to be labeled $1$ and {\em push} the label $i$ of an edge to a vertex label $i+1$ away from the chosen vertex; see Figure~\ref{fig:comparing cacti}. Using this $k$-to-$1$ map on \eqref{eq: expression RT} gives
\[
 \binom{k}{m_1,\ldots,m_n} \cdot RT(\lambda) 
 =  \sum_{\Pi} \sum_{\substack{T \text{ a tree on vertex } \\ \text{set } \{1, \ldots, k\}}} \ \prod_{i = 1}^{k} \lambda_i ^{\deg(i)} 
 = k\sum_{\Pi} \sum_{\substack{T' \text{ a tree with} \\ \text{edges } \{1, \ldots, k-1\}}} \prod_{i = 1}^{k} \lambda_i ^{\deg(i)}.
\]
Lastly, by \eqref{eq: expression DPS}, this is just $k\cdot \prod_i \lambda_i \cdot DPS(m_1,\ldots,m_n)$, verifying \eqref{eq:CayleyCacti}.

\begin{figure}
    \centering
    \includegraphics{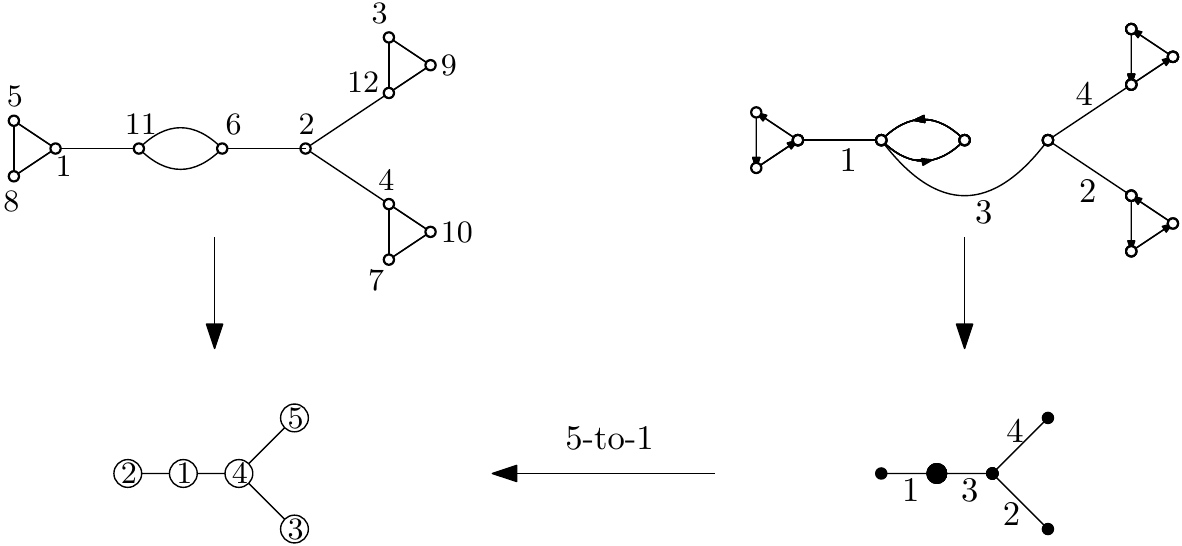}
    \caption{The relative tree representing the relative generating set $\{(1,11),(2,6),(2,4),(2,12)\}$ of $\sigma=(1\, 5\, 8)(2)(3\, 12\,9)(4\,7\,10)(6\,11)$ and its underlying vertex-labeled tree (left) and the edge-labeled cactus appearing in the work Duchi--Poulhalhon--Schaeffer and its underlying  edge-labeled tree (right). The underlying labeled trees are in correspondence by selecting a vertex in the edge-labeled tree (illustrated with a larger vertex) to be labeled $1$ and pushing away the label $i$ of an edges to a vertex label $i+1$.}
    \label{fig:comparing cacti}
\end{figure}

\section{Further remarks}
\label{sec: gen_quest}

We end with a number of remarks and open questions.

\subsection{Elements that do not belong to proper reflection subgroups}
\label{sec:8:1}

As discussed in \S\ref{sec.origins}, the quasi-Coxeter elements in a \emph{real} reflection group $W$ are precisely the elements that do not belong to any proper reflection subgroup of $W$.  In the complex case, this is \emph{not} a characterization; for example, the element $[\id; (1, 1, 1)] \in W = G(3, 3, 3)$ does not belong to any proper reflection subgroup of $W$ but is not quasi-Coxeter (e.g., because it has reflection length $4$).  It is not even clear whether all quasi-Coxeter elements have this property: it seems possible in principle that a quasi-Coxeter element $g \in W$ could belong to a reflection subgroup $W' < W$ for which $\lR[W'](g) > \lR(g)$ (so that its reduced $W$-factorizations all generate $W$, but some longer factorization generates only the subgroup $W'$).  This raises several questions.

\begin{question}
Is there an alternate characterization of the elements in a complex reflection group that do not belong to any proper reflection subgroup?  What about those for which $\lR(g) = \codim(V^g)$?
\end{question}
\begin{question}
\label{q:not in proper refn subgroup}
Is it true for every well generated complex reflection group $W$ that every quasi-Coxeter element in $W$ does not belong to any proper reflection subgroup of $W$? 
\end{question}
\begin{question}
\label{q:more about proper refn subgroups}
Is it true for every well generated complex reflection group $W$ that \emph{among the elements $g\in W$ that satisfy $\lR(g) = \codim(V^g)$}, the property of not belonging to any proper reflection subgroup $W' < W$ is equivalent to being a quasi-Coxeter element? 
\end{question}

\begin{remark} \label{rem: open questions}
In the exceptional groups, Questions~\ref{q:not in proper refn subgroup} and~\ref{q:more about proper refn subgroups} should be approachable by an exhaustive computer search using SageMath \cite{sagemath} and CHEVIE \cite{chevie}.
\end{remark}

\subsection{Heredity of reflection factorizations to parabolic subgroups}
\label{sec:8:2}

Proposition~\ref{Prop: Red_W=Red_(W_g)}~(2) establishes a sort of heredity property for reflection factorizations and parabolic subgroups, under the hypothesis $\lR(g) = \codim(V^g)$.  Remarks~\ref{rem_par_1} and~\ref{rem_par_2} show that the condition that the subgroup is parabolic is essential.  We do not know if the other hypothesis can be relaxed.

\begin{question}
Is it true for every parabolic subgroup $W_X$ of a complex reflection group $W$ and every $g \in W_X$ that $\op{Red}_{W_X}(g) = \op{Red}_W(g)$, even if $\lR(g) > \codim(V^g)$?
\end{question}

In the other direction, Proposition~\ref{Prop: X=V^g and l_R(g)=codim(V^g)} has as a hypothesis that a reduced factorization of $g$ is a \emph{good} generating set for a parabolic subgroup $W_X$; attempting to weaken the hypothesis leads to the next question.

\begin{question}
Does there exist an element $g$ in a complex reflection group $W$ such that the factors of a reduced reflection factorization of $g$ generate a parabolic subgroup $W_X$ different from $W_g$?  (Necessarily $\lR(g) > \codim(V^g)$ in this case.)
\end{question}

\subsection{Properties that do not extend from the real to the complex setting}
\label{sec:8:3}

In \S\S\ref{sec:8:1}--\ref{sec:8:2} we have asked about various properties that hold in real reflection groups but fail in general in the complex setting, whether they are still true in the collection of elements that satisfy $\lR(g)=\codim(W^g)$. However, not everything extends \emph{even if we restrict to this subfamily}, as the examples in this section illustrate.

Proposition~\ref{Prop:gob.char.hurw.orb} (from \cite[Cor.~3.11]{gobet-cycle}) shows that for any element in a real reflection group, the number of Hurwitz orbits on the reduced factorizations of $g$ equals the number of minimal reflection subgroups containing $g$. This is not true for complex reflection groups: a counter-example is furnished by the element $g:=[\id; (1,1,1)]\in W:= G(3,1,3)$. This is a central element with reflection length $\lR(g)=\codim(W^g)=3$ and with a single Hurwitz orbit for its reduced reflection factorizations ($[\id; (1,0,0)]\cdot [\id; (0,1,0)]\cdot [\id; (0,0,1)]$ and the permutations of these three factors). The group generated by such a factorization is the normal subgroup $G:=C_3\times C_3\times C_3 \lhd W$ and $G$ is a minimal reflection subgroup containing $g$. As we saw in \S\ref{sec:8:1}, the group $G(3,3,3)$ is also a \emph{minimal} reflection subgroup containing $g$ but it does not correspond to any new Hurwitz orbit.

As another example, Carter \cite[Lem.~2]{carter} showed that for elements $g$ in finite real reflection groups $W$,  the two intervals $[1,g]_{\leq_{\codim}}$ and $[1,g]_{\leq_{\RRR}}$  coincide (an easy consequence of the fact that $\lR(g)=\codim(V^g)$ for \emph{all} elements $g$ in real reflection groups). These intervals also coincide for Coxeter elements in well generated complex reflection groups \cite[Cor. 6.6]{LM}, but for $g=[\id;(1,2,3,4,5,6)]$ in $G(7,7,6)$ -- which does indeed satisfy $\codim(V^g)=\lR(g)=6$ -- we have $[1,g]_{\leq_{\codim}}\neq [1,g]_{\leq_{\RRR}}$ \cite[\S8.6]{LM}.

\begin{question}
Is there a natural way to extend either Proposition~\ref{Prop:gob.char.hurw.orb} or Carter's result to complex reflection groups, notwithstanding  the examples above?
\end{question}

\subsection{Root and coroot lattices for complex reflection groups}\label{sec:8:4} 
One of the most important properties of good generating sets for Weyl groups $W$ is their relation to bases of the root and coroot lattices of $W$, as given in Proposition~\ref{Prop: BW charn of gen sets} (from \cite[Cor.~1.2]{BW}). In particular, we rely on this statement for our characterization of parabolic quasi-Coxeter elements $g\in W$ in Theorem~\ref{Prop: pqCox characterization for Weyl W}. Recently Brou\'e--Corran--Michel \cite{BCM} defined (co)root systems and (co)root lattices (in fact, also a connection index) for complex reflection groups.  In the case of a well generated group $G$, they gave a complete generalization of Proposition~\ref{Prop: BW charn of gen sets}: they proved \cite[Prop.~3.44 + Thm.~6.6]{BCM} that $G$ always has a special (\emph{distinguished, principal}) root system $\mathfrak{R}$ for which any set of $\rank(G)$-many reflections of $G$ will generate the full group $G$ if and only if their corresponding roots and coroots form bases of the root and coroot lattices associated to $\mathfrak{R}$.
Moreover, as we discussed in \S\S\ref{sec.parabolic qce}--\ref{Sec: full factorizations and p-q-Cox}, a big part of the arguments needed in the complete proof of Theorem~\ref{Prop: characterization of pqCox} can be made uniform in the subfamily of Weyl groups, largely relying on the characterization of parabolic quasi-Coxeter elements given in Theorem~\ref{Prop: pqCox characterization for Weyl W}. We may ask then:

\begin{question}
Is there a way to apply the constructions in \cite{BCM} to give a more conceptual explanation for the equivalences in Theorem~\ref{Prop: characterization of pqCox}?
\end{question}

\subsection{Algebraic solutions for the WDVV equations}\label{sec:8:wdvv}

In Remark~\ref{remark:Fred_regulars_extend}, we gave a variation of Proposition~\ref{Obs. Fred regular qCox} that holds for regular elements in non-Weyl groups. It is then reasonable to ask whether there is an interpretation of the numbers $\delta_g$ appearing in \eqref{EQ:reg_freds} analogous to the one in Remark~\ref{Rem. numbers dg are hard}. We pursue this question here in some detail for the group $W=H_4$.

In the case of Weyl groups, Dinar's constructions of Frobenius manifolds described in \S\ref{sssec.reg_qcox_freds} make use of the geometry of the associated Lie algebra in a highly non-trivial way; there is a priori no reason to expect that for the non-Weyl groups $H_3, H_4$, the exponents of a regular element $g$ might be related to the weights of the corresponding Frobenius manifold $\mathcal{F}_g$ (assuming that $\mathcal{F}_g$ exists!). Nevertheless, for all quasi-Coxeter classes in $H_3$, the Frobenius manifolds $\mathcal{F}_g$ have been constructed and there is a connection between their weights and the exponents of $g$; we describe this relationship next.

In a complex reflection group $W$, any number that is the order of a regular element $g \in W$ is called a \defn{regular number} for $W$.  For a regular number $\gamma$, there always exists a regular element $g_0\in W$ that has $\gamma-1$ as an exponent, and by \cite[Thm.~11.56(iii)]{LehrerTaylor}, all the exponents of this element are given by $e_i(g_0)=(d_i-1) \mod \gamma$. Moreover, \emph{any} regular element $g$ of order $\gamma$ must be $W$-conjugate to $g_0^m$ for some power $m$ relatively prime to $\gamma$ (compare with \cite[Thm.~1.3(ii)]{RRS}); there may be many options for $m$ but we always choose the smallest positive one. It is easy to see from \cite[\S6.1]{KMS} that for all regular quasi-Coxeter elements $g\in H_3$, the weights of the corresponding manifolds $\mathcal{F}_g$ are precisely the numbers
\begin{equation}
\dfrac{\big(e_i(g)+m\big)\!\!\!\!\mod|g|}{|g|},\label{eq:Frob-weights-real}
\end{equation}
where the $e_i(g)$ are the exponents of $g$. If we assume that this procedure also produces the weights associated to Frobenius manifolds $\mathcal{F}_g$ for regular elements $g\in H_4$, then, after comparing  Conjecture~\ref{conj alg frob mfld} with Table~\ref{Table: Fred(g)}, we can compute the algebraicity degrees $d_g$ of the corresponding prepotentials of $\mathcal{F}_g$; this data is displayed in Table~\ref{Table: H_4-conj}.

\begin{table}
\[
\begin{array}{|l|c|c||l|c|c|}
\hline
g\phantom{\big)} & \text{weights} & d_g & g\phantom{\big)} & \text{weights} & d_g \\\hline
H_4\phantom{\Bigg)} &  \left(\dfrac{2}{30},\dfrac{12}{30},\dfrac{20}{30},\dfrac{30}{30}\right) & 1 & H_4(9)\phantom{\Bigg)} & \left(\dfrac{14}{30},\dfrac{20}{30},\dfrac{24}{30},\dfrac{30}{30}\right) & 14 \\\hline 
H_4(1)\phantom{\Bigg)} &  \left(\dfrac{2}{20},\dfrac{10}{20},\dfrac{12}{20},\dfrac{20}{20}\right) & 2 & H_4(7)\phantom{\Bigg)} & \left(\dfrac{6}{20},\dfrac{10}{20},\dfrac{16}{20},\dfrac{20}{20}\right) & 8 \\\hline 
H_4(2)\phantom{\Bigg)} &  \left(\dfrac{2}{15},\dfrac{5}{15},\dfrac{12}{15},\dfrac{15}{15}\right) & 3 & H_4(6)\phantom{\Bigg)} & \left(\dfrac{4}{15},\dfrac{9}{15},\dfrac{10}{15},\dfrac{15}{15}\right) & 9 \\\hline
H_4(3)\phantom{\Bigg)} &  \left(\dfrac{2}{12},\dfrac{6}{12},\dfrac{8}{12},\dfrac{12}{12}\right) & 4 & H_4(8)\phantom{\Bigg)} & \left(\dfrac{2}{6},\dfrac{2}{6},\dfrac{6}{6},\dfrac{6}{6}\right) & 15 \\\hline 
H_4(4)\phantom{\Bigg)} &  \left(\dfrac{2}{10},\dfrac{2}{10},\dfrac{10}{10},\dfrac{10}{10}\right) & 5 & H_4(10)\phantom{\Bigg)} & \left(\dfrac{6}{10},\dfrac{6}{10},\dfrac{10}{10},\dfrac{10}{10}\right) & 45 \\\hline 
\end{array}
\]
\caption{Conjectural weights and algebraicity degrees $d_g$ for  Frobenius manifolds $\mathcal{F}_g$ associated with regular quasi-Coxeter elements $g\in H_4$.  Weights are given over the common denominator $|g|$.
\label{Table: H_4-conj}}
\end{table} 

Any algebraic Frobenius manifold associated to a quasi-Coxeter element of $H_4$ would be $4$-dimensional, and the data in Table~\ref{Table: H_4-conj} is often sufficient to allow for a computer-assisted construction of the corresponding prepotential. Sekiguchi \cite{sekiguchi-talk,seki_constr} was successful in doing this for the cases $H_4(2),H_4(3),H_4(7),H_4(9)$. We then ask:

\begin{question}
Is there a conceptual construction of algebraic Frobenius manifolds associated to the regular quasi-Coxeter elements of $H_4$ that explains the data of Table~\ref{Table: H_4-conj}?
\end{question}

The complex case is more problematic. Boalche's icosahedral solution $37$ (see \cite[\S6.3]{KMS}) corresponds to a regular element in $G_{27}$ but its weights are incompatible with the formula \eqref{eq:Frob-weights-real}. 

\subsection{Dual braid presentations and quasi-Coxeter elements}

As mentioned in the second proof of Proposition~\ref{Prop. extensions <W',t_i>=W}, the generalized braid group $B(W)$ of a well generated complex reflection group $W$ is defined to be the fundamental group $\pi_1(V^{\op{reg}}/W)$ of the space of free orbits of $W$, or equivalently of the complement $V/W - \mathcal{H}$ of the discriminant hypersurface $\mathcal{H}$ of $W$.  One of the most important contributions of Bessis' work in \cite{bessis-dual-braid} and \cite{Bessis_Annals} was a new presentation for $B(W)$. This \defn{dual braid presentation} is given as follows, in terms of the set $\op{Red}_W(c)$ of reduced reflection factorizations of a Coxeter element $c\in W$. For each reflection $t_i\in W$, $i=1,\ldots, N$ consider a formal generator $\bm t_{i}$. Then, after \cite[Thm.~2.2.5]{bessis-dual-braid} and \cite[Rem.~8.9]{Bessis_Annals} we have that
\begin{equation}
    B(W)=\langle \bm t_{1},\ldots, \bm t_{N}\, |\,  \underbrace{\bm{t}_{i_1}\cdots \bm{t}_{i_n}=\bm{t}_{i'_1}\cdots\bm{t}_{i'_n}=\ldots}_{h^n n!/\#W\text{-many words}}\ \rangle,\label{eq:dual_braid_pres}
\end{equation}
where each word $\bm{t}_{j_1}\cdots \bm{t}_{j_n}$ that appears in the relations corresponds to a reduced reflection factorization $t_{j_1}\cdots t_{j_n} = c$ in $\op{Red}_W(c)$.

Starting with \emph{any} element $g\in W$, one may consider groups presented analogously to \eqref{eq:dual_braid_pres} but with the relations coming instead from the set $\op{Red}_W(g)$ of reduced reflection factorizations of $g$. Such constructions were called \emph{generated groups} by Michel (see \cite[Thm.~0.5.2]{bessis-dual-braid}) and \emph{interval groups} by McCammond (see \cite[Def.~2.4]{jon_failure}). In the case that $g$ is a quasi-Coxeter element, Baumeister et al.\ made an extensive study of the corresponding interval groups and they asked \cite[\S4: Ques.~(f)]{neaime} for a geometric interpretation for them. The main ingredient in Bessis' proof was that the $LL$ map encoded both the set $\op{Red}_W(c)$ and the braid monodromy of $\mathcal{H}$ (see the exposition in \cite[\S6]{chapuy_theo_lapl}). These objects exist also for the Frobenius manifolds $\mathcal{F}_g$ and they remain related in the same way, see \S\ref{Sec: Frobenius mflds}. In particular, the discriminant hypersurface \defn{$\mathcal{D}_{\mathcal{F}_g}$} is defined as the subset of points in $\mathcal{F}_g$ where the Frobenius algebra on their tangent planes is  not semi-simple, see \cite[Def.~3.18]{hertling_book}.  It is therefore natural to ask: 
\begin{question}
Let $W$ be a real reflection group and $g$ a quasi-Coxeter element in $W$, and assume that the corresponding Frobenius manifold $\mathcal{F}_g$ exists. Is it true that the interval group associated to the set of reflection factorizations $\op{Red}_W(g)$ is isomorphic to the fundamental group $\pi_1(\CC^n-\mathcal{D}_{\mathcal{F}_g})$ of the complement of the discriminant of $\mathcal{F}_g$? 
\end{question}

\section*{Acknowledgements}

We thank David Jackson and Jiayuan Wang  for helpful comments and suggestions.
This work was facilitated by computer experiments using Sage \cite{sagemath}, its algebraic combinatorics features developed by the Sage-Combinat community \cite{Sage-Combinat}, and CHEVIE \cite{chevie}. 

The second author was partially supported by an ORAU Powe award and by a grant from the Simons Foundation (634530). The third author  was partially supported by NSF grants DMS-1855536 and DMS-22030407.  

An extended abstract of this work appeared as \cite{DLM-extended-abstract}.

\printbibliography

\end{document}